\documentclass[a4paper, 11pt,headings=small,abstract]{scrartcl}
\usepackage[english]{babel}

\usepackage{amscd,amsmath,amsthm,array,hhline,mathrsfs, enumerate, booktabs,fancybox,calc,textcomp,xcolor,graphicx,bbm,xspace,nicefrac,stmaryrd,url,arcs,listings,amssymb}
\usepackage[title]{appendix}
\usepackage[citestyle = numeric-comp,bibstyle = numeric, firstinits=true, natbib = true, backend = bibtex,maxbibnames=99, url=false, isbn=false]{biblatex}
\usepackage[semibold]{libertine}
\usepackage[upint,amsthm]{libertinust1math}
\usepackage[scr=boondoxupr, scrscaled = 1.0, cal =stixtwoplain]{mathalpha}
\usepackage[scaled=0.95]{inconsolata}
\usepackage{accents}
\usepackage{dsfont}
\DeclareMathAlphabet{\mathsf}{OT1}{\sfdefault}{m}{n}
\SetMathAlphabet{\mathsf}{bold}{OT1}{\sfdefault}{b}{n}
\DeclareMathAlphabet{\mathfrak}{U}{jkpmia}{m}{it}
\SetMathAlphabet{\mathfrak}{bold}{U}{jkpmia}{bx}{it}
\usepackage[utf8]{inputenc}
\usepackage{enumitem}
\usepackage{etoolbox}
\usepackage[normalem]{ulem}
\usepackage{mathtools}
\usepackage{geometry}
\usepackage{float}

\usepackage{bm}
\usepackage{tikz}
\usepackage[outdir =./]{epstopdf}
\numberwithin{equation}{section}

\usepackage{chngcntr}
\counterwithin{figure}{section}

\usepackage{xcolor}
\usepackage[colorlinks=true, allcolors = myteal]{hyperref}

\definecolor{WIMgreen}{RGB}{60 134 132}
\definecolor{UMblue}{RGB}{4 47 86}
\definecolor{myteal}{RGB}{0 123 137}
\definecolor{material_green}{RGB}{27 43 52}
\definecolor{dracula_pink}{RGB}{180 93 149}
\definecolor{dracula_blue}{RGB}{40 42 54}
\definecolor{dracula_turq}{RGB}{92 143 159}
\definecolor{dracula_orange}{RGB}{255 184 108}
\definecolor{material_petrol}{RGB}{2 119 189}
\definecolor{Purple}{RGB}{103 58 183}

\usepackage{verbatim}
\usepackage[title]{appendix}

\definecolor{refkey}{gray}{.5}
\definecolor{labelkey}{gray}{.5}
\setkomafont{sectioning}{\normalcolor\bfseries}
\setkomafont{descriptionlabel}{\normalcolor\bfseries}
\setkomafont{section}{\large}
\setkomafont{subsection}{\normalsize}
\setkomafont{subsubsection}{\normalsize}
\setkomafont{paragraph}{\normalsize}
\setkomafont{subparagraph}{\normalsize}
\setkomafont{author}{\large}
\setkomafont{title}{\large\selectfont\bfseries}
\setkomafont{date}{\normalsize}

\usepackage{subcaption}

\theoremstyle{plain}
\newtheorem{theorem}{Theorem}
\newtheorem{lemma}[theorem]{Lemma}

\newtheorem{assumption}[theorem]{Assumption}
\newtheorem{corollary}[theorem]{Corollary}
\newtheorem{proposition}[theorem]{Proposition}

\theoremstyle{definition}
\newtheorem{remark}[theorem]{Remark}

\newtheorem{innercustomgeneric}{\customgenericname}
\providecommand{\customgenericname}{}
\newcommand{\newcustomtheorem}[2]{%
  \newenvironment{#1}[1]
  {%
   \renewcommand\customgenericname{#2}%
   \renewcommand\theinnercustomgeneric{##1}%
   \innercustomgeneric
  }
  {\endinnercustomgeneric}
}

\newcustomtheorem{customass}{Assumption}

\numberwithin{equation}{section}
\numberwithin{theorem}{section}

\DeclareMathOperator*{\argmin}{arg\,min}
\DeclareMathOperator*{\argmax}{arg\,max}

\newcommand{\sc}[2]{\langle#1,#2\rangle}
\newcommand{\norm}[1]{\lVert#1\rVert}
\let\P\relax
\DeclareMathOperator{\P}{{\mathbb P}}
\DeclareMathOperator{\PP}{{\mathbb P}}
\DeclareMathOperator{\E}{{\mathbb E}}
\DeclareMathOperator{\R}{{\mathbb R}}
\DeclareMathOperator{\N}{{\mathbb N}}

\renewcommand{\theta}{\vartheta}    
\newcommand*\diff{\mathop{}\!\mathrm{d}}

\newcommand{\lebesgue}{\bm{\lambda}}
\newcommand{\one}{\bm{1}}
\renewcommand{\hat}{\widehat}
\renewcommand{\tilde}{\widetilde}
\newcommand{\Sq}{\operatorname{Sq}}

\numberwithin{equation}{section}
\numberwithin{theorem}{section}
\allowdisplaybreaks

\bibliography{references}

\title{\fontsize{16}{19} \selectfont Multivariate change estimation for a stochastic heat equation from local measurements}
\author{Anton Tiepner\thanks{Aarhus University, Department of Mathematics, Ny Munkegade 118, 8000 Aarhus C, Denmark. \newline Email: \href{mailto:tiepner@math.au.dk}{tiepner@math.au.dk}} \and Lukas Trottner\thanks{University of Birmingham, School of Mathematics, B15 2TT Birmingham, UK. \newline Email: \href{mailto:l.trottner@bham.ac.uk}{l.trottner@bham.ac.uk}}}
\date{}

\begin{document}
\maketitle
\begin{abstract}
We study a stochastic heat equation with piecewise constant diffusivity $\theta$ having a jump at a hypersurface $\Gamma$ that splits the underlying space $[0,1]^d$, $d\geq2,$ into two disjoint sets $\Lambda_-\cup\Lambda_+.$ Based on multiple spatially localized measurement observations on a regular $\delta$-grid of $[0,1]^d$, we propose a joint M-estimator for the diffusivity values and the set $\Lambda_+$ that is inspired by statistical image reconstruction methods. We study convergence of the domain estimator $\hat{\Lambda}_+$ in the vanishing resolution level regime $\delta \to 0$ and with respect to the expected symmetric difference pseudometric. As a first main finding we give a characterization of the convergence rate for $\hat{\Lambda}_+$ in terms of the complexity of $\Gamma$ measured by the number of intersecting hypercubes from the regular $\delta$-grid. Furthermore, for the special case of domains $\Lambda_+$ that are built from hypercubes from the $\delta$-grid, we demonstrate that perfect identification with overwhelming probability is possible with a slight modification of the estimation approach. Implications of our general results are discussed under two specific structural assumptions on $\Lambda_+$. For a $\beta$-Hölder smooth boundary fragment $\Gamma$, the set $\Lambda_+$ is estimated with rate $\delta^\beta$. If we assume $\Lambda_+$ to be convex, we obtain a $\delta$-rate. While our approach only aims at optimal domain estimation rates, we also demonstrate consistency of our diffusivity estimators, which is strengthened to a CLT at minimax optimal rate for sets $\Lambda_+$ anchored on the $\delta$-grid.
\end{abstract}





\section{Introduction}
Over the last decades interest in statistics for stochastic partial differential equations (SPDEs) has continuously increased for several reasons. Not only is it advantageous to model many natural space-time phenomena by SPDEs as they automatically account for model uncertainty by including random forcing terms that describe a more accurate picture of data dynamics, but also the general surge in data volume  combined with enlarged computational power of modern computers makes it more appealing to investigate statistical problems for SPDEs.

In this paper we study a multivariate change estimation model for a stochastic heat equation on $\Lambda=(0,1)^d,$ $d\geq2,$ given by 
\begin{equation}
    \label{eq: intro SPDE}
    \diff X(t)=\Delta_\theta X(t)\diff t+\diff W(t),\quad 0\leq t\leq T,
\end{equation}
with discontinuous diffusivity $\theta.$ The driving force is space-time white noise $\dot W(t)$ and the weighted Laplace operator $\Delta_\theta=\nabla\cdot\theta\nabla$ is characterized by a jump in the diffusivity
\begin{equation}
    \label{eq: intro theta}
    \theta(x)=\theta_-\mathbf{1}_{\Lambda_-}(x)+\theta_+\mathbf{1}_{\Lambda_+}(x),\quad x\in(0,1)^d,
\end{equation}
where the sets $\Lambda_{\pm}$ form a partition of $\overline{\Lambda} = [0,1]^d.$ Our primary interest lies in the construction of a nonparametric estimator of the change domain $\Lambda_+$, which is equivalently characterized by the hypersurface
\begin{equation}
    \label{eq: intro change}
    \Gamma\coloneqq\partial\Lambda_-\cap\partial\Lambda_+\subset[0,1]^d.
\end{equation}
The SPDE \eqref{eq: intro SPDE} can, for instance, be used to describe the heat flow through two distinct materials with different heat conductivity, colliding in $\Gamma$. An exemplary illustration of the solution $X$ to \eqref{eq: intro SPDE} with $\Gamma = \{(x,y): y = x, 0 \leq x,y \leq 1\}$ in two spatial dimensions is given in Figure \ref{fig:heat_change}.
\begin{figure}[t!]
\centering
\begin{subfigure}{.32\textwidth}
		\centering\includegraphics[width=1.\linewidth]{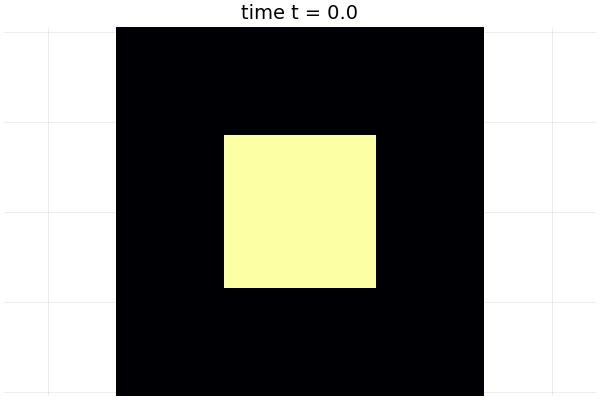}
	\end{subfigure}
\begin{subfigure}{.32\textwidth}
		\centering\includegraphics[width=1.\linewidth]{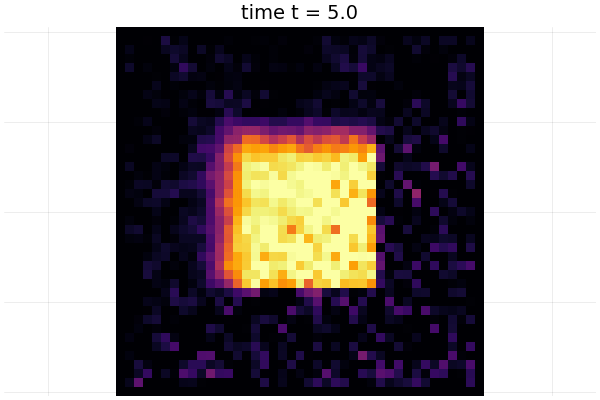}
	\end{subfigure}
 \begin{subfigure}{.32\textwidth}
		\centering\includegraphics[width=1.\linewidth]{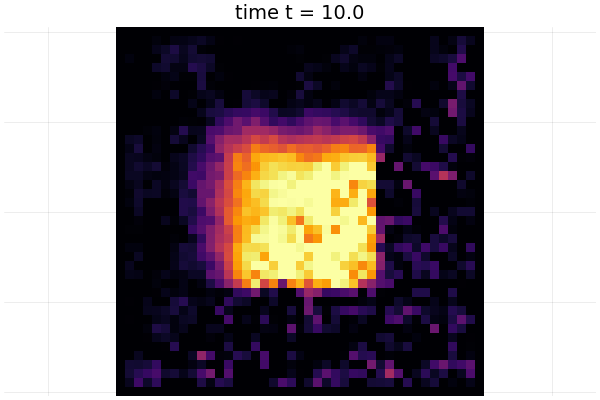}
	\end{subfigure}
 \begin{subfigure}{.32\textwidth}
		\centering\includegraphics[width=1.\linewidth]{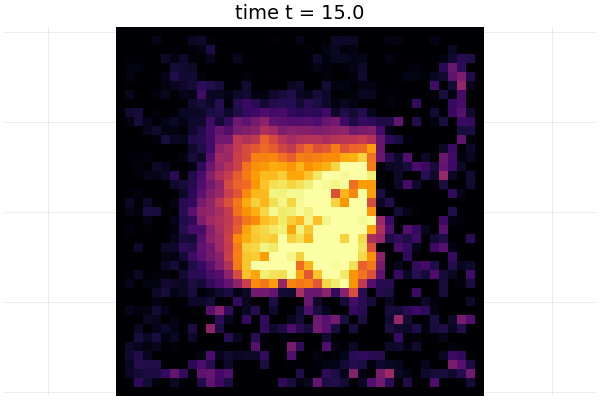}
	\end{subfigure}
 \begin{subfigure}{.32\textwidth}
		\centering\includegraphics[width=1.\linewidth]{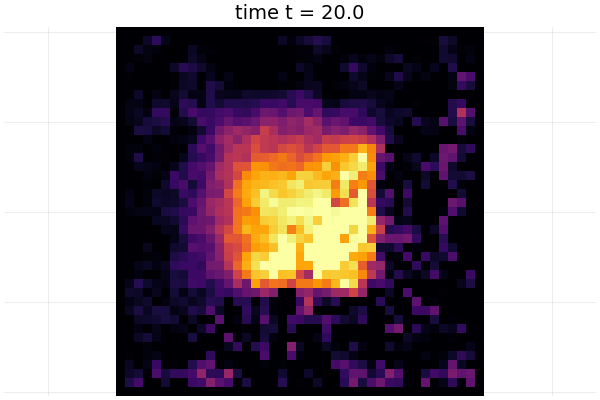}
	\end{subfigure}
 \begin{subfigure}{.32\textwidth}
		\centering\includegraphics[width=1.\linewidth]{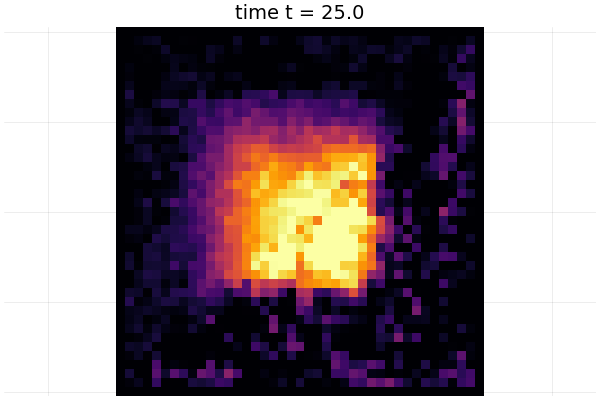}
	\end{subfigure}
\caption{Realization of the solution $X$ to $\eqref{eq: intro SPDE}$ in $d=2$ with change curve $\Gamma = \{(x,y): y = x, 0 \leq x,y \leq 1\}$ at increasing time points. $X_0=5\cdot\mathbf{1}_{(x,y)\in[0.3,0.7]^2}$, $\theta_-=0.3$ and $\theta_+=3.$}
\label{fig:heat_change}
\end{figure}

Structurally, the statistical problem of estimating $\Gamma$ is closely related to image reconstruction problems where one typically considers a regression model with (possibly random) design points $X_k$ and observational noise  $\varepsilon_k$ given by
\[Y_k= f(X_k)+\varepsilon_k,\quad 1\leq k\leq N,\]
where the $Y_k$ correspond to the observed color of a pixel centred around the spatial point $X_k\in[0,1]^d$, and the otherwise continuous function $f\colon [0,1]^d\rightarrow[0,1]$ has a discontinuity along the hypersurface $\Gamma,$ that is 
\[f(x) = f_-(x) \one_{\Lambda_-}(x) + f_+(x)\one_{\Lambda_+}(x), \quad x \in [0,1]^d.\]
Such problems are, for instance, studied in \cite{koro95a,koro95b,tsy93,tsy94,mueller_cube_94,mueller_indexed_96,mueller_maximin_94,rudemo_boundary_94,qiu_jump_2007,carlstein_boundary_94}. Assuming specific structures such as boundary fragments \cite{tsy94} or star-shapes \cite{rudemo_boundary_94,rudemo_1994_approx}, nonparametric regression methods are employed to consistently estimate both the image function $f$ as well as the edge $\Gamma$, which in the boundary fragment case is characterized as the epigraph of a function $\tau \colon [0,1]^{d-1} \to [0,1]$. The rates of convergence depend on the smoothness of $f$ and $\tau$, the dimension $d$ as well as the imposed distance function. Moreover, optimal convergence rates for higher-order Hölder smoothness $\beta > 1$ can in general not be achieved under equidistant, deterministic design, cf.\ \cite[Chapter 3-5]{tsy93}. 

Estimation of scalar parameters in SPDEs is well-studied in the literature. When observing spectral measurements $(\sc{X(t)}{e_k})_{0\leq t\leq T, k\leq N}$ for an eigenbasis $(e_k)_{k\in\N}$ of a parameterized differential operator $A_\theta$, \cite{huebner_asymptotic_1995} derive criteria for identifiability of $\theta$ depending on the order of $A_\theta$ and the dimension. This approach was subsequently adapted to joint parameter estimation \cite{lototsky_parameter_2003}, hyperbolic equations \cite{lototsky_multichannel_2010}, lower-order nonlinearities \cite{pasemann_drift_2020}, temporal discretization \cite{cialenco_drift_2019} or fractional noise \cite{cialenco_asymptotic_2009}. If only discrete points $X(x_k,t_i)$ on a space-time grid are available, then estimation procedures relying on power variation approaches and minimum-contrast estimators are analyzed, amongst others, in \cite{kaino_parametric_2021,tonaki2022parameter,hildebrandt_parameter_2019,chong_high-frequency_2020}. For a comprehensive overview of statistics for SPDEs we refer to the survey paper \cite{cialenco_statistical_2018} and the website \cite{altmeyer_web}.

Our estimation approach is based  on \textit{local measurements}, as first introduced in \cite{altmeyer_nonparametric_2020}, which are continuous in time and localized in space around $\delta$-separated grid center points $x_\alpha\in(0,1)^d$, $\alpha \in \{1,\ldots,\delta^{-1}\}^d$. More precisely, for a compactly supported and sufficiently smooth kernel function $K$ and a resolution level $\delta \in 1/\N$, we observe
\begin{align*}
    &(X_{\delta,\alpha}(t))_{0\leq t\leq T}=(\sc{X(t)}{K_{\delta,\alpha}})_{0\leq t\leq T},\\
    &(X_{\delta,\alpha}^\Delta(t))_{0\leq t\leq T}=(\sc{X(t)}{\Delta K_{\delta,\alpha}})_{0\leq t\leq T},\quad \alpha \in \{1,\ldots,\delta^{-1}\}^d,
\end{align*}
for the rescaled and recentered functions $K_{\delta,\alpha}(\cdot)=\delta^{-d/2}K(\delta^{-1}(\cdot-x_\alpha))$. The function $\delta^{-d/2}K(-\delta^{-1}\cdot)$ is also referred to as \textit{point-spread function}, which is motivated from applications in optical systems, and the local measurement $X_{\delta,\alpha}$ represents a blurred image---typically owing to physical measurement limitations---that is obtained from convoluting the solution with the point spread function at the measurement location $x_\alpha$. The asymptotic regime $\delta \to 0$ therefore allows for higher resolution images of the heat flow at the chosen measurement locations. 
Given such local measurements, we employ a modified likelihood approach leading to an M-estimator for the quantities $(\theta_-,\theta_+,\Lambda_+).$ 

Since their introduction in \cite{altmeyer_nonparametric_2020}, local measurements have been used in numerous statistical applications. In \cite{altmeyer_nonparametric_2020} it was shown that a continuously differentiable diffusivity $\theta$ can be identified at location $x_\alpha$ from the observation of a single local measurement $(X_{\delta,\alpha}(t))_{0\leq t\leq T}.$ Subsequently, their approach has been extended to semilinear equations \cite{altmeyer_parameterSemi_2020}, convection-diffusion equations \cite{altmeyer_anisotrop2021,strauch_velocity_2023}, multiplicative noise \cite{janak_2023_multiplicative} and wave equations \cite{ziebell_wave23}. In \cite{altmeyer_parameter_2020} the practical relevance of the method has been demonstrated in a biological application to cell repolarization. 

Closely related to this paper is the one-dimensional change point estimation problem for a stochastic heat equation studied in \cite{reiß_2023_change}, which should be understood as the one-dimensional analogue to our problem setting. Indeed, in $d=1$, the estimation of \eqref{eq: intro change} boils down to the estimation of a single spatial change point at the jump location of the diffusivity. In \cite{reiß_2023_change} two different jump height regimes are analyzed, where the absolute jump height is given by $\eta \coloneqq \lvert \theta_+ - \theta_-\rvert$. In the vanishing jump height regime $\eta\rightarrow0$ as $\delta\rightarrow0$, the authors demonstrate distributional convergence of the centralized change point estimator, where the asymptotic distribution is given by the law of the minimizer of a two-sided Brownian motion with drift, cf.\ \cite[Theorem 4.2]{reiß_2023_change}. In contrast, if $\eta$ is uniformly bounded away from $0$, it is shown in \cite[Theorem 3.12]{reiß_2023_change} that the change point can be identified with rate $\delta$ while the estimators for $(\theta_-,\theta_+)$ achieve the optimal rate $\delta^{3/2}$ in one dimension, cf.\ \cite{altmeyer_anisotrop2021} regarding optimality for parameter estimation, also in higher dimensions. 

 While \cite{reiß_2023_change} focus on the one-dimensional case, some core technical results on the order and concentration of the empirical local Fisher observations translate immediately to our multivariate setting, since the employed techniques are spectral in nature, cf.\ Proposition \ref{prop: genweaksol} and Propositions  \ref{lem: orderFisher} -- \ref{prop:coupling} in Section \ref{sec: estconst}.
This allows us in our multivariate setting \eqref{eq: intro SPDE} to focus on the substantial challenges arising from the fact that change estimation is no longer a parametric problem but becomes an inherently nonparametric and geometric one that requires additional techniques from statistics for SPDEs and image reconstruction. Concretely, we may either target $\Lambda_+$ directly or indirectly via estimation of the change interface \eqref{eq: intro change} (often referred to as edge estimation in image reconstruction). In this paper, we will first discuss the estimation problem for general sets $\Lambda_+$ and then specialize our estimation strategy and result to specific domain  shapes. 

Let us briefly describe our estimation approach in non-technical terms. For simplicity and to underline the correspondence to image reconstruction problems, let us consider for the moment only the case $d=2$. We may then interpret the regular $\delta$-grid as pixels, indexed by $\alpha \in \{1,\ldots,\delta^{-1}\}^2$. By the nature of local observations that give only aggregated information on the heat flow on each of these pixels, the best we can hope for is a good approximation of a pixelated version of the true ``foreground image'' $\Lambda_+^0$ that we wish to distinguish from the true ``background image'' $\Lambda_-^0$. The pixelated version $\Lambda^\updownarrow_+$ is defined as the union of pixels that have a non-zero area intersection with $\Lambda_+^0$ as illustrated in Figure \ref{fig:minkowski_intro}.
\begin{figure}[t!]
\centering
\begin{subfigure}{.49\textwidth}
		\centering\includegraphics[width=1.\linewidth]{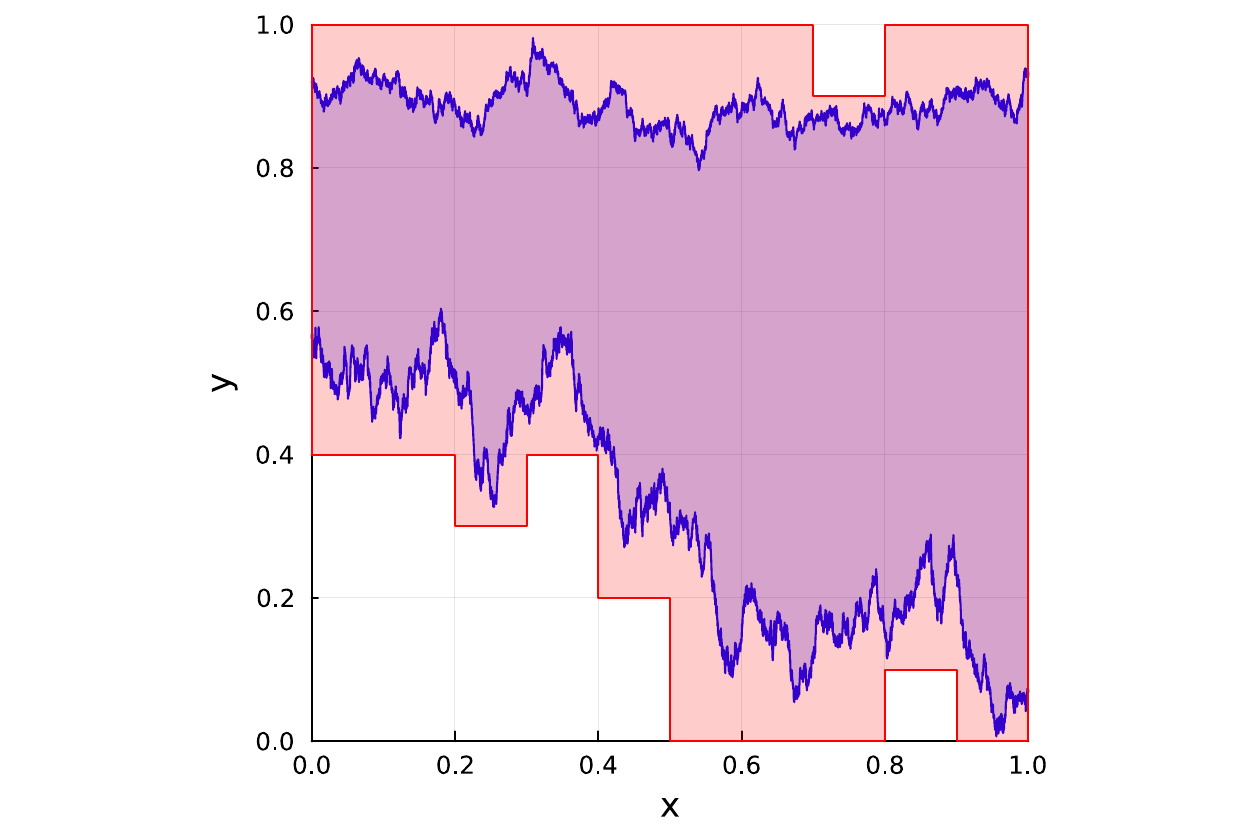}
	\end{subfigure}
\begin{subfigure}{.49\textwidth}
		\centering\includegraphics[width=1.\linewidth]{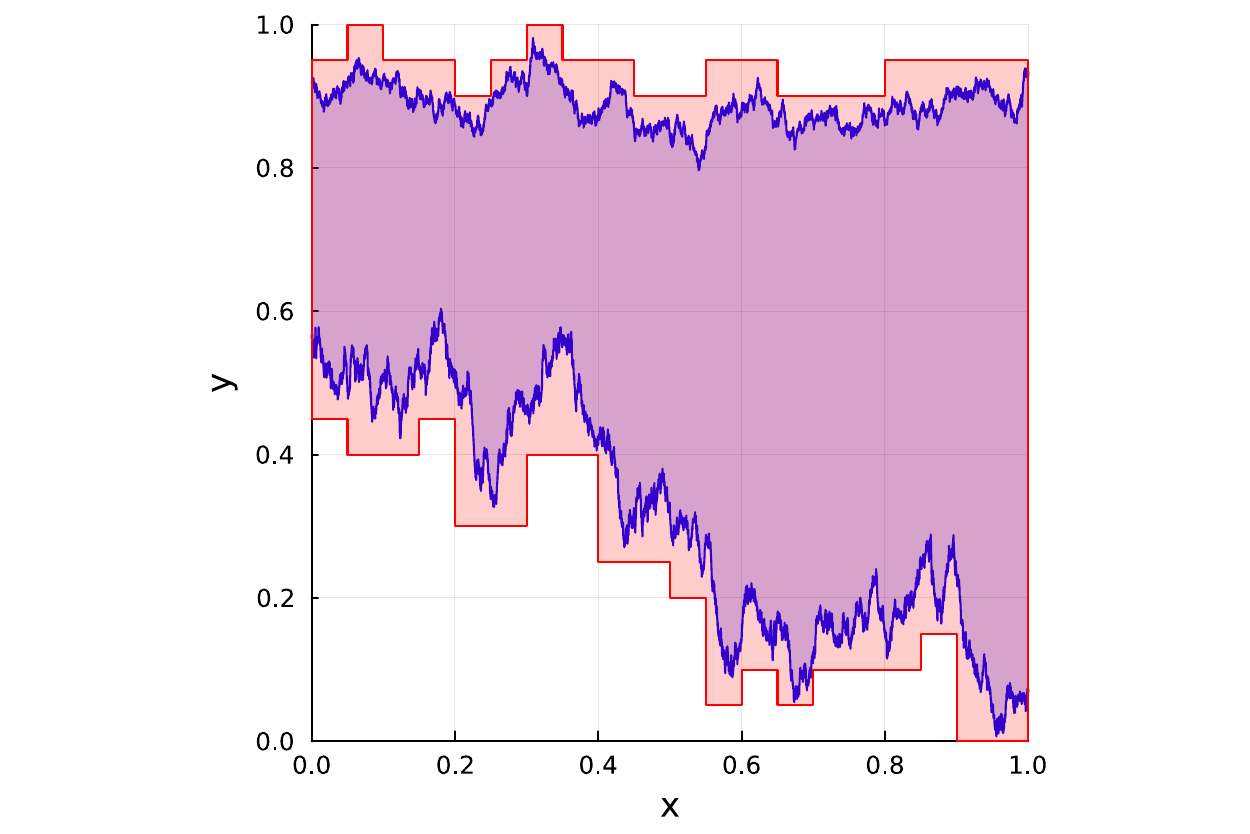}
	\end{subfigure}
\caption{$\Lambda_+^0$ (blue) is approximated by the pixelated version $\Lambda_+^\updownarrow$ (red); left: $\delta=0.1$; right: $\delta=0.05$.}
\label{fig:minkowski_intro}
\end{figure}
Based on a generalized Girsanov theorem for Itô processes, we can assign a modified local log-likelihood $\ell_{\delta,\alpha}(\theta_-,\theta_+,\Lambda_+)$ to each $\alpha$-pixel for all pixelated candidate sets $\Lambda_+ \in \mathcal{A}_+$ that assigns the diffusivity value $\theta_\pm$ to the $\alpha$-pixel if and only if $x_\alpha \in \Lambda_\pm$. An estimator $(\hat{\theta}_-,\hat{\theta}_+,\hat{\Lambda}_+)$ is then obtained as the maximizer of the aggregated contrast function 
\[(\theta_-,\theta_+,\Lambda_+) \mapsto \sum_\alpha \ell_{\delta,\alpha}(\theta_-,\theta_+,\Lambda_+).\]
Let us emphasize that we only require $\Lambda^\updownarrow_+ \in \mathcal{A}_+$ of the pixelated candidate sets $\mathcal{A}_+$, which, given specific information on the shape of the true domain $\Lambda^0_+$, allows for much more parsimonious choices than the canonical choice of all possible black and white $\delta^{-1} \times \delta^{-1}$-images. On a more technical note, to establish the convergence bound, we reformulate our estimator as an M-estimator based on an appropriate empirical process $\chi \mapsto Z_\delta(\chi)$, so that quite naturally, concentration analysis of $Z_\delta$ becomes key. 
 The basic idea of taking $\hat{\Lambda}_+$ as union of best explanatory pixels by optimizing over a given family of candidate sets originates  from classical statistical image reconstruction methods \cite{tsy94,tsy93,mueller_cube_94,mueller_indexed_96}.

The convergence rate of our estimator $\hat{\Lambda}_+$ is entirely characterized by the complexity of the separating hypersurface $\Gamma$ that induces a bias between the true domain $\Lambda^0_+$ and its pixelated version $\Lambda^0_+$. In particular, assuming that the set $\mathcal{B}$, describing the number of pixels that are sliced by $\Gamma$ into two parts of non-zero volume, is of size 
\begin{equation}\label{eq:boundary_complex}
\lvert \mathcal{B} \rvert\lesssim \delta^{-d+\beta}, \quad \beta \in (0,1],
\end{equation}
we show in Theorem \ref{theo:rate_cp} that 
\[
\E[\lebesgue(\hat{\Lambda}_+\vartriangle\Lambda_+^0)]\lesssim\delta^\beta,\]
with the symmetric set difference $\vartriangle$. This result immediately entails estimation rates for $\Lambda^0_+$ in terms of the  Minkowski dimension of its boundary. 
Furthermore, the estimation procedure results in the diffusivity parameter estimation rates $|\hat{\theta}_\pm-\theta_\pm|=O_{\PP}(\delta^{\beta/2})$, which yields the same estimation rate for the diffusivity or ``image'' estimator $\hat{\theta} \coloneqq \hat{\theta}_+ \one_{\hat{\Lambda}_+} + \hat{\theta}_-\one_{\hat{\Lambda}_-}$. 

In the limiting case where $\lvert \mathcal{B} \rvert = 0$, we show that a modification $(\hat{\theta}{}^\ast_-,\hat{\theta}{}^\ast_+, \hat{\Lambda}{}^\ast_+)$ of the estimators allows for strongly improved identification results. More precisely in Theorem \ref{theo:tiling} and Theorem \ref{theo:clt}, we demonstrate that perfect identification of the change domain $\Lambda^0_+$ with overwhelming probability is possible in the sense that 
\[\lim_{\delta \to 0} \PP\big(\hat{\Lambda}^\ast_+ = \Lambda^0_+ \big) = 1\]
and that a CLT at minimax optimal rate $\delta^{d/2+1}$ holds for the diffusivity estimators $\hat{\theta}^\ast_\pm$, provided that we can construct candidate sets $\mathcal{A}_+$ that only grow polynomially wrt the inverse resolution $\delta^{-1}$. The assumption $\lvert \mathcal{B} \rvert = 0$  corresponds to $\Lambda^0_+$ being anchored on the $\delta$-grid, which is the perspective taken in the works \cite{mueller_cube_94, mueller_indexed_96} on statistical image reconstruction and in our case liberates the convergence analysis of the change domain estimator from the dominating geometric bias induced by imperfect grid approximations of $\Lambda^0_+$ when $\lvert \mathcal{B} \rvert > 0$.

To make these general estimation strategies and results concrete, we apply it to two specific shape constraints on $\Lambda_+^0$. Assuming that $\Gamma$ is a boundary fragment that is described by a change interface with graph representation $\tau^0\colon [0,1]^{d-1}\rightarrow[0,1]$, that is, $\Lambda_+^0 = \{(x,y) \in [0,1]^d: y > \tau^0(x)\}$, we choose closed epigraphs of piecewise constant grid functions $\tau \colon [0,1]^{d-1} \to [0,1]$ as candidate sets $\mathcal{A}_+$. The boundary of the estimator $\hat{\Lambda}_+$ may then be interpreted as the epigraph of a random function $\hat{\tau}\colon [0,1]^{d-1} \to [0,1]$ that gives a nonparametric estimator of the true change interface $\tau^0$. Then, given 
$\beta$-Hölder smoothness assumptions on $\tau^0$, where $\beta\in(0,1],$ we can verify \eqref{eq:boundary_complex} and obtain 
\begin{equation*}
    \E[\norm{\hat{\tau}-\tau^0}_{L^1([0,1]^{d-1})}]\lesssim\delta^\beta\quad\text{or equivalently}\quad \E[\lebesgue(\hat{\Lambda}_+\vartriangle\Lambda_+^0)]\lesssim\delta^\beta.
\end{equation*}
When $\tau^0$ is already perfectly represented by a grid function, our stronger estimation results from Theorem \ref{theo:tiling} and Theorem \ref{theo:clt} apply. In a second model, we  assume that $\Lambda_+^0$ is a convex set with boundary $\Gamma$. Based on the observation that any ray that intersects the interior of a convex set does so in exactly two points, we construct a family of candidate sets $\mathcal{A}_+$ of size $\lvert \mathcal{A}_+ \rvert \asymp\delta^{-(d+1)}$ and show that
\[\E[\lebesgue(\hat{\Lambda}_+\vartriangle\Lambda^0_+)]\lesssim\delta.\]
For the particular case of a hypercuboid supported on the $\delta$-grid, these results again are strengthened according to Theorem \ref{theo:tiling} and \ref{theo:clt}.
Optimality of the obtained rates in both models is discussed in the related image reconstruction problem.  

\smallskip
\noindent \textbf{Outline}
The paper is structured as follows. In Section \ref{sec: setup} we formalize the model and discuss fundamental properties of the solution to \eqref{eq: intro SPDE}. The general estimation strategy and our main results are given in Section \ref{sec: estconst}. In Section \ref{sec: main}, those findings are applied to the two explicitly studied change domain structures outlined above. Lastly, we summarize our results and discuss potential extensions in future work in Section \ref{sec: summary}.  Remaining technical proofs are deferred to Section \ref{app:proofs}.

\smallskip
\noindent\textbf{Notation}
Throughout this paper, we work on a filtered probability space $(\Omega, \mathcal{F},(\mathcal{F}_t)_{0\leq t\leq T},\P)$ with fixed time horizon $T<\infty$. The resolution level $\delta$ is such that $n\coloneqq\delta^{-1}\in\N$ and $N\coloneqq n^d.$ For a set $A \subset \R^d$, the notation $A^\circ$ is exclusively reserved for its interior in $\R^d$ endowed with the standard Euclidean topology. For a general topological space $\mathcal{X}$ and a subset $A \subset \mathcal{X}$, we denote its interior by $\operatorname{int} A$, let $\overline{A}$ be its closure and $\partial A$ be its boundary in $\mathcal{X}$.  If not mentioned explicitly otherwise, we always understand the topological space in this paper to be $\mathcal{X} = [0,1]^d$ endowed with the standard subspace topology. 
For two numbers $a,b \in \R$, we write $a\lesssim b$ if $a\leq Cb$ holds for a constant $C$ that does not depend on $\delta$. 
For an open set $U\subset\R^d$, $L^2(U)$ is the usual $L^2$-space with inner product $\sc{\cdot}{\cdot}_{L^2(U)}$ and we set $\sc{\cdot}{\cdot}\coloneqq\sc{\cdot}{\cdot}_{L^2(\Lambda)}$. 
The Euclidean norm of a vector $a \in\R^p$ is denoted by $\lvert a \rvert$. If $f\colon \R^d \to \R^d$ is a vector valued function, we write $\lVert f \rVert_{L^2(\R^d)} \coloneqq \lVert \lvert f \rvert \rVert_{L^2(\R^d)}$. 
By $H^k(U)$ we denote the usual Sobolev spaces, and let $H_0^1(U)$ be the completion of $C_c^{\infty}(U)$, the space of smooth compactly supported functions, relative to the $H^1(U)$ norm. 
The gradient, divergence and Laplace operator are denoted by $\nabla$, $\nabla\cdot$ and $\Delta$, respectively.
\section{Setup}
We start by formally introducing the SPDE model, discussing existence of solutions and introducing the local measurement observation scheme that we will be working with in our statistical analysis.
\label{sec: setup}
\subsection{The SPDE model}
\label{sec: SPDE model}
In the following we consider a stochastic partial differential equation on $\Lambda \coloneqq (0,1)^d$, where $d \geq 2$, with Dirichlet boundary condition, which is specified by
\begin{equation}
	\label{eq: SPDE}
\begin{cases}
	   \diff X(t)=\Delta_\theta X(t)\diff t+\diff W(t), &0\leq t\leq T,\\
      X(0) \equiv 0,\\
      X(t)|_{\partial_\Lambda}=0, &0\leq t\leq T,
    \end{cases}
\end{equation}
for driving space-time white noise $(\dot{W}(t))_{t \in [0,T]}$ on $L^2(\Lambda)$. The operator $\Delta_\theta$ with domain $\mathcal{D}(\Delta_\theta)$ is given formally by
\begin{equation}
    \label{eq: operator}
    \Delta_\theta u=\nabla\cdot \theta\nabla u = \sum_{i=1}^d \partial_i (\theta \partial_i u),\quad u\in\mathcal{D}(\Delta_\theta),
\end{equation}
where $\theta$ is piecewise constant in space, given by
\[\theta(x)=\theta_-\mathbf{1}_{\Lambda_-}(x)+\theta_+\mathbf{1}_{\Lambda_+}(x), \quad x \in (0,1)^d,\]
for two measurable  and disjoint sets $\Lambda_{\pm}$ s.t.\ $\Lambda_- \cup \Lambda_+ = [0,1]^d$ and $\theta_-,\theta_+ \in [\underline{\theta},\overline{\theta}] \subset (0,\infty)$.  Equivalently, we may rewrite the diffusivity in terms of the jump height $\eta \coloneqq \theta_+ - \theta_-$ as 
\[\theta(x) = \theta_- + \eta \one_{\Lambda_+}(x), \quad x \in (0,1)^d.\]
Under these assumptions, $-\Delta_\theta$ is a uniformly elliptic divergence-form operator. We assume that $\Lambda_-$ and $\Lambda_+$ are separated by a hypersurface $\Gamma$ parameterizing the set of points in the intersection of the boundaries
\begin{equation}
    \label{eq: changepoint}
    \Gamma\coloneqq\partial\Lambda_+ =\partial\Lambda_- \subset [0,1]^d,
\end{equation}
where $\partial \Lambda_\pm$ denotes the boundary of $\Lambda_\pm$ as a subset of the topological space $[0,1]^d$. 
We are mainly interested in estimating the domain $\Lambda_+$ which is intrinsically related to $\Gamma.$ We first propose a general estimator based on local measurements of the solution on a uniform grid of hypercubes, whose convergence properties are determined by the complexity of the boundary $\Gamma$ measured in terms of the number of hypercubes that are required to cover it.

The more structural information we are given on the set $\Lambda_+$, the better we can fine-tune the family of candidate sets underlying the estimator in order to increase the feasibility of implementation. Specifically, we will consider two different models for $\Lambda_\pm$.

\smallskip
\noindent\textbf{Model A: Graph representation}\ \\
$\Gamma$ forms a boundary fragment that has a graph representation, denoted by a change interface $\tau\colon [0,1]^{d-1}\rightarrow[0,1],$ i.e.,  
\begin{equation}\label{eq:change_graph}
\Gamma= \big\{(x,\tau(x)) : x \in [0,1]^{d-1} \big\}
\end{equation}
and the set $\Lambda_+$ takes the form
\begin{equation*}
    \Lambda_+=\{(x,y)\in[0,1]^d:y > \tau(x)\}.
\end{equation*}
Accordingly, the estimation problem of identifying $\Lambda_{\pm}$ can equivalently be broken down to the nonparametric estimation of the function $\tau$. Specifically, for an estimator $\hat{\tau}$ of $\tau$ we let $\hat{\Lambda}_+ \coloneqq \overline{\{(x,y) \in [0,1]^{d}: y > \hat{\tau}(x)\}}$ be the closure of the epigraph of $\hat{\tau}$ and  $\hat{\Lambda}_- \coloneqq [0,1]^d \setminus \hat{\Lambda}_+$ be its complement. This gives
\begin{equation}\label{est:domain}
\E\big[\bm{\lambda}\big(\hat{\Lambda}_\pm \vartriangle \Lambda_\pm\big)\big] = \E\big[\lVert \hat{\tau} - \tau \rVert_{L^1([0,1]^{d-1})} \big],
\end{equation}
such that evaluating the quality of the domain estimators $\hat{\Lambda}_{\pm}$ measured in terms of the expected Lebesgue measure of the symmetric differences $\hat{\Lambda}_{\pm} \vartriangle \Lambda_{\pm}$, is equivalent to studying the $L^1$-risk of the nonparametric estimator $\hat{\tau}$ of the change interface. 

\smallskip
\noindent\textbf{Model B: Convex set}\ \\
$\Lambda_+$ is convex. By convexity, for any $x \in\operatorname{int}\Lambda_+^0$, the vertical ray $y \mapsto x + ye_d$ intersects  $\partial \Lambda_+$ in exactly two points and those intersection points can be modeled by a lower convex function $f_1$ and an upper concave function $f_2$. Estimation of $\Lambda_+$ is then heuristically speaking equivalent to the estimation of the upper and lower function, taking the closure of $\{(x,y)\in[0,1]^{d}:\hat{f}_1(x)\leq y\leq \hat{f}_2(x)\}$ as an estimator for $\Lambda_+.$

\subsection{Characterization of the solution}
We shall first discuss properties of the operator $\Delta_\theta$ based on general theory of elliptic divergence form operators with measurable coefficients from \cite{davies90}. Let the closed quadratic form $\mathcal{E}_\theta$ with domain $\mathcal{D}(\mathcal{E}_\theta)$ be given by 
\begin{equation*}
    \begin{cases}
        \mathcal{D}(\mathcal{E_\theta})=H^1_0(\Lambda),\\
        \mathcal{E}_\theta(u,v)=\int_\Lambda \theta \nabla u \cdot \nabla v.
    \end{cases}
\end{equation*}
By \cite[Theorem 1.2.1]{davies90}, $\mathcal{E}_\theta$ is the form of a positive self-adjoint operator $-\Delta_\theta$ on $L^2(\Lambda)$ in the sense that $\mathcal{E}_\theta(u,u)=\lVert (-\Delta_\theta)^{1/2}u \rVert^2$ for $\mathcal{D}(\mathcal{E}_\theta)=\mathcal{D}((-\Delta_\theta)^{1/2})$ and according to \cite[Theorem 1.2.7]{davies90}, we have $u \in \mathcal{D}(-\Delta_\theta) \subset H^1_0(\Lambda)$ if there exists $g \in L^2(\Lambda)$ such that for any $v \in C_c^\infty(\Lambda)$,
\[\mathcal{E}_\theta(u,v) = \int_\Lambda gv,\]
in which case $-\Delta_\theta u = g$. Thus, \eqref{eq: operator} can be interpreted in a distributional sense and we have the relation
\[\mathcal{E}_\theta(u,v)=-\sc{\Delta_\theta u}{v},\quad (u,v)\in \mathcal{D}(\Delta_\theta)\times\mathcal{D}(\mathcal{E}_\theta).\]
Moreover, \cite[Theorem 1.3.5]{davies90} shows that $\mathcal{E}_\theta$ is a Dirichlet form, whence $\Delta_\theta$ generates a strongly continuous, symmetric semigroup $(S_\theta(t))_{t \geq 0} \coloneqq (\exp(\Delta_\theta t))_{t \geq 0}$. The spectrum of $-\Delta_\theta$ is discrete and the minimal eigenvalue, denoted by $\underline{\lambda}$, is strictly positive, cf.\ \cite[Theorem 6.3.1]{davies95}. Thus, $(-\Delta_\theta)^{-1}$ exists as a bounded linear operator with domain $L^2(\Lambda)$ and we may fix an orthonormal basis $\{e_k, k \in \N\}$ consisting of eigenvectors corresponding to the eigenvalues $\{\lambda_k, k \in \N\} = \sigma(-\Delta_\theta)$ that we denote in increasing order. Using the heat kernel bounds for the transition density of $S_\theta(t)$ given in \cite[Corollary 3.2.8]{davies90}, it follows that for any $t > 0$, $S_\theta(t)$  is a Hilbert--Schmidt operator,  but no weak or mild solution to \eqref{eq: SPDE} in the sense of \cite[Theorem 5.4]{da_prato_stochastic_2014} exists in $L^2((0,1)^d)$ since $d\geq 2$ implies that $\int_0^T\norm{S_\theta(t)}^2_{\mathrm{HS}}\diff t=\infty$.

However, following the discussion in \cite[Section 2.1]{altmeyer_nonparametric_2020} and \cite[Section 6.4]{altmeyer_anisotrop2021}, taking into account that by \cite[Theorem 6.3.1]{davies95} we have $\lambda_k \asymp k^{2/d}$ for any $k \in \N$, the stochastic convolution
\[X(t) \coloneqq \int_0^{t} S_\theta(t-s) \diff{W(s)}, \quad t \in [0,T],\]
is well-defined as a stochastic process on the embedding space $\mathcal{H}_1 \supset L^2((0,1)^d)$, where $\mathcal{H}_1$ can be chosen as a Sobolev space of negative order $-s < -d/2 +1$ that is induced by the eigenbasis $(\lambda_k,e_k)_{k \in \N}$. Extending the dual pairings $\langle X(t),z \rangle_{\mathcal{H}_1 \times \mathcal{H}^\prime_1}$ then allows us to obtain a Gaussian process $(\langle X(t),z \rangle)_{t \in [0,T], z \in L^2((0,1)^d)}$ given by 
\[\sc{X(t)}{z}\coloneqq\int_0^t\sc{S_{\theta}(t-s)z}{\diff W(s)}, \quad t \in [0,T],\quad z \in L^2((0,1)^d),\]
that solves the SPDE in the sense that for any  $z\in \mathcal{D}(\Delta_\theta)$,
\begin{equation}
    \label{eq: Ito}
    \sc{X(t)}{z}= \int_0^t \sc{X(s)}{\Delta_\theta z}\diff s+ \sc{W(t)}{z}, \quad t \in [0,T].
\end{equation}
\subsection{Local measurements}
We decompose $[0,1]^{d}$ into $n^{d}$ closed $d$-dimensional hypercubes $(\Sq(\alpha))_{\alpha \in [n]^{d-1}}$, where $\Sq(\alpha)$ has edge length $\delta$ and is centered at $x_\alpha \coloneqq \delta(\alpha - \tfrac{1}{2}\one)$ for any $\alpha \in [n]^{d}$. By $\Sq(\alpha)^\circ$ we denote the interior of $\Sq(\alpha)$ in $\R^d$. Let also $\mathfrak{P} = \mathfrak{P}(\delta) \coloneqq 2^{\{\Sq(\alpha): \alpha \in [n]^d\}}$ be the power set of $\{\Sq(\alpha): \alpha \in [n]^d\}$ and let $\mathcal{P} = \mathcal{P}(\delta) \coloneqq \{\bigcup_{C \in \mathfrak{C}} C: \mathfrak{C} \in \mathfrak{P}\}$ be the family of sets that can be built from taking unions of hypercubes in $\{\Sq(\alpha): \alpha \in [n]^d\}$. We refer to the hypercubes $\Sq(\alpha)$ as tiles and for any set $A \subset [0,1]^d$ we call a set $C \in \mathcal{P}$ such that $A \subset C$ a tiling of $A$.

Our estimation procedure is based on continuous-time observations of 
\[X_{\delta,\alpha}(t) = \langle X(t), K_{\delta,\alpha}\rangle,\quad X^\Delta_{\delta,\alpha}(t) = \langle X(t), \Delta K_{\delta,\alpha} \rangle,\quad \alpha\in[n]^d,\quad 0\leq t\leq T,\]
with $K_{\delta,\alpha}(y) = \delta^{-d/2}K(\delta(y - x_{\alpha}))$,  and $K\colon \R^d \to \R$ a kernel function s.t.\ $\operatorname{supp}K\subset [-1/2,1/2]^d$ and $K \in H^2(\R^d)$. To ease the notation, we assume that $K$ is normed in $L^2(\R^d)$, i.e. $\norm{K}_{L^2(\R^d)}=1.$ The measurement points $x_{\alpha}$, $\alpha\in[n]^d,$ are separated by an Euclidean distance of order $\delta$ such that the supports of the $K_{\delta,\alpha}$ are non-overlapping. In other words, we have  $N= \delta^{-d}$ measurement locations. Note that $K_{\delta,\alpha}\in\mathcal{D}(\Delta_\theta)$ whenever $\operatorname{supp}K_{\delta,\alpha}\cap\partial\Lambda_+=\varnothing,$ since then $\theta$ is constant on the support of $K_{\delta,\alpha}$. Thus, for $g = -\theta(x_{\alpha})\Delta K_{\delta,\alpha}$, integration by parts reveals $\int_\Lambda gv = \int_\Lambda \theta \nabla K_{\delta,\alpha} \cdot \nabla v$  for any $v \in C_c^\infty(\Lambda)$, i.e., $g = -\Delta_\theta K_{\delta,\alpha}$.

\section{Estimation strategy and main results}
\label{sec: estconst}
From here on, we denote the truth, i.e., the true values of the diffusivity and the true set partition of $\Lambda$, by an additional superscript $0$ for statistical purposes. To draw analogies to related work in image reconstruction \cite{mueller_cube_94,mueller_indexed_96} we allow $\Lambda^0_+$ to vary with $\delta$, but suppress the dependence notationally. To avoid some technicalities, we impose from now on the following assumption on $\Lambda^0_+$.
\begin{assumption} 
$\Lambda^0_+$ is open in $[0,1]^d$.
\end{assumption}
For some set $C \subset [0,1]^d$ define 
\[C^+ \coloneqq \bigcup_{\alpha \in [n]^d: \Sq(\alpha)^\circ \cap C \neq \varnothing} \Sq(\alpha), \]
which, if $C$ is open in $[0,1]^d$, is the minimal tiling of $C$, so that in particular $C \subset C^+\in\mathcal{P}$. Let us also set
\[\Lambda^\updownarrow_+ \coloneqq (\Lambda^0_+)^+,\]
which is the minimal tiling of $\Lambda^0_+$ by our assumption that $\Lambda^0_+$ is open in $[0,1]^d$. 
Let $\mathcal{A}_+ = \mathcal{A}_+(\delta) \subset \mathcal{P}$ be a family of candidate sets for a tiling of $\Lambda^0_+$ such that $\Lambda^\updownarrow_+ \in \mathcal{A_+}$. Note that $\mathcal{A}_+ = \mathcal{P}$ is always a valid choice. However, as we shall see later, much more parsimonious choices are possible if we can assume some structure on the set $\Lambda_+^0.$
We now introduce the modified local log-likelihood 
\begin{equation}
\label{eq: loglikelihood}
    \ell_{\delta,\alpha}(\theta_-,\theta_+,\Lambda_+) = \theta_{\delta,\alpha}(\Lambda_+) \int_0^T X^\Delta_{\delta,\alpha}(t) \diff{X_{\delta,\alpha}(t)} - \frac{\theta_{\delta,\alpha}(\Lambda_+)^2}{2} \int_0^T X^\Delta_{\delta,\alpha}(t)^2 \diff{t},
\end{equation}
for the decision rule
\begin{equation}
    \label{eq: theta}
    \theta_{\delta,\alpha}(\Lambda_+) \coloneqq \begin{cases} \theta_+, &\operatorname{Sq}(\alpha)\subset\Lambda_+,\\ \theta_-, &\text{else}, \end{cases} = \begin{cases} \theta_+, &x_\alpha \in \Lambda_+,\\ \theta_-, &\text{else}, \end{cases} 
\end{equation}
where $\theta_\pm \in \Theta_\pm$ and  the candidate sets $\Lambda_+ \in \mathcal{A}_+$ are anchored on the grid $\mathcal{P}$ spanned by the hypercubes $\Sq(\alpha)$. 
The interpretation of the stochastic integral in \eqref{eq: loglikelihood} is provided by the following result that characterizes the tested processes $X_{\delta,\alpha}$ as semimartingales whose dynamics are determined by the location of the hypercube $\Sq(\alpha)$ relative to $\Lambda^0_\pm$.
\begin{proposition}[Modification of Proposition 2.1 in \cite{reiß_2023_change}]
\label{prop: genweaksol}
For any $\alpha\in[n]^{d}$ and $t\in[0,T]$ we have 
\[X_{\delta,\alpha}(t)=\begin{cases}
    \theta_-^0\int_0^tX^\Delta_{\delta,\alpha}(s)\diff s+B_{\delta,\alpha}(t), &\operatorname{Sq}(\alpha)\subset\overline{\Lambda_-^0},\\
    \theta_+^0\int_0^tX^\Delta_{\delta,\alpha}(s)\diff s+B_{\delta,\alpha}(t), &\operatorname{Sq}(\alpha)\subset\overline{\Lambda_+^0},\\
    \int_0^t\int_0^s\sc{\Delta_{\theta^0} S_{\theta^0}(s-u)K_{\delta,\alpha}}{\diff W(u)}\diff s+B_{\delta,\alpha}(t), &\text{else},
\end{cases}\]
where $(B_{\delta,\alpha})_{\alpha\in[n]^{d}}$ is an $n^d$-dimensional vector of independent scalar Brownian motions.
\begin{proof}
The first two lines follow from \eqref{eq: Ito} using $K_{\delta,\alpha}\in\mathcal{D}(\Delta_{\theta^0})$ if $\operatorname{supp}K_{\delta,\alpha}\cap\partial\Lambda_-\cap\partial\Lambda_+=\varnothing,$ in which cases $\Delta_{\theta^0} K_{\delta,\alpha}=\theta_-^0\Delta K_{\delta,\alpha}$ and $\Delta_\theta K_{\delta,\alpha}=\theta_+^0\Delta K_{\delta,\alpha}$, respectively. The expression on the change  areas, where generally $K_{\delta,\alpha} \notin \mathcal{D}(\Delta_{\theta^0})$, is proven in complete analogy to \cite[Lemma A.1]{reiß_2023_change} using spectral calculus and the stochastic Fubini theorem. 
\end{proof}
\end{proposition}

\subsection{Estimator and convergence rates depending on boundary complexity}\label{sec:boundary complexity}

In this section we make the following identifiability assumption on the true diffusivity parameters $\theta^0_\pm$ to define a general estimator.
\begin{assumption} 
We have access to two compact sets $\Theta_{-},\Theta_+ \subset [\underline{\theta}, \overline{\theta}]$ such that
\begin{enumerate}
    \item[(i)] $\theta^0_- \in \Theta_{-}$ and $\theta^0_+\in\Theta_+,$ and
    \item[(ii)]$\Theta_-$ and $\Theta_+$ are separated by $\underline{\eta} > 0$, i.e., for any $\theta_-\in\Theta_-$ and $\theta_+\in\Theta_+$, it holds $\lvert \theta_+ - \theta_- \rvert \geq \underline{\eta}.$
\end{enumerate}
\end{assumption}
In particular, this assumption implies that we have access to a lower bound $\underline{\eta} > 0$ on the absolute diffusivity jump height $\lvert \eta^0 \rvert = \lvert \theta^0_+ - \theta^0_-\rvert$. 
Let us note that without either feeding an estimator some information on the diffusivity ranges or specific geometric constraints on $\Lambda^0_\pm$ represented by the candidate set $\mathcal{A}_+$, we cannot obtain a precise allocation of an estimator pair $(\hat{\theta}_+,\hat{\Lambda}_+)$ to one of $(\theta^0_\pm,\Lambda^0_\pm)$, but could only aim for  optimally identification of the domain split. We will come back to this aspect in the next section.

We now use the local log-likelihoods in \eqref{eq: loglikelihood} to specify specify an estimator $(\hat{\theta}_-, \hat{\theta}_+, \hat{\Lambda}_+)$ via
\begin{equation}
\label{eq: sum likelihood}
    (\hat{\theta}_-,\hat{\theta}_+, \hat{\Lambda}_+) \in \argmax_{(\theta_-,\theta_+,\Lambda_+)\in \Theta_- \times \Theta_+ \times \mathcal{A}_+} \sum_{\alpha\in [n]^{d}} \ell_{\delta,\alpha}(\theta_-,\theta_+,\Lambda_+).
\end{equation}
Here, the set of maximizers is well defined and we can make a measurable choice for a maximizer since $\Theta_\pm$ are compact and $\mathcal{A}_+$ is a finite set. Setting 
\[ \hat{\Lambda}_- \coloneqq [0,1]^d \setminus \hat{\Lambda}_+,\]
as the corresponding estimator of $\Lambda^0_-$ we obtain the nonparametric diffusivity estimator 
\[\hat{\theta}(x) \coloneqq \hat{\theta}_- \one_{\hat{\Lambda}_-}(x) + \hat{\theta}_+ \one_{\hat{\Lambda}_+}(x), \quad x \in (0,1)^{d-1}.\]
Introduce further 
\begin{equation}
\label{eq: theta0}
    \theta^0_{\delta,\alpha} \coloneqq \begin{cases} \theta^0_+, &\operatorname{Sq}(\alpha)\subset\Lambda_+^\updownarrow,\\ \theta^0_-, &\text{else},\end{cases} = \begin{cases} \theta^0_+, &x_\alpha \in \Lambda^\updownarrow_+\\ \theta^0_-, &\text{else},\end{cases}
\end{equation}
and define by 
\begin{equation*}
\begin{split}
    \mathcal{B} 
    &= \{\alpha\in[n]^d:\operatorname{Sq}(\alpha)^\circ \cap \partial\Lambda_+^0\neq\varnothing\},
\end{split}
\end{equation*}
the indices $\alpha$ of hypercubes whose interiors intersect the boundary of $\Lambda^0_+$. It is important to observe that the boundary tiles $\mathcal{B}$ may equivalently be expressed as follows.
\begin{lemma}
It holds that 
\begin{equation*}\label{lem:tiling_boundary}
\mathcal{B} = \{\alpha \in [n]^d: \Sq(\alpha)^\circ \cap (\Lambda^\updownarrow_+ \vartriangle \Lambda^0_+) \neq \varnothing\}.
\end{equation*}
\end{lemma}
\begin{proof} 
Since $\Lambda^0_+$ is open in $[0,1]^d$, it holds that $\Lambda^0_+ \subset \Lambda^\updownarrow_+$ and therefore $\Lambda^\updownarrow_+ \vartriangle \Lambda^0_+ = \Lambda^\updownarrow_+ \setminus \Lambda^0_+$. Since $\Lambda^\updownarrow_+$ is closed and $\partial \Lambda^0_+ \cap \Lambda^0_+ = \varnothing$, it follows that $\partial \Lambda^0_+ \subset \Lambda^\updownarrow_+ \vartriangle \Lambda^0_+$, which gives the inclusion 
\[\mathcal{B} \subset \{\alpha \in [n]^d: \Sq(\alpha)^\circ \cap (\Lambda^\updownarrow_+ \vartriangle \Lambda^0_+) \neq \varnothing\}.\]
Conversely, if $\varnothing \neq \Sq(\alpha)^\circ \cap (\Lambda^\updownarrow_+ \vartriangle \Lambda^0_+) = \Sq(\alpha)^\circ \cap (\Lambda^\updownarrow_+ \setminus\Lambda^0_+)$, it follows that $\Sq(\alpha)^\circ \not\subset \Lambda^0_+$, but $\Sq(\alpha) \subset \Lambda^\updownarrow_+$ since the latter belongs to $\mathcal{P}.$ Thus, by definition of $\Lambda^\updownarrow_+$, we have $\Sq(\alpha)^\circ \cap \Lambda^0_+ \neq \varnothing$, which because $\Sq(\alpha)^\circ$ is connected and $\Sq(\alpha)^\circ \not\subset \Lambda^0_+$ implies that $\Sq(\alpha)^\circ \cap \partial \Lambda^0_+ \neq \varnothing$. This now also yields 
\[\mathcal{B} \supset \{\alpha \in [n]^d: \Sq(\alpha)^\circ \cap (\Lambda^\updownarrow_+ \vartriangle \Lambda^0_+) \neq \varnothing\}.\]
\end{proof}
Plugging the given representation of Proposition \ref{prop: genweaksol} into \eqref{eq: loglikelihood}, we may now express $\ell_{\delta,\alpha}$ in the following way, 
\begin{align*}
    \ell_{\delta,\alpha}(\theta_-,\theta_+,\Lambda_+) &= \big(\theta_{\delta,\alpha}(\Lambda_+)\theta^0_{\delta,\alpha} - \theta_{\delta,\alpha}(\Lambda_+)^2/2 \big)I_{\delta,\alpha} + \theta_{\delta,\alpha}(\Lambda_+) M_{\delta,\alpha} \\
    &\quad+ \one_{\mathcal{B}}(\alpha) \theta_{\delta,\alpha}(\Lambda_+) R_{\delta,\alpha},
\end{align*}
where we denote
\[M_{\delta,\alpha} \coloneqq \int_0^T X^\Delta_{\delta,\alpha}(t) \diff{B_{\delta,\alpha}(t)}, \quad I_{\delta,\alpha} \coloneqq \int_0^T X^\Delta_{\delta,\alpha}(t)^2 \diff{t},\]
and 
\[R_{\delta,\alpha} \coloneqq \int_0^T X^\Delta_{\delta,\alpha}(t) \Big(\int_0^t \langle \Delta_{\theta^0} S_{\theta^0}(t-s) K_{\delta,\alpha} - \theta^0_+ S_{\theta^0}(t-s) \Delta K_{\delta,\alpha}, \diff{W}(s) \rangle\Big) \diff{t}.  \]
The estimator \eqref{eq: sum likelihood} therefore can be represented as  
\begin{align} 
&(\hat{\theta}_-,\hat{\theta}_+, \hat{\Lambda}_+)\nonumber\\
&\quad\in \argmax_{(\theta_-,\theta_+,\Lambda_+) \in \Theta_- \times \Theta_+ \times \mathcal{A}_+} \Big\{\sum_{\alpha \in [n]^{d}} \Big((\theta_{\delta,\alpha}(\Lambda_+) - \theta^0_{\delta,\alpha}) M_{\delta,\alpha} - \frac{1}{2}(\theta_{\delta,\alpha}(\Lambda_+) - \theta^0_{\delta,\alpha})^2 I_{\delta,\alpha} \Big) \nonumber\\
&\hspace{110pt} + \sum_{\alpha\in \mathcal{B}} \theta_{\delta,\alpha}(\Lambda_+) R_{\delta,\alpha} + \sum_{\alpha \in [n]^d} \Big(\theta^0_{\delta,\alpha}M_{\delta,\alpha} + \frac{(\theta^0_{\delta,\alpha})^2}{2} I_{\delta,\alpha} \Big)\Big\}\nonumber\\ 
&= \argmin_{(\theta_-,\theta_+,\Lambda_+) \in \Theta_- \times \Theta_+\times \mathcal{A}_+} \Big\{Z_{\delta}(\theta_-,\theta_+,\Lambda_+) - \sum_{\alpha\in\mathcal{B}} \theta_{\delta,\alpha}(\Lambda_+) R_{\delta,\alpha} \Big\},\label{eq: est decomp}
\end{align}
where the empirical process $Z_\delta(\cdot)$ is given by
\begin{equation}\label{eq: Z}
\begin{split}
    Z_{\delta}(\theta_-,\theta_+,\Lambda_+) &\coloneqq \sum_{\alpha \in [n]^{d}} \frac{1}{2}(\theta_{\delta,\alpha}(\Lambda_+) - \theta^0_{\delta,\alpha})^2 I_{\delta,\alpha} - \sum_{\alpha \in [n]^{d}} (\theta_{\delta,\alpha}(\Lambda_+) - \theta^0_{\delta,\alpha}) M_{\delta,\alpha}\\ 
    &= \frac{1}{2} I_{T,\delta}(\theta_-,\theta_+,\Lambda_+) - M_{T,\delta}(\theta_-,\theta_+,\Lambda_+),
\end{split}
\end{equation}
where 
\begin{align*}
I_{T,\delta}(\theta_-,\theta_+,\Lambda_+) &\coloneqq \sum_{\alpha \in [n]^{d}} (\theta_{\delta,\alpha}(\Lambda_+) - \theta^0_{\delta,\alpha})^2 I_{\delta,\alpha}, \\ 
M_{T,\delta}(\theta_-,\theta_+,\Lambda_+) &\coloneqq \sum_{\alpha \in [n]^{d}} (\theta_{\delta,\alpha}(\Lambda_+) - \theta^0_{\delta,\alpha}) M_{\delta,\alpha}.
\end{align*}
Note in particular that $Z_\delta(\theta_-^0,\theta_+^0,\Lambda^\updownarrow_+) = 0$, so that $Z_\delta(\cdot)$ is centered around the truth $(\theta_-^0,\theta^0_+,\Lambda^\updownarrow_+)$. It will be crucial to have good control on the empirical process $Z_\delta$, which is provided by the following propositions that specify the order and concentration behavior of the observed Fisher informations $I_{\delta,\alpha}$ as well as of the terms $M_{\delta,\alpha}$ via a martingale coupling result based on the Dambis--Dubins--Schwarz theorem. The proofs are straightforward extensions of the corresponding one-dimensional results \cite[Lemma 3.3, Proposition 3.4, Proposition 3.5]{reiß_2023_change} using the analogous spectral arguments in higher dimensions and can therefore be  omitted.
\begin{proposition}[Modification of Lemma 3.3 and Proposition 3.4 in \cite{reiß_2023_change}]\label{lem: orderFisher}
\hspace{1pt}
\begin{itemize}
\item [(i)] For any $\alpha \in[n]^d$ with $\alpha\notin \mathcal{B},$ it holds that
\[\E[I_{\delta,\alpha}]=\frac{T}{2\theta_{\delta,\alpha}^0}\norm{\nabla K}^2_{L^2(\R^d)}\delta^{-2}+O(1).\]
\item [(ii)]For any $\alpha\in\mathcal{B},$ it holds that
\[\E[I_{\delta,\alpha}]\in \left[\frac{2\underline{\lambda}T-1+\mathrm{e}^{-2\underline{\lambda}T}}{4\underline{\lambda}\overline{\theta}}\norm{\nabla K}^2_{L^2(\R^d)}\delta^{-2},\frac{T}{2\underline{\theta}}\norm{\nabla K}^2_{L^2(\R^d)}\delta^{-2}\right].\]
\item [(iii)] For any vector $\beta\in \R^{n^d}$ with $\beta_{\alpha}=0$ if $\alpha\in\mathcal{B}$, we have
\[ \operatorname{Var}\Big(\sum_{\alpha\in[n]^d}\beta_{\alpha}I_{\delta,\alpha}\Big)\leq \frac{T}{2\underline{\theta}^3}\delta^{-2}\norm{\beta}^2_{l^2}\norm{\nabla K}^2_{L^2(\R^d)}.\]
\item[(iv)] For any vector $\beta\in \R_+^{n^d}$ with $\beta_{\alpha}=0$ if $\alpha\in\mathcal{B}$, we have
\[\PP\Big(\big\lvert \sum_{\alpha \in [n^d]}\beta_\alpha(I_{\delta,\alpha} - \E[I_{\delta,\alpha}]) \big\rvert \geq z\Big) \leq 2\exp\Big(-\frac{\underline{\theta}^2}{2\lVert \beta \rVert_\infty} \frac{z^2}{2z + \lVert \beta \rVert_{\ell^1}T \underline{\theta}^{-1}\lVert \nabla K \rVert_{L^2(\R^d)}^2\delta^{-2}} \Big), \quad z > 0. \]
\end{itemize}
\end{proposition}

\begin{proposition}[Modification of Proposition 3.5 in \cite{reiß_2023_change}] \label{prop:coupling}
Let 
\[\sigma_\alpha \coloneqq \inf\Big\{t \geq 0: \int_0^t X^{\Delta}_{\delta,\alpha}(s)^2 \diff{s} > \E[I_{\delta,\alpha}]\Big\}, \quad \alpha \in [n]^d,\] 
and  
\[\overline{M}_{\delta,\alpha} \coloneqq \int_0^{\sigma_\alpha} X^{\Delta}_{\delta,\alpha}(t) \diff{B_{\delta,\alpha}(t)}, \quad \alpha \in [n]^d.\] 
Then, $(\overline{M}_{\delta,\alpha})_{\alpha \in [n]^d}$ are independent with $\overline{M}_{\delta,\alpha} \sim \mathcal{N}(0,\E[I_{\delta,\alpha}])$ and  
\[\PP\Big(\big\lvert\sum_{\alpha \in [n]^d} \beta_\alpha (M_{\delta,\alpha} - \overline{M}_{\delta,\alpha}) \big\rvert \geq z, \sum_{\alpha \in [n]^d} \beta_\alpha^2 \lvert I_{\delta,\alpha} - \E[I_{\delta,\alpha}] \rvert \leq L \Big) \leq 2\exp\Big(-\frac{z^2}{L} \Big), \quad z,L > 0, \beta_\alpha \in \R, \alpha \in [n]^d.\]
\end{proposition}
Furthermore, convergence of the estimator $(\hat{\theta}_-,\hat{\theta}_+,\hat{\Lambda}_+)$ requires insight on the order of the remainders $R_{\delta,\alpha}$ in the representation \eqref{eq: est decomp}, given by the following lemma.
\begin{lemma} 
\label{lem: orderrest}
For any $\alpha \in \mathcal{B}$ we have 
\[\E\big[\lvert R_{\delta,\alpha} \rvert\big] \lesssim \delta^{-2}.\]
\end{lemma}

We are now ready to prove our first main result.
\begin{theorem} \label{theo:rate_cp}
Suppose that for some $\beta \in (0,1]$ and some constant $c > 0$ it holds that 
\begin{equation}\label{ass:bound}
\lvert \mathcal{B} \rvert \leq c \delta^{-d+\beta}.
\end{equation}
Then, for some absolute constant $C$ depending only on $c,d,\underline{\theta},\overline{\theta},T$ and $\underline{\eta}$ it holds that
\[\E\big[\lebesgue(\hat{\Lambda}_+ \vartriangle \Lambda^0_+) \big] \leq {C}\delta^\beta.\] 
\end{theorem}
\begin{proof} 
Let $\chi^0 \coloneqq (\theta^0_-, \theta^0_+, (\Lambda^0_+)^+) = (\theta_-^0,\theta_+^0,\Lambda^\updownarrow_+)$ and set for $\chi = (\theta_-,\theta_+,\Lambda_+) \in \Theta_- \times \Theta_+ \times \mathcal{A}_+,$
\[L_\delta(\chi) = \delta^{d+2} Z_\delta(\theta_-,\theta_+,\Lambda_+), \quad \tilde{L}_\delta(\chi) = \E[L_\delta(\chi)].\]
We first observe that  \eqref{eq: est decomp} implies  
\begin{equation}\label{eq:min_emp}
L_\delta(\hat{\theta}_-,\hat{\theta}_+, \hat{\Lambda}_+)  \leq \min_{\chi \in  \Theta_- \times \Theta_+ \times \mathcal{A}_+} L_\delta(\chi) + 2\delta^{2+d} \overline{\theta}\sum_{\alpha \in \mathcal{B}} \lvert R_{\delta,\alpha} \rvert.
\end{equation}
Furthermore, since  $L_\delta(\chi^0) = 0$ and $\E[M_{\delta,\alpha}] = 0$, we obtain with Proposition \ref{lem: orderFisher} that for any  $\chi = (\theta_-,\theta_+,\Lambda_+) \in \Theta_- \times \Theta_+ \times \mathcal{A}_+$, 
\begin{equation}\label{eq:low_exp}
\begin{split}
\tilde{L}_\delta(\chi) - \tilde{L}_\delta(\chi^0) &= \E[L_\delta(\chi)] \asymp \delta^d\sum_{\alpha\in[n]^d}(\theta_{\delta,\alpha}(\Lambda_+)-\theta^0_{\delta,\alpha})^2\\
&=\delta^d\sum_{\alpha\in[n]^d: \Sq(\alpha)^\circ \subset (\Lambda_+ \cap \Lambda^\updownarrow_+)}(\theta_+-\theta^0_+)^2+\delta^d\sum_{\alpha\in[n]^d: \Sq(\alpha)^\circ \subset (\Lambda_+ \cup \Lambda_+^\updownarrow)^{\mathrm{c}}}(\theta_--\theta^0_-)^2\\
&\quad +\delta^d\sum_{\alpha\in[n]^d: \Sq(\alpha)^\circ \subset (\Lambda_+ \vartriangle \Lambda^\updownarrow_+)}(\theta_{\delta,\alpha}(\Lambda_+) - \theta^0_{\delta,\alpha})^2\\
&\geq \underline{\eta}^2 \delta^d \big\lvert\big\{\alpha\in[n]^d: \Sq(\alpha)^\circ \subset (\Lambda_+ \vartriangle \Lambda^\updownarrow_+)\big\}\big\rvert   \\ 
&= \underline{\eta}^2 \lebesgue(\Lambda_+ \vartriangle \Lambda^\updownarrow_+),
\end{split}
\end{equation}
where for the penultimate line we observe that for $\Sq(\alpha)^\circ \subset (\Lambda_+ \vartriangle \Lambda^\updownarrow_+)$ we have $\theta_{\delta,\alpha}(\Lambda_+) = \theta_\pm$ if and only if $\theta^0_{\delta,\alpha} = \theta^0_\mp$, and conclude with $\theta_\pm, \theta^0_\pm \in \Theta_\pm$ and the fact that $\Theta_+$ and $\Theta_-$ are $\underline{\eta}$-separated. Now, using the characterization from Lemma \eqref{lem:tiling_boundary} and the assumption \eqref{ass:bound}, it follows that 
\begin{equation}\label{eq:tiling_err}
\lebesgue(\Lambda^0_+ \vartriangle \Lambda^\updownarrow_+) \leq \delta^d \lvert \mathcal{B} \rvert \leq c\delta^\beta.
\end{equation}
Combined with \eqref{eq:low_exp}, triangle inequality for the symmetric difference pseudometric therefore yields 
\[\lebesgue(\Lambda_+ \vartriangle \Lambda^0_+) \lesssim \tilde{L}_\delta(\chi) - \tilde{L}_\delta(\chi^0) + \delta^\beta, \quad \chi = (\theta_-,\theta_+,\Lambda_+) \in \Theta_- \times \Theta_+ \times \mathcal{A}_+,\]
whence the assertion follows once we have verified that 
\begin{equation} \label{eq:cons_rate1}
\E[\tilde{L}_\delta(\hat{\chi}) - \tilde{L}_\delta(\chi^0)] \lesssim \delta^\beta, \quad \hat{\chi} \coloneqq (\hat{\theta}_-,\hat{\theta}_+, \hat{\Lambda}_+),
\end{equation}
where by definition $\tilde{L}_\delta(\hat{\chi}) = \E[L_\delta(\chi)]\vert_{\chi = \hat{\chi}}$.  
Taking into account \eqref{ass:bound}, \eqref{eq:min_emp} and Lemma \ref{lem: orderrest}, we arrive at 
\begin{equation}\label{eq:basic_up}
\begin{split}
\E[\tilde{L}_\delta(\hat{\chi}) - \tilde{L}_\delta(\chi^0) ] &\leq \E\big[\tilde{L}_\delta(\hat{\chi}) - \tilde{L}_\delta(\chi^0) + L_\delta(\chi^0) - L_\delta(\hat{\chi})\big]\\
&\quad + 2\delta^{d+2} c\delta^{\beta -d} \max_{\alpha \in \mathcal{B}} \E[\lvert R_{\delta,\alpha}\rvert]\\  
&\lesssim \E\Big[\sup_{\chi \in \Theta_- \times \Theta_+ \times \mathcal{A}_+} \lvert L_\delta(\chi) - \tilde{L}_\delta(\chi) \rvert \Big]+ \delta^\beta. 
\end{split}
\end{equation}
To prove \eqref{eq:cons_rate1} it therefore remains to show
\[\E\Big[\sup_{\chi \in \Theta_- \times \Theta_+ \times \mathcal{A}_+} \lvert L_\delta(\chi) - \tilde{L}_\delta(\chi) \rvert \Big] \lesssim \delta^\beta,\]
which, recalling the decomposition \eqref{eq: Z} and using $\E[M_{T,\delta}(\chi)] = 0$, boils down to show 
\begin{align} 
\E\Big[\sup_{\chi \in \Theta_- \times \Theta_+ \times \mathcal{A}_+} \lvert I_{T,\delta}(\chi) - \E[I_{T,\delta}(\chi)]\rvert \Big] &\lesssim \delta^{\beta - d -2 }\label{eq:fisher_uni},\\ 
\E\Big[\sup_{\chi \in \Theta_- \times \Theta_+ \times \mathcal{A}_+} \lvert M_{T,\delta}(\chi)\rvert \Big] &\lesssim \delta^{\beta-d -2}.\label{eq:mart_uni}
\end{align}
Clearly, by triangle inequality, we get 
\begin{equation}\label{eq:rough}
\begin{split}
\sup_{\chi \in \Theta_- \times \Theta_+ \times \mathcal{A}_+} \lvert I_{T,\delta}(\chi) - \E[I_{T,\delta}(\chi)]\rvert &\leq (\overline{\theta} - \underline{\theta})^2 \sum_{\alpha \in [n]^d} \lvert I_{\delta,\alpha} - \E[I_{\delta,\alpha}] \rvert,\\ 
\sup_{\chi \in \Theta_- \times \Theta_+ \times \mathcal{A}_+} \lvert M_{T,\delta}(\chi) \rvert &\leq \lvert \overline{\theta} - \underline{\theta}\rvert \sum_{\alpha \in [n]^d} \lvert M_{\delta,\alpha} \rvert.
\end{split}
\end{equation}
Thus, applying once more the assumption \eqref{ass:bound} and the bounds from Proposition \ref{lem: orderFisher},
\begin{align*} 
&\E\Big[\sup_{\chi \in \Theta_- \times \Theta_+ \times \mathcal{A}_+} \lvert I_{T,\delta}(\chi) - \E[I_{T,\delta}(\chi)]\rvert\Big]\\ 
&\,\lesssim \delta^{\beta -d} \max_{\alpha \in \mathcal{B} } \E[\lvert I_{\delta,\alpha} - \E[I_{\delta,\alpha}] \rvert] + \delta^{-d} \max_{\alpha \in [n]^d \setminus \mathcal{B}} \E[\lvert I_{\delta,\alpha} - \E[I_{\delta,\alpha}] \rvert]\\ 
&\,\leq 2\delta^{\beta  -d } \max_{\alpha \in \mathcal{B}} \E[\lvert I_{\delta,\alpha}  \rvert] + \delta^{-d} \max_{\alpha \in [n]^d\setminus \mathcal{B}} \operatorname{Var}( I_{\delta,\alpha})^{1/2} \\ 
&\,\lesssim \delta^{\beta -d -2} + \delta^{-d-1},
\end{align*}
which establishes \eqref{eq:fisher_uni}. For the martingale part, using $\E[\lvert M_{\delta,\alpha} \rvert^2] = \E[I_{\delta,\alpha}]$, Cauchy--Schwarz and Proposition \ref{lem: orderFisher}, we have
\begin{align*} 
\E\Big[\sup_{\chi \in \Theta_- \times \Theta_+ \times \mathcal{A}_+} \lvert M_{T,\delta}(\chi)\rvert\Big] &\lesssim \delta^{-d} \max_{\alpha \in [n]^d} \E[\lvert M_{\delta,\alpha} \rvert]\\ 
&\leq \delta^{-d} \max_{\alpha \in [n]^d} (\E[ I_{\delta,\alpha} ])^{1/2} \\ 
&\lesssim \delta^{-d-1},
\end{align*}
showing \eqref{eq:mart_uni}. This finishes the proof. 
\end{proof} 

This result entails convergence rates for the estimation of domains $\Lambda^0_+$ with boundary of Minkowski dimension (also called \textit{box counting dimension}) $d-\beta$. Recall that if for a set $A \subset [0,1]^d$ we let $N(A,\delta)$ be the minimal number of hypercubes $\Sq(\alpha)$ needed to cover $A$ and set 
\[\underline{\operatorname{dim}}_{\mathcal{M}}(A) =  \liminf_{\delta \to 0} \frac{\log N(A,\delta)}{\log 1/\delta}, \quad \overline{\operatorname{dim}}_{\mathcal{M}}(A) = \limsup_{\delta \to 0} \frac{\log N(A,\delta)}{\log 1/\delta},\]
then, if $\underline{\operatorname{dim}}_{\mathcal{M}}(A) = \overline{\operatorname{dim}}_{\mathcal{M}}(A)$, we call 
\[\operatorname{dim}_{\mathcal{M}}(A) = \lim_{\delta \to 0} \frac{\log N(A,\delta)}{\log 1/\delta},\]
the Minkowski dimension of $A$ (see \cite[p.2]{bishop17} for the fact that this is an equivalent characterization of the Minkowski dimension in the Euclidean space $[0,1]^d$).
Clearly, 
\[\lvert \mathcal{B} \rvert \leq N(\partial \Lambda^0_+,\delta),\]
so that Theorem \ref{theo:rate_cp} yields the following corollary. 
\begin{corollary}
\label{cor: minkowski}
Suppose that for some $\beta \in (0,1]$ it holds that 
\[\operatorname{dim}_{\mathcal{M}}(\partial \Lambda^0_+) \leq d- \beta.\]
Then, for any $\varepsilon > 0$,
\[\E\big[\lebesgue(\hat{\Lambda}_+ \vartriangle \Lambda^0_+) \big] = o(\delta^{\beta - \varepsilon}).\] 
\end{corollary}
\begin{remark} 
The Minkowski dimension $\operatorname{dim}_{\mathcal{M}}$ always dominates the Hausdorff dimension $\operatorname{dim}_{\mathcal{H}}$. For many reasonable sets they coincide and in these cases the condition on  $\operatorname{dim}_{\mathcal{M}}(\partial\Lambda_+^0)$ may be replaced by the same one on $\operatorname{dim}_{\mathcal{H}}(\partial\Lambda_+^0)$. In most cases, $\dim_{\mathcal{M}}(\partial \Lambda^0_+) = d -\beta$ is verified explicitly by establishing that $N(\partial \Lambda^0_+,\delta) \asymp c\delta^{-d+\beta}$, which then improves the result to $\E\big[\lebesgue(\hat{\Lambda}_+ \vartriangle \Lambda^0_+) \big] = \mathcal{O}(\delta^{\beta})$.
\end{remark}

Given a necessary identifiability assumption,  the diffusivity parameters $\theta^0_{\pm}$ are consistently estimated by $\hat{\theta}_{\pm}$ under the assumptions of Theorem \ref{theo:rate_cp}.
\begin{corollary} \label{cor: rate}
Suppose that $\lvert \mathcal{B} \rvert \leq c\delta^{-d+\beta}$ for some $\beta \in (0,1]$ and some constant $c > 0$. If $\liminf_{\delta \to 0} \lebesgue(\Lambda_\pm^0) > 0$, then $\hat{\theta}_\pm$  is a  consistent estimator satisfying $\lvert \hat{\theta}_\pm - \theta^0_\pm \rvert = \mathcal{O}_{\PP}(\delta^{\beta/2})$. In particular, if both $\Lambda_+^0$ and $\Lambda_-^0$ have volume uniformly bounded away from 0, it holds that $\lVert \hat{\theta} - \theta^0 \rVert_{L^1((0,1)^d)} \in \mathcal{O}_{\PP}(\delta^{\beta/2}).$
\end{corollary}
The proof is deferred to Section \ref{app:proofs}.
Given the results from \cite{altmeyer_anisotrop2021}, which can be applied to a stochastic heat equation with constant diffusivity, this rate is not expected to be optimal. As argued in \cite{reiß_2023_change}, a more careful approach in the design of the estimator that accounts for the irregularities of the diffusivity in a more elaborate way would be needed to also achieve rate-optimality for the diffusivity parameters. We emphasize this aspect in the next subsection, where we show that if the technical difficulties arising from hypercube intersections with the boundary $\partial \Lambda^0_+$ are eliminated by assuming that $\Lambda^0_+$ is a tiling set, we can construct diffusivity estimators that satisfy a CLT at minimax optimal rate.

\subsection{Convergence results under a tiling set assumption}\label{sec:tiling}
A close investigation of the proof of Theorem \ref{theo:rate_cp} reveals that the error rate is dominated in two ways by $\Lambda^0_+$ generally not being a tiling set: first, there is a bias of order $\delta^\beta$ when comparing $\Lambda^0_+$ to its minimal tiling $\Lambda^\updownarrow_+$; second, the total contribution of the remainder terms $R_{\delta,\alpha}$ for $\alpha \in \mathcal{B}$ is of order $\delta^\beta$ as well. For this reason, it was not necessary to carry out a finer concentration analysis of the empirical process $\mathcal{L}_\delta$, but it sufficed to exploit a rough triangle inequality bound in \eqref{eq:rough}. 

In contrast, we now proceed with a scenario that is not obscured by any geometric bias, by studying the rate of convergence of the estimator when $\Lambda^0_+$ is a tiling set, that is $\overline{\Lambda^0_+} \in \mathcal{P}$, or, equivalently, $\overline{\Lambda^0_+} = (\Lambda^0_+)^+ = \Lambda^\updownarrow_+$. Note that this implies that we must allow $\Lambda^0_+$ to be dependent on $\delta$, unless we can choose some subsequence $(\delta_k = 1/n_k)_{k \in \N}$ that yields a sequence of refining grids that anchor $\Lambda^0_+$ for $k$ large enough, say $n_k = l^k$ for some $1< l \in \N$. We emphasize that while this is clearly a mathematically idealized setting, it is akin to assumptions imposed in related works \cite{mueller_cube_94,mueller_indexed_96} on image reconstruction problems. We will work with the following set of assumptions.

\begin{assumption}\label{ass:bounds_domain}
There exists $\kappa \in (0,1/2)$ such that 
\[\kappa < \liminf_{\delta \to 0} \lebesgue(\Lambda^0_+) \leq \limsup_{\delta \to 0} \lebesgue(\Lambda^0_+) < 1 - \kappa.\]
\end{assumption}

In particular, we rule out the case of a heat equation with constant diffusivity. This is only for ease of exposition and this particular border case could be dealt with  analogous arguments.

\begin{assumption}\label{ass:candidates} 
The size of $\mathcal{A}_+$ grows at most polynomially in $\delta^{-1}$, i.e., $\lvert \mathcal{A}_+ \rvert \lesssim \delta^{-c}$ for some $c > 0$.
\end{assumption}
\begin{remark} 
In the special case where $\Lambda^0_+$ is a fixed tiling set and, accordingly, we only consider a refining grid sequence, both assumptions can be trivially satisfied (provided $\overline{\Lambda^0_+} \neq \varnothing$), by fixing $\mathcal{A}_+ = \mathcal{P}(\delta^\prime)$ for fixed $\delta^\prime$ small enough such that $\overline{\Lambda^0_+} \in \mathcal{P}(\delta^\prime)$. More generally, if for given $\delta$ we have $\overline{\Lambda^0_+} \in \mathcal{P}(\delta^\prime)$ for $1/\delta^\prime \lesssim (\log(\delta^{-1}))^{1/d}$, Assumption \ref{ass:candidates} is satisfied for the canonical choice $\mathcal{A}_+(\delta) = \mathcal{P}(\delta^\prime)$. These situations are comparable to \cite{mueller_cube_94}. As demonstrated in the next section, specific geometric assumptions on $\Lambda^0_+$ can also induce appropriate classes of candidate sets $\mathcal{A}_+$ such that the polynomial growth condition is satisfied.
\end{remark}

As a slight reinforcement to Assumption \ref{ass:bounds_domain} we also introduce the following. 

\begin{assumption}\label{ass:volume_conv}
It holds that 
\[\lim_{\delta \to 0} \lebesgue(\Lambda^0_+) = \nu^0_+ \in (0,1).\]
\end{assumption}

As we now show, we can perfectly identify $\Lambda^0_\pm$ with overwhelming probability as the resolution $\delta$ approaches zero in this setting. To this end, 
we first observe that $\overline{\Lambda^0_+} \in \mathcal{P}$ implies that $\mathcal{B} = \varnothing$. For fixed $\Lambda_+ \in \mathcal{A}_+ \setminus \{\varnothing,[0,1]^d\}$ this yields that the modified log-likelihood 
\[\R^2 \ni (\theta_-,\theta_+) \mapsto  \sum_{\alpha \in [n]^d} \ell_{\delta,\alpha}(\theta_-,\theta_+,\Lambda_+)\]
is uniquely maximized in 
\begin{equation}
\label{eq: theta_augm}
\theta_-^{\Lambda_+} \coloneqq \frac{\sum_{\Sq(\alpha)^\circ \subset \Lambda_-} \int_0^T X^{\Delta}_{\delta,\alpha}(t) \diff{X_{\delta,\alpha}(t)}}{\sum_{\Sq(\alpha)^\circ \subset \Lambda_-} I_{\delta,\alpha}}, \quad \theta_+^{\Lambda_+} \coloneqq \frac{\sum_{\Sq(\alpha)^\circ \subset \Lambda_+} \int_0^T X^{\Delta}_{\delta,\alpha}(t) \diff{X_{\delta,\alpha}(t)}}{\sum_{\Sq(\alpha)^\circ \subset \Lambda_+} I_{\delta,\alpha}},
\end{equation}
where $\Lambda_- \coloneqq [0,1]^d \setminus \Lambda_+$. 
Thus, if  we set
\[\tilde{\Lambda}_+ \in \argmax_{\Lambda_+ \in \mathcal{A}_+^\prime} \sum_{\alpha \in [n]^d} \ell_{\delta,\alpha}\big(\theta_-^{\Lambda_+},\theta_+^{\Lambda_+}, \Lambda_+ \big),\]
for $\mathcal{A}_+^\prime = \mathcal{A}_+ \setminus \{\varnothing, [0,1]^d\}$ and 
\[\big(\tilde{\theta}_-,\tilde{\theta}_+\big) = \big(\theta_-^{\tilde{\Lambda}_+}, \theta_+^{\tilde{\Lambda}_+} \big),\]
then, since $\mathcal{B}=\varnothing,$
\[(\tilde{\theta}_-,\tilde{\theta}_+,\tilde{\Lambda}_+) \in \argmax_{(\theta_-,\theta_+,\Lambda_+) \in \R^2 \times \mathcal{A}_+^\prime} \sum_{\alpha \in [n]^d} \ell_{\delta,\alpha}(\theta_-,\theta_+,\Lambda_+) = \argmin_{(\theta_-,\theta_+,\Lambda_+) \in \R^2 \times \mathcal{A}_+^\prime} Z_{\delta}(\theta_-,\theta_+,\Lambda_+).\]
Only optimizing over $\mathcal{A}^\prime_+$ instead of $\mathcal{A}_+$ is in line with Assumption \ref{ass:bounds_domain}. Although this is not strictly necessary it simplifies notation in the following.  
As discussed before, not restricting the optimization domain for the diffusivity parameters in the construction of $(\tilde{\theta}_-,\tilde{\theta}_+, \tilde{\Lambda}_+)$ can only result in a good estimator for the pairs $(\theta^0_\pm, \Lambda^0_\pm)$ if geometric constraints on $\Lambda^0_\pm$ are represented in the candidate set $\mathcal{A}_+$. In fact, to obtain a precise allocation of domains and corresponding diffusivities it is sufficient to have not only the geometric constraint $\overline{\Lambda_+^0} \in \mathcal{A}_+$, but also $\Lambda^0_- \notin \mathcal{A}_+$.  Otherwise, $\tilde{\Lambda}_+$ has the characteristics of an unsupervised learner and will therefore only aim at clustering in the sense that $\partial \tilde{\Lambda}_+$ is a good separator of the true domains $\Lambda^0_\pm$. In this case, we may exploit the consistent estimators $\hat{\Lambda}_\pm$ from before, which allowed identification by incorporating information on $\underline{\eta}$-separated domains $\Theta_\pm$ of the diffusivities $\theta^0_\pm$. More precisely, we set 
\begin{align*}
\hat{\Lambda}^\ast_+ = \begin{cases} \tilde{\Lambda}_+, &\text{if } \lebesgue(\tilde{\Lambda}_+ \cap \hat{\Lambda}_+) \geq \lebesgue(\tilde{\Lambda}_+ \cap \hat{\Lambda}_-), \\ 
\overline{[0,1]^d\setminus \tilde{\Lambda}_+}, &\text{else},
\end{cases} 
\end{align*}
and 
\begin{align*}
(\hat{\theta}^\ast_-,\hat{\theta}_+^\ast) &= \begin{cases} (\tilde{\theta}_-,\tilde{\theta}_+), &\text{if } \hat{\Lambda}^\ast_+ = \tilde{\Lambda}_+,\\ (\tilde{\theta}_+,\tilde{\theta}_-), &\text{else,}\end{cases}\\
&= \big(\theta_-^{\hat{\Lambda}_+^\ast}, \theta_+^{\hat{\Lambda}^\ast_+}\big).
\end{align*}
That is, $\tilde{\Lambda}_+$ is assigned to $\Lambda^0_\pm$ according to the size of the overlap with $\hat{\Lambda}_+$ and the diffusivity estimators are permuted accordingly. For the next theorem, let us introduce the Hausdorff distance 
\[d_H(A,B) \coloneqq \max\big\{\sup_{x \in B} d(x,A), \sup_{x \in A} d(x,B) \big\}, \quad A,B \subset \R^d,\] 
on $\R^d$.

\begin{theorem}\label{theo:tiling} 
Suppose that $\overline{\Lambda^0_+} = \Lambda^\updownarrow_+$ and that Assumptions \ref{ass:bounds_domain} and  \ref{ass:candidates} hold. Then, 
\[\lim_{\delta \to 0} \PP\Big(\tilde{\Lambda}_+ \in \big\{\overline{\Lambda^0_+}, \Lambda^0_-\big\}\Big)  = 1.\]
In particular,
\[\lim_{\delta \to 0} \PP\Big(d_H\big(\partial \tilde{\Lambda}_+, \partial \Lambda^0_+\big) = 0\Big) = 1,\]
and, if $\Lambda^0_- \notin \mathcal{A}_+$, also 
\[\lim_{\delta \to 0} \PP\Big(\tilde{\Lambda}_+ = \overline{\Lambda^0_+}\Big) = 1.\]
Moreover, 
\[\lim_{\delta \to 0} \PP\Big(\hat{\Lambda}^\ast_+ = \overline{\Lambda^0_+}\Big)  = 1.\]
\end{theorem}

    This domain identification result is much stronger than findings in the associated image reconstruction literature. Indeed, the main result of \cite{mueller_cube_94} states the convergence rate $\sqrt{\log(\delta^{-d})\delta^{d}}$ for the volume of the symmetric difference $\lebesgue(\tilde{\Lambda}_+\vartriangle \Lambda_+^0),$ given the additional identifiability criterion
    \[\lebesgue(\Lambda_+\cap\Lambda_+^0)\geq\kappa,\quad \lebesgue(\Lambda_-\cap \Lambda_-^0)\geq\kappa,\quad \forall\Lambda_+\in\mathcal{A}_+.\]
    Under such extra constraint, the proof of Theorem \ref{theo:tiling} reveals the convergence rate
    \[\lebesgue(\tilde{\Lambda}_+\vartriangle\Lambda_+^0)=O_{\PP}(\delta^{d+1}\ln(\delta^{-1})).\]
    Consequently, we can identify $\overline{\Lambda_+^0}$ with overwhelming probability as even the slightest deviation of $\overline{\Lambda_+^0}$, that is a single wrongly assigned hypercube $\Sq(\alpha),$ leads to an error larger than $\delta^d.$ It is therefore of no surprise that we can also identify $\theta_\pm^0$ with better rates leading to the following CLT for the diffusivity estimators $(\hat{\theta}^\ast_-, \hat{\theta}^\ast_+)$ at rate $\delta^{d/2+1},$ which is the minimax optimal rate for diffusivity estimation without change points found in \cite{altmeyer_anisotrop2021}.

This domain identification result also allows us to prove a CLT for the diffusivity estimators $(\hat{\theta}^\ast_-, \hat{\theta}^\ast_+)$ at rate $\delta^{-(d/2+1)},$ which is the minimax optimal rate for diffusivity estimation without change points found in \cite{altmeyer_anisotrop2021}. 

\begin{theorem}\label{theo:clt}
Suppose  that $\overline{\Lambda^0_+} = \Lambda^\updownarrow_+$ and that Assumptions \ref{ass:bounds_domain} Assumption \ref{ass:volume_conv} hold. Then, 
\[\delta^{-(d/2+1)}\big(\hat{\theta}^\ast_\pm - \theta^0_\pm) \overset{d}{\longrightarrow} \mathcal{N}\left(0, \frac{2\theta^0_\pm}{T\lVert \nabla K\rVert^2_{L^2}v^0_\pm} \right),\]
where $\nu^0_- = 1 - \nu^0_+$. 
If $\Lambda^0_- \notin \mathcal{A}_+$, then also
\[\delta^{-(d/2+1)}\big(\tilde{\theta}_\pm - \theta^0_\pm) \overset{d}{\longrightarrow} \mathcal{N}\left(0, \frac{2\theta^0_\pm}{T\lVert \nabla K\rVert^2_{L^2}v^0_\pm} \right).\]
\end{theorem}

To prepare the proofs of these theorems, we make some important preliminary observations. Fix some $\Lambda_+ \in \mathcal{A}_+^\prime$. Using Proposition \ref{prop: genweaksol} to decompose the domain-specific estimators defined in \eqref{eq: theta_augm}, the empirical process $Z_{\delta}(\theta_-^{\Lambda_+},\theta_+^{\Lambda_+},\Lambda_+)$ admits the representation
\begin{equation}\label{eq:Z_split}
\begin{split}
    Z_{\delta}(\theta_-^{\Lambda_+},\theta_+^{\Lambda_+},\Lambda_+)&=\frac{\eta^2}{2}\frac{\sum_{\Sq(\alpha)^\circ \subset\Lambda_+\cap\Lambda_+^0}I_{\delta,\alpha}\sum_{\Sq(\alpha)^\circ \subset\Lambda_+\setminus\Lambda_+^0}I_{\delta,\alpha}}{\sum_{\Sq(\alpha)^\circ \subset\Lambda_+}I_{\delta,\alpha}}+\frac{\eta^2}{2}\frac{\sum_{\Sq(\alpha)^\circ \subset\Lambda_-\cap\Lambda_-^0}I_{\delta,\alpha}\sum_{\Sq(\alpha)^\circ \subset\Lambda_+^0\setminus\Lambda_+}I_{\delta,\alpha}}{\sum_{\Sq(\alpha)^\circ \subset\Lambda_-}I_{\delta,\alpha}}\\
    &\phantom{=}-\frac{(\sum_{\Sq(\alpha)^\circ \subset\Lambda_+}M_{\delta,\alpha})^2}{2\sum_{\Sq(\alpha)^\circ \subset\Lambda_+}I_{\delta,\alpha}}-\frac{(\sum_{\Sq(\alpha)^\circ \subset\Lambda_-}M_{\delta,\alpha})^2}{2\sum_{\Sq(\alpha)^\circ \subset\Lambda_-}I_{\delta,\alpha}}\\
    &\phantom{=}-\frac{\eta}{\sum_{\Sq(\alpha)^\circ \subset\Lambda_+}I_{\delta,\alpha}}\left(\sum_{\Sq(\alpha)^\circ \subset\Lambda_+\cap\Lambda_-^0}M_{\delta,\alpha}\sum_{\Sq(\alpha)^\circ \subset\Lambda_+\cap \Lambda_+^0}I_{\delta,\alpha}-\sum_{\Sq(\alpha)^\circ \subset\Lambda_+\cap\Lambda_+^0}M_{\delta,\alpha}\sum_{\Sq(\alpha)^\circ \subset\Lambda_+\cap\Lambda_-^0}I_{\delta,\alpha}\right)\\
    &\phantom{=} -\frac{\eta}{\sum_{\Sq(\alpha)^\circ \subset\Lambda_-}I_{\delta,\alpha}}\left(\sum_{\Sq(\alpha)^\circ \subset\Lambda_-\cap\Lambda_+^0}I_{\delta,\alpha}\sum_{\Sq(\alpha)^\circ \subset\Lambda_-\cap\Lambda_-^0}M_{\delta,\alpha}-\sum_{\Sq(\alpha)^\circ \subset\Lambda_-\cap\Lambda_-^0}I_{\delta,\alpha}\sum_{\Sq(\alpha)^\circ \subset\Lambda_-\cap\Lambda_+^0}M_{\delta,\alpha}\right),
\end{split}
\end{equation}
which is proven in Section \ref{app:proofs}.
Note here  that 
\[Z_{\delta}(\theta_-^{\Lambda^0_-}, \theta_+^{\Lambda^0_-}, \Lambda^0_-) = Z_{\delta}(\theta_-^{\overline{\Lambda^0_+}}, \theta_+^{\overline{\Lambda^0_+}}, \Lambda^0_+), \]
so that the estimator $\tilde{\Lambda}_+$, which optimizes over $\mathcal{A}_+^\prime$, cannot distinguish between $\Lambda^0_+$ and $\Lambda^0_-$ whenever $\Lambda^0_- \in \mathcal{A}_+$. This is why in this case, $\tilde{\Lambda}_+$ can only identify the domain split but not $\Lambda^0_+$ directly, as discussed before. 

To control the above terms we establish the following technical results.
\begin{lemma} \label{lem:order_frac}
The following hold: 
\begin{enumerate} 
\item[(i)]
\[\sup_{\Lambda_+\in \mathcal{A}_+^\prime} \frac{\sum_{\Sq(\alpha)^\circ \subset \Lambda_\pm} M_{\delta,\alpha}}{\sum_{\Sq(\alpha)^\circ \subset \Lambda_\pm} I_{\delta,\alpha}} = \mathcal{O}_{\PP}\Big(\sqrt{\delta\log(\delta^{-1})}\Big).
\]
\item[(ii)] \[\sup_{\Lambda_+ \in \mathcal{A}_+^\prime: \Lambda_\pm \cap \Lambda^0_\pm \neq \varnothing} \frac{\sum_{\Sq(\alpha) \subset \Lambda_\pm \cap \Lambda^0_\pm} M_{\delta,\alpha}}{\sum_{\Sq(\alpha) \subset \Lambda_\pm \cap \Lambda^0_\pm} I_{\delta,\alpha}} = \mathcal{O}_{\PP}\Big(\sqrt{\delta\log(\delta^{-1})}\Big)\]
and also 
\[\sup_{\Lambda_+ \in \mathcal{A}_+^\prime: \Lambda_\pm \cap \Lambda^0_\mp \neq \varnothing} \frac{\sum_{\Sq(\alpha)^\circ \subset \Lambda_\pm \cap \Lambda^0_\mp} M_{\delta,\alpha}}{\sum_{\Sq(\alpha)^\circ \subset \Lambda_\pm \cap \Lambda^0_\mp} I_{\delta,\alpha}} = \mathcal{O}_{\PP}\Big(\sqrt{\delta\log(\delta^{-1})}\Big).\]
\item[(iii)] \[\sup_{\Lambda_+ \in \mathcal{A}_+^\prime} \frac{(\sum_{\Sq(\alpha)^\circ \subset \Lambda_\pm} M_{\delta,\alpha})^2}{\sum_{\Sq(\alpha)^\circ \subset \Lambda_\pm} I_{\delta,\alpha}} = \mathcal{O}_{\PP}\Big(\delta^{-1}\log(\delta^{-1})\Big).\]
\item[(iv)] \[\sup_{\Lambda_+ \in \mathcal{A}_+^\prime} \frac{\sum_{\Sq(\alpha)^\circ \subset \Lambda_\pm}\E[I_{\delta,\alpha}]}{\sum_{\Sq(\alpha)^\circ \subset \Lambda_\pm} I_{\delta,\alpha}} = 1 + \mathcal{O}_{\PP}\Big(\delta \sqrt{\log(\delta^{-1})} \Big).\]
\item[(v)] \[\sup_{\Lambda_+ \in \mathcal{A}_+^\prime} \frac{\sum_{\Sq(\alpha)^\circ \subset \Lambda_\pm}I_{\delta,\alpha}}{\sum_{\Sq(\alpha)^\circ \subset \Lambda_\pm} \E[I_{\delta,\alpha}]} = 1 + \mathcal{O}_{\PP}\Big(\delta \sqrt{\log(\delta^{-1}}) \Big).\]
and also 
\begin{align*}
\sup_{\Lambda_+ \in \mathcal{A}_+^\prime: \Lambda_\pm \cap \Lambda^0_{\pm} \neq \varnothing} \frac{\sum_{\Sq(\alpha)^\circ \subset \Lambda_\pm \cap \Lambda^0_\pm}I_{\delta,\alpha}}{\sum_{\Sq(\alpha)^\circ \subset \Lambda_\pm \cap \Lambda^0_\pm} \E[I_{\delta,\alpha}]} &= 1 + \mathcal{O}_{\PP}\Big(\delta \sqrt{\log(\delta^{-1}}) \Big), \\ 
\sup_{\Lambda_+ \in \mathcal{A}_+^\prime: \Lambda_\pm \cap \Lambda^0_{\mp} \neq \varnothing} \frac{\sum_{\Sq(\alpha)^\circ \subset \Lambda_\pm \cap \Lambda^0_\mp}I_{\delta,\alpha}}{\sum_{\Sq(\alpha)^\circ \subset \Lambda_\pm \cap \Lambda^0_\mp} \E[I_{\delta,\alpha}]} &= 1 + \mathcal{O}_{\PP}\Big(\delta \sqrt{\log(\delta^{-1}}) \Big)
\end{align*}
\end{enumerate}
\end{lemma}
The proof, which is based on the concentration results Proposition \ref{lem: orderFisher} and Proposition \ref{prop:coupling}, is deferred to Section \ref{app:proofs}. This result yields the following corollary, which is also proved in Section \ref{app:proofs}.

\begin{corollary}\label{coro:order_frac}
It holds that 
\[\frac{\delta^{d+2}}{\sum_{\Sq(\alpha)^\circ \subset \tilde{\Lambda}_+} I_{\delta,\alpha}}\sum_{\Sq(\alpha)^\circ \subset \tilde{\Lambda}_+ \cap \Lambda^0_\mp} I_{\delta,\alpha}\sum_{\Sq(\alpha) \subset \tilde{\Lambda}_+ \cap \Lambda^0_\pm} M_{\delta,\alpha}  \asymp \frac{\lebesgue(\tilde{\Lambda}_+ \setminus \Lambda^0_+)\lebesgue(\tilde{\Lambda}_+ \cap \Lambda^0_+)}{\lebesgue(\tilde{\Lambda}_+)}\mathcal{O}_{\PP}\Big(\sqrt{\delta \log(\delta^{-1})}\Big), \]
and
\[\frac{\delta^{d+2}}{\sum_{\Sq(\alpha)^\circ \subset \tilde{\Lambda}_-} I_{\delta,\alpha}}\sum_{\Sq(\alpha)^\circ \subset \tilde{\Lambda}_- \cap \Lambda^0_\mp} I_{\delta,\alpha}\sum_{\Sq(\alpha) \subset \tilde{\Lambda}_- \cap \Lambda^0_\pm} M_{\delta,\alpha}  \asymp \frac{\lebesgue(\Lambda^0_+\setminus \tilde{\Lambda}_+ )\lebesgue(\tilde{\Lambda}_- \cap \Lambda^0_-)}{\lebesgue(\tilde{\Lambda}_-)}\mathcal{O}_{\PP}\Big(\sqrt{\delta \log(\delta^{-1})}\Big), \]
as well as 
\[\delta^{d+2} \frac{\sum_{\Sq(\alpha)^\circ \subset \tilde{\Lambda}_+ \setminus \Lambda^0_+} I_{\delta,\alpha}\sum_{\Sq(\alpha)^\circ  \in \tilde{\Lambda}_+ \cap \Lambda^0_+} I_{\delta,\alpha}}{\sum_{\Sq(\alpha)^\circ \subset \tilde{\Lambda}_+} I_{\delta,\alpha}} \asymp \frac{\lebesgue(\tilde{\Lambda}_+ \setminus \Lambda^0_+)\lebesgue(\tilde{\Lambda}_+ \cap \Lambda^0_+)}{\lebesgue(\tilde{\Lambda}_+)}\Big(1 + \mathcal{O}_{\PP}\big(\delta \sqrt{\log(\delta^{-1})}\big)\Big), \]
and 
\[\delta^{d+2} \frac{\sum_{\Sq(\alpha)^\circ \subset \tilde{\Lambda}_- \cap \Lambda^0_+} I_{\delta,\alpha}\sum_{\Sq(\alpha)^\circ  \in \tilde{\Lambda}_- \cap \Lambda^0_-} I_{\delta,\alpha}}{\sum_{\Sq(\alpha)^\circ \subset \tilde{\Lambda}_-} I_{\delta,\alpha}} \asymp \frac{\lebesgue( \Lambda^0_+ \setminus \tilde{\Lambda}_+)\lebesgue(\tilde{\Lambda}_- \cap \Lambda^0_-)}{\lebesgue(\tilde{\Lambda}_-)}\Big(1 + \mathcal{O}_{\PP}\big(\delta \sqrt{\log(\delta^{-1})}\big)\Big). \]
\end{corollary} 
With these technical preparations we can now prove Theorem \ref{theo:tiling}.
\begin{proof}[Proof of Theorem \ref{theo:tiling}]
By Lemma \ref{lem:order_frac} and Corollary \ref{coro:order_frac}, it follows from \eqref{eq:Z_split} that 
\begin{align*}
\delta^{d+2} Z_\delta\big(\tilde{\theta}_-,\tilde{\theta}_+, \tilde{\Lambda}_+\big) &= \delta^{d+2} Z_\delta\big(\theta_-^{\tilde{\Lambda}_+},\theta_+^{\tilde{\Lambda}_+}, \tilde{\Lambda}_+\big)\\
&\asymp \Big(\frac{\lebesgue(\tilde{\Lambda}_+\setminus\Lambda_+^0)\lebesgue(\tilde{\Lambda}_+\cap\Lambda_+^0)}{\lebesgue(\tilde{\Lambda}_+)}+\frac{\lebesgue(\Lambda_+^0\setminus\tilde{\Lambda}_+)\lebesgue(\tilde{\Lambda}_-\cap \Lambda_-^0)}{\lebesgue(\tilde{\Lambda}_-)}\Big)(1+ o_{\PP}(1))\\
&\quad + \mathcal{O}_{\PP}\Big(\delta^{d+1}\log(\delta^{-1})\Big).
\end{align*}
Since 
\[(\tilde{\theta}_-,\tilde{\theta}_+,\tilde{\Lambda}_+) \in \argmin_{(\theta_-,\theta_+,\Lambda_+) \in \R^2 \times \mathcal{A}_+^\prime} Z_{\delta}(\theta_-,\theta_+,\Lambda_+),\]
and, again by Lemma \ref{lem:order_frac},
\begin{align*}
0 > \delta^{d+2}Z_\delta(\theta^{\overline{\Lambda^0_+}}_-,\theta^{\overline{\Lambda^0_+}}_+,\Lambda^0_+) &= -\delta^{d+2}\frac{(\sum_{\Sq(\alpha) \subset \Lambda_+^0}M_{\delta,\alpha})^2}{2\sum_{\Sq(\alpha)^\circ \subset\Lambda_+^0}I_{\delta,\alpha}}-\delta^{d+2}\frac{(\sum_{\Sq(\alpha)^\circ \subset \Lambda_-^0}M_{\delta,\alpha})^2}{2\sum_{\Sq(\alpha)^\circ \subset \Lambda_-^0}I_{\delta,\alpha}}\\
&=\mathcal{O}_{\PP}\Big(\delta^{d+1}\log(\delta^{-1})\Big),
\end{align*}
where $\Lambda^0_+ \in \mathcal{A}_+^\prime$ for $\delta$ small enough by Assumption \ref{ass:bounds_domain},  we conclude that 
\begin{equation}\label{eq:rate1}
\frac{\lebesgue(\tilde{\Lambda}_+\setminus\Lambda_+^0)\lebesgue(\tilde{\Lambda}_+\cap\Lambda_+^0)}{\lebesgue(\tilde{\Lambda}_+)}+\frac{\lebesgue(\Lambda_+^0\setminus\tilde{\Lambda}_+)\lebesgue(\tilde{\Lambda}_-\cap \Lambda_-^0)}{\lebesgue(\tilde{\Lambda}_-)}  = \mathcal{O}_{\PP}\Big(\delta^{d+1}\log(\delta^{-1})\Big).
\end{equation}
Some simple calculations show that 
\begin{equation}\label{eq:lower}
\tilde{\Lambda}_+ \notin \{\overline{\Lambda^0_+}, \Lambda^0_-\} \implies \frac{\lebesgue(\tilde{\Lambda}_+\setminus\Lambda_+^0)\lebesgue(\tilde{\Lambda}_+\cap\Lambda_+^0)}{\lebesgue(\tilde{\Lambda}_+)}+\frac{\lebesgue(\Lambda_+^0\setminus\tilde{\Lambda}_+)\lebesgue(\tilde{\Lambda}_-\cap \Lambda_-^0)}{\lebesgue(\tilde{\Lambda}_-)} \geq \delta^d \kappa/2,
\end{equation}
for $\delta$ small enough, see Section \ref{app:proofs} for the full argument. Thus, using \eqref{eq:rate1}, it follows that 
\begin{align*}
1 &= \lim_{\delta \to 0} \PP\left(\frac{\lebesgue(\tilde{\Lambda}_+\setminus\Lambda_+^0)\lebesgue(\tilde{\Lambda}_+\cap\Lambda_+^0)}{\lebesgue(\tilde{\Lambda}_+)}+\frac{\lebesgue(\Lambda_+^0\setminus\tilde{\Lambda}_+)\lebesgue(\tilde{\Lambda}_-\cap \Lambda_-^0)}{\lebesgue(\tilde{\Lambda}_-)} \leq \delta^{d+1} (\log(\delta^{-1}))^2 \right)\\
&= \lim_{\delta \to 0} \PP\Big(\tilde{\Lambda}_+ \in \big\{ \overline{\Lambda^0_+},  \Lambda^0_- \big\}\Big). 
\end{align*}
Clearly, this is equivalent to 
\[\lim_{\delta \to 0} \PP\big(d_H(\partial \Lambda^0_+, \partial \tilde{\Lambda}_+) = 0\big) = 1.\]
Moreover, if $\Lambda^0_- \notin \mathcal{A}_+$, it follows from $\tilde{\Lambda}_+, \overline{\Lambda^0_+} \in \mathcal{A}_+$ that 
\[\PP\Big(\tilde{\Lambda}_+ = \overline{\Lambda^0_+}\Big) = \PP\Big(\tilde{\Lambda}_+ \in \big\{ \overline{ \Lambda^0_+},  \Lambda^0_- \big\}\Big) \underset{\delta \to 0}{\longrightarrow} 1. \]
Finally, let us treat the estimator $\hat{\Lambda}{}^\ast_+$. Note first that 
\begin{align*} 
\lebesgue(\Lambda^0_\pm \cap \hat{\Lambda}_\pm) \geq \lebesgue(\Lambda^0_\pm \cap \hat{\Lambda}_\mp) &\iff \lebesgue(\Lambda^0_\pm \cap \hat{\Lambda}_\mp) \leq \lebesgue(\Lambda^0_\pm)/2.
\end{align*}
Thus, for $\delta$ small enough,
\begin{align*} 
\PP\big(\hat{\Lambda}{}^\ast_+ = \overline{\Lambda^0_+}\big) &= \PP\big(\tilde{\Lambda}_+ = \overline{\Lambda^0_+}, \lebesgue(\tilde{\Lambda}_+ \cap \hat{\Lambda}_+) \geq \lebesgue(\tilde{\Lambda}_+ \cap \hat{\Lambda}_-)\big)\\ 
&\quad + \PP\Big(\overline{[0,1]^d \setminus \tilde{\Lambda}_+} = \overline{\Lambda^0_+}, \lebesgue(\tilde{\Lambda}_+ \cap \hat{\Lambda}_-) > \lebesgue(\tilde{\Lambda}_+ \cap \hat{\Lambda}_+)\Big)\\ 
&= \PP\big(\tilde{\Lambda}_+ = \overline{\Lambda^0_+}, \lebesgue(\Lambda^0_+ \cap \hat{\Lambda}_+) \geq \lebesgue(\Lambda^0_+ \cap \hat{\Lambda}_-)\big)\\ 
&\quad + \PP\big(\tilde{\Lambda}_+ = \Lambda^0_-, \lebesgue(\Lambda^0_- \cap \hat{\Lambda}_-) > \lebesgue(\Lambda^0_- \cap \hat{\Lambda}_+)\big)\\  &\geq \PP\big(\tilde{\Lambda}_+ = \overline{\Lambda^0_+},  \lebesgue(\Lambda^0_+ \cap \hat{\Lambda}_-) \leq \lebesgue(\Lambda^0_+)/2\big)\\ 
&\quad + \PP\big(\tilde{\Lambda}_+ = \Lambda^0_-,  \lebesgue(\Lambda^0_- \cap \hat{\Lambda}_+) < \lebesgue(\Lambda^0_-)/2\big)\\
&\geq \PP\big(\tilde{\Lambda}_+ = \overline{\Lambda^0_+},  \lebesgue(\Lambda^0_+ \cap \hat{\Lambda}_-) \leq \kappa/2\big)\\ 
&\quad + \PP\big(\tilde{\Lambda}_+ = \Lambda^0_-,  \lebesgue(\Lambda^0_- \cap \hat{\Lambda}_+) \leq \kappa/2\big)\\
&\geq \PP\big(\tilde{\Lambda}_+ = \overline{\Lambda^0_+},  \lebesgue(\Lambda^0_+ \vartriangle \hat{\Lambda}_+) \leq \kappa/2\big)\\ 
&\quad + \PP\big(\tilde{\Lambda}_+ = \Lambda^0_-,  \lebesgue(\Lambda^0_+ \vartriangle \hat{\Lambda}_+) \leq \kappa/2\big)\\
&= \PP\big(\tilde{\Lambda}_+ \in \{\overline{\Lambda^0_+}, \Lambda^0_-\},  \lebesgue(\Lambda^0_+ \vartriangle \hat{\Lambda}_+) \leq \kappa/2\big).
\end{align*}
By Theorem \ref{theo:rate_cp} we have $\lebesgue(\hat{\Lambda}_+ \vartriangle \Lambda^0_+) \overset{\PP}{\longrightarrow} 0$ and from above we know that $\PP(\tilde{\Lambda}_+ \in \{\overline{\Lambda^0_+}, \Lambda^0_-\}) \longrightarrow 1$. Thus, the probability in the last line converges to $1$, which yields 
\[\lim_{\delta \to 0} \PP\big(\hat{\Lambda}{}^\ast_+ = \overline{\Lambda^0_+}\big) = 1. \]
\end{proof}

Using Theorem \ref{theo:tiling} we can prove Theorem \ref{theo:clt}.
\begin{proof}[Proof of Theorem \ref{theo:clt}]
We only prove the statement for $\hat{\theta}_\pm^\ast$, the proof for $\tilde{\theta}_\pm$ given $\Lambda^0_- \notin \mathcal{A}_+$ is identical. From \eqref{eq: theta_augm} we obtain the decomposition 
\[\theta_\pm^{\Lambda^0_+} = \theta^0_\pm + \frac{\sum_{\Sq(\alpha)^\circ \subset \Lambda^0_\pm} M_{\delta,\alpha}}{\sum_{\Sq(\alpha)^\circ \subset \Lambda^0_\pm} I_{\delta,\alpha}},\]
and from Proposition \ref{lem: orderFisher} it follows that
\[\delta^{d+2}\E\Big[\sum_{\Sq(\alpha)^\circ \subset \Lambda^0_\pm} I_{\delta,\alpha}\Big] = \frac{T}{2\theta^0_\pm}\lVert \nabla K\rVert^2_{L^2} \lebesgue(\Lambda^0_\pm) + \mathcal{O}(\delta^2) \underset{\delta \to 0}{\longrightarrow} \frac{T}{2\theta^0_\pm}\lVert \nabla K\rVert^2_{L^2} v^0_\pm.\]
Moreover, the variance bound in Proposition \ref{lem: orderFisher} yields that 
\[\frac{\sum_{\Sq(\alpha)^\circ \subset  \Lambda^0_+}I_{\delta,\alpha}}{\E[\sum_{\Sq(\alpha)^\circ \subset \Lambda^0_+} I_{\delta,\alpha}]} \overset{\PP}{\longrightarrow} 1.\]
Combining these, a standard martingale CLT therefore gives  
\[\delta^{-(d+2)/2}(\theta^{\Lambda^0_+}_+ - \theta^0_+) \overset{d}{\longrightarrow} \mathcal{N}\left(0, \frac{2\theta^0_+}{T\lVert \nabla K\rVert^2_{L^2}v^0_+} \right).\]
Using also that by Theorem \ref{theo:tiling} we have 
\[\one_{\{\hat{\Lambda}{}^\ast_+ = \overline{\Lambda^0_+}\}} \overset{\PP}{\longrightarrow} 1,\]
it therefore follows from Slutsky's theorem that
\[\delta^{-(d/2+1)}(\hat{\theta}^{\ast}_\pm - \theta^0_\pm)\one_{\{\hat{\Lambda}^\ast_+ = \overline{\Lambda^0_+}\}} = \delta^{-(d/2+1)}(\theta_\pm^{\Lambda^0_+} - \theta^0_\pm) \one_{\{\hat{\Lambda}^\ast_+ = \overline{\Lambda^0_+}\}}  \overset{d}{\longrightarrow} \mathcal{N}\left(0, \frac{2\theta^0_\pm}{T\lVert \nabla K\rVert^2_{L^2}v^0_\pm} \right).\]
Since, again by Theorem \ref{theo:tiling},
\begin{align*}
    \PP\Big(\{\lvert\delta^{-d/2-1}(\hat{\theta}^\ast_\pm-\theta_\pm^0)\one_{\{\hat{\Lambda}_+^\ast \neq \overline{\Lambda_+^0}\}} \rvert >\varepsilon\}\Big)&=\PP\Big(\{\lvert \delta^{-d/2-1}(\hat{\theta}^\ast_\pm-\theta_\pm^0)\rvert>\varepsilon\}\cap \{\hat{\Lambda}_+^\ast\neq \overline{\Lambda_+^0}\}\Big)\\
    &\leq \PP\Big(\hat{\Lambda}_+^\ast \neq \overline{\Lambda_+^0}\Big)\underset{\delta \to 0}{\longrightarrow} 0
\end{align*}
and therefore
\[\delta^{-d/2-1}(\hat{\theta}^\ast_\pm-\theta_\pm^0)\one_{\{\hat{\Lambda}_+^\ast \neq \overline{\Lambda_+^0}\}} \overset{\PP}{\longrightarrow} 0,\]
we conclude that indeed 
\[\delta^{-(d/2+1)}(\hat{\theta}^{\ast}_\pm - \theta^0_\pm)  \overset{d}{\longrightarrow} \mathcal{N}\left(0, \frac{2\theta^0_\pm}{T\lVert \nabla K\rVert^2_{L^2}v^0_\pm} \right).\]
\end{proof}

\section{Results for specific models}
In this section we give explicit constructions of the candidate sets $\mathcal{A}_+$ and estimator convergence rates for two specific shape restrictions. 
\label{sec: main}
\subsection{Estimation of change interfaces with graph representation}\label{sec: graph}
Consider model A from Section \ref{sec: setup}, for which $\Lambda_+^0$ is fully determined by the continuous change interface \\$\tau^0\colon [0,1]^{d-1} \to [0,1]$. For any function  $\tau \colon [0,1]^{d-1} \to [0,1]$ let us define the open epigraph 
\[\operatorname{epi} \tau \coloneqq \{(x,y) \in [0,1]^{d-1} \times [0,1]: \tau(x) > y\}.\]
For $\gamma \in [n]^{d-1}$ let $\tilde{\Sq}_{d-1}(\gamma)$ be a $(d-1)$-dimensional hypercube with edge length $\delta$ that is centered at the point $z_\gamma \coloneqq \delta(\gamma - \tfrac{1}{2}\one)$. The hypercubes $\tilde{\Sq}_{d-1}(\gamma)$ are chosen such that $(\tilde{\Sq}_{d-1}(\gamma))_{\gamma \in [n]^{d-1}}$ forms a partition of $[0,1]^{d-1}$, i.e., $[0,1]^{d-1} = \bigcup_{\gamma \in [n]^{d-1}} \tilde{\Sq}_{d-1}(\gamma)$. Let us also denote $\Sq_{d-1}(\gamma) \coloneqq \overline{\tilde{\Sq}_{d-1}(\gamma)}$. With 
\[\mathcal{G} \coloneqq \{i\delta: i=0,1,\ldots,n\}^{n^{d-1}},\]
we now define the grid functions 
\[\tau_{\zeta}(x) = \sum_{\gamma \in [n]^{d-1}} \zeta_\gamma \bm{1}_{\tilde{\Sq}_{d-1}(\gamma)}(x), \quad x \in [0,1]^{d-1}, \quad \zeta = (\zeta_\gamma)_{\gamma \in [n]^{d-1}} \in \mathcal{G},\]
and set 
\[\Lambda_+(\zeta) \coloneqq (\operatorname{epi} \tau_\zeta)^+ \in \mathcal{P}, \quad \zeta \in \mathcal{G}.\]
In other words, $\Lambda_+(\zeta)$ can be written as 
\[\Lambda_+(\zeta)=\bigcup_{\gamma\in [n]^{d-1}:\zeta_\gamma<1} \Sq_{d-1}(\gamma)\times[\zeta_\gamma,1].\]
We then choose our candidate tiling sets as 
\[\mathcal{A}_+ = \big\{\Lambda_+(\zeta) : \zeta \in \mathcal{G}\big\}.\]
Note that the size $\lvert \mathcal{A}_+ \rvert = n^{d-1}(n+1)$ of this family of candidate sets if significantly smaller than that of the uninformed choice $\mathcal{A}_+ = \mathcal{P}$, which is $2^{n^d}$. 
To see that $\mathcal{A}_+$ is a valid choice, note that for 
\[\zeta^\uparrow_\gamma \coloneqq \delta \lceil\delta^{-1}\sup\{\tau^0(x): x \in \tilde{\Sq}_{d-1}(\gamma)\}\rceil, \quad \zeta^\downarrow_\gamma \coloneqq \delta \lfloor\delta^{-1}\inf\{\tau^0(x): x \in \tilde{\Sq}_{d-1}(\gamma)\}\rfloor,\]
it holds that $\tau_{\zeta^\downarrow}$ is the maximal grid function dominated by $\tau^0$ and therefore 
\[\Lambda^\updownarrow_+ = (\operatorname{epi} \tau^0)^+ = (\operatorname{epi} \tau_{\zeta^\downarrow})^+ = \Lambda_+(\zeta^\downarrow),\]
whence in particular $\Lambda^\updownarrow_+ \in \mathcal{A}_+$ as required. Furthermore, the function
\[\varphi\colon \mathcal{A}_+ \to \mathcal{G}, \quad \Lambda_+(\zeta) \to \zeta,\]
is a bijection. Thus, for the estimator $\hat{\Lambda}_+$ from the previous section, if we set 
\[\hat{\zeta} \coloneqq \varphi(\hat{\Lambda}_+),\]
it follows that 
\[\hat{\Lambda}_+ = \Lambda_+(\hat{\zeta}) = (\operatorname{epi} \tau_{\hat{\zeta}})^+.\]
Consequently, if we define the change interface estimator 
\[\hat{\tau} = \tau_{\hat{\zeta}},\]
then we have the identity
\[\lVert \hat{\tau} - \tau^0 \rVert_{L^1([0,1]^{d-1})} = \lebesgue(\hat{\Lambda}_+ \vartriangle \Lambda^0_+).\]
An example of a change interface and its grid function approximation can be seen in Figure \ref{fig: Changepoint2D}.
\begin{figure}[t!]
\centering
		\includegraphics[width=0.6\linewidth]{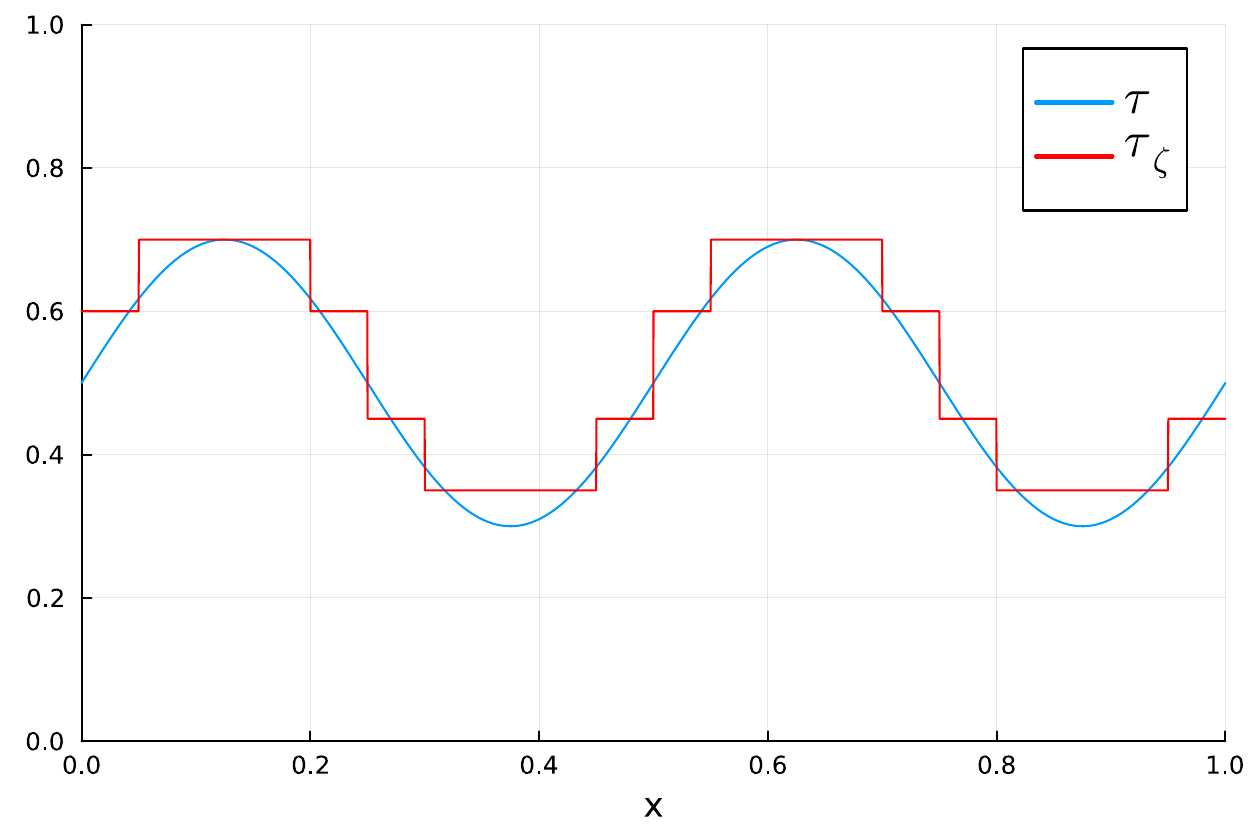}
	\caption{Example of a change interface $\tau$ and a piecewise constant approximation $\tau_\zeta$ in dimension $d=2$ with $\delta=0.05$.}
	\label{fig: Changepoint2D}
\end{figure}
Under Hölder smoothness assumptions on the change  interface $\tau^0$, Theorem \ref{theo:rate_cp} allows us to obtain a convergence rate for the change domain estimator $\hat{\Lambda}_+$ in the symmetric difference pseudometric, or equivalently, for the change interface estimator $\hat{\tau}$ in the $L^1$-metric. For $\beta \in (0,1]$, $L > 0$, let $\mathcal{H}(\beta,L)$ be the $(\beta,L)$-Hölder class on $[0,1]^{d-1}$ w.r.t.\ the maximum metric, i.e., 
\[\mathcal{H}(\beta,L) \coloneqq \big\{f\colon [0,1]^{d-1}\to \R: \lvert f(x) - f(y) \rvert \leq L \lVert x-y \rVert_\infty^\beta \text{ for all } x,y\in [0,1]^{d-1} \big\}.\]
For ease of exposition we implicitly assume in the next statement that $\tau^0$ does not vary with $\delta$.
\begin{proposition}\label{prop:interface}
If $\tau^0 \in \mathcal{H}(\beta,L)$ for $\beta \in (0,1]$ and $L > 0$, then there exists a constant $C$ depending only on $\beta,L,d,\underline{\theta},\overline{\theta},T$ and $\underline{\eta}$ such that 
\[\E\big[\lVert \hat{\tau} - \tau^0 \rVert_{L^1([0,1]^{d-1})}\big] = \E\big[\lebesgue(\hat{\Lambda}_+ \vartriangle \Lambda^0_+)\big] \leq C\delta^\beta.\]
Moreover, if $\tau^0$ is not identically $1$ (resp.\ not identically $0$), then $\lvert \hat{\theta}_+ - \theta^0_+ \rvert = \mathcal{O}_{\PP}(\delta^{\beta/2})$\, (resp.\ $\lvert \hat{\theta}_- - \theta^0_- \rvert = \mathcal{O}_{\PP}(\delta^{\beta/2})$). In particular, if $\tau^0(x) \in (0,1)$ for some $x \in [0,1]^{d-1}$, it holds that $\lVert \hat{\theta} - \theta^0 \rVert_{L^1((0,1)^d)} \in \mathcal{O}_{\PP}(\delta^{\beta/2}).$
\end{proposition}
\begin{proof} 
By Theorem \ref{theo:rate_cp} and Corollary \ref{cor: rate}, it suffices to show that $\lvert \mathcal{B} \rvert \leq L\delta^{d-\beta}$.  To this end, we first observe that in the change interface model it holds that 
\begin{align*}
\mathcal{B} &= \big\{\alpha \in [n]^d: \Sq(\alpha) \cap \partial \Lambda^0_+ \neq \varnothing\big\}\\
&= \big\{(\gamma,j) \in [n]^{d-1} \times [n]: (\Sq_{d-1}(\gamma)^\circ \times ((j-1)\delta,j\delta)) \cap \tau^0(\Sq_{d-1}(\gamma)) \neq \varnothing\big\} \\
&\subset \big\{(\gamma,j) \in [n]^{d-1} \times [n]: j \in \delta^{-1}(\zeta^\downarrow_\gamma,\zeta^\uparrow_\gamma]\big\}.
\end{align*}
Because $\tau^0 \in \mathcal{H}(\beta,L)$ implies  that $\lvert \zeta^\uparrow_\gamma - \zeta^\downarrow_\gamma \rvert \leq L \delta^\beta$ for any $\gamma \in [n]^{d-1}$, we obtain \[\lvert \{j \in \delta^{-1}(\zeta^\downarrow_\gamma, \zeta^\uparrow_\gamma] \cap[n]\}\rvert \leq L\delta^{\beta - 1}.\]
Thus, from above, 
\[\lvert \mathcal{B} \rvert \leq \delta^{-(d-1)} L \delta^{\beta -1} = L\delta^{-d+\beta},\]
as desired.
\end{proof}
\begin{remark}\label{rem:opti}
The domain estimation rate $\delta^\beta$ translates to $N^{-\beta/d}$ in terms of the  number of observations $N=\delta^{-d}$. As pointed out in \cite[Chapter 3-5]{tsy93}, for appropriately designed random measurement locations, the minimax rate for estimating a boundary fragment in image reconstruction is given by $N^{-\beta/(\beta+d-1)}$ for arbitrary $\beta>0$. Unless we have Lipschitz regularity $\beta =1$, this rate is, however, not achievable for an equidistant deterministic design, which is usually referred to as \textit{regular design}. In fact, it can be shown that with regular  design, the rate $N^{-\beta/d}$ is the optimal rate for $\beta \in (0,1]$ in the edge estimation problem by adapting \cite[Theorem 3.3.1]{tsy93}. Indeed, consider the regular  design image reconstruction problem 
\[Y_{\alpha} = \theta(x_{\alpha}) + \varepsilon_{\alpha}, \quad \alpha \in [n]^d,\]
where 
\[\theta(x) = \theta_- \one_{\Lambda_-}(x) + \theta_+ \one_{\Lambda_+}(x), \quad x \in [0,1]^d,\]
for known values $\theta_- \neq \theta_+$ and $\Lambda_- \uplus \Lambda_+ = [0,1]^d$ and denote by $\PP_{\Lambda_+}$ the law generated by the observations for fixed $\Lambda_+$.  Let $\tau^0 \equiv 0$ and further let
\[\tau^1(x) \coloneqq \sum_{\gamma\in [n]^{d-1}}\norm{x-z_\gamma}_\infty^\beta\one_{\{\norm{x-z_\gamma}_\infty\leq\delta/2\}}, \quad x \in [0,1]^{d-1}.\]
Then, for $\Lambda_+^i \coloneqq \{(x,y) \in [0,1]^d: \tau^i(x) > y\},$ $i=0,1$, we have 
\[\{x_{\alpha} : \alpha \in [n]^d\} \subset \Lambda_+^0 \cap \Lambda_+^1 = \Lambda_+^1,\] 
and therefore $\PP_{\Lambda_+^1} = \PP_{\Lambda_+^0}$. Since furthermore $\lebesgue(\Lambda_+^0 \vartriangle \Lambda_+^1) = \int \tau^1(x) \diff{x} \asymp \delta^\beta$ and $\tau_i \in \mathcal{H}(\beta,1)$ for $i=0,1$,  we obtain the minimax lower bound 
\[\inf_{\hat{\Lambda}_+}\sup_{\Lambda_+ \in \Xi(\beta,L)}  \E_{\Lambda_+}\big[\lebesgue(\hat{\Lambda}_+ \vartriangle \Lambda_+)\big] \gtrsim \delta^\beta = N^{-\beta/d},\]
for the class $\Xi(\beta,L)$ of epigraphs of functions $\tau \in \mathcal{H}(\beta,1)$.
\end{remark}

\begin{remark} 
Note that the underlying assumption on the graph orientation, that is $\Lambda^0_+$ is an epigraph w.r.t.\ the $d$-th coordinate, can be easily circumvented by adapting the candidate set $\mathcal{A}_+$ to represent any graph structure that may only become visible after rotation of the cube $[0,1]^d$. Since a $d$-dimensional hypercube admits $2d$ faces, the size of the adapted candidate set would increase from $n^{d-1}(n+1)$ to $2dn^{d-1}(n+1)$.
\end{remark}
If $\tau^0$ is a grid function, the polynomial growth of the candidate set $\mathcal{A}_+$ implies a much stronger result by Theorem \ref{theo:tiling}. Analogously to before we set $\tilde{\zeta} = \varphi(\tilde{\Lambda}_+)$ and $\tilde{\tau} = \tau_{\tilde{\zeta}}$ for the estimator $\tilde{\Lambda}_+$ from Section \ref{sec:tiling}.
\begin{corollary} 
Suppose that $\tau^0(\delta) = \tau_{\zeta^0(\delta)}$ for  $\zeta^0(\delta) \in \{ \zeta \in \mathcal{G}(\delta): \lVert \tau_{\zeta(\delta)} \rVert_{L^1([0,1]^{d-1}} \in [\kappa,1-\kappa]\}$ for some $\kappa \in (0,1/2)$. Then, 
\[\lim_{\delta \to 0} \PP\Big(\lVert \tilde{\tau} - \tau^0 \rVert_{L^1([0,1]^{d-1})} = 0\Big) = 1.\]
If we also have the pointwise convergence $\tau^0(\delta) \underset{\delta \to 0}{\longrightarrow} \tau^0$ for some right-or left-continuous $\tau^0$ such that neither $\tau^0 \equiv 1$, nor $\tau^0 \equiv 0$, then
\[\delta^{-(d/2+1)}\big(\tilde{\theta}_\pm - \theta^0_\pm) \overset{d}{\longrightarrow} \mathcal{N}\left(0, \frac{2\theta^0_\pm}{T\lVert \nabla K\rVert^2_{L^2}v^0_\pm} \right),\]
where $\nu^0_- = \lVert \tau^0 \rVert_{L^1([0,1]^{d-1})}$ and $\nu^0_+ = 1 - \nu^0_-$.
\end{corollary}
\begin{proof} 
It suffices to note that since $\Lambda^0_-$ is the subgraph of $\tau^0(\delta)$ and the latter is not identical $0$ or $1$, $\Lambda^0_-$ cannot be an epigraph, that is, $\Lambda^0_- \notin \mathcal{A}_+$. The statement is now a direct consequence of Theorem \ref{theo:tiling} and Theorem \ref{theo:clt}.
\end{proof}

\subsection{Estimation of convex sets}
Suppose that model B from Section \ref{sec: setup} holds, i.e., $\Lambda_+^0 \subset \overline{\Lambda}$ is convex. A simple choice for $\mathcal{A}_+$ is given by $\mathcal{A}_+ = \{C^+: C\subset [0,1]^d \text{ convex}\}$. This choice of candidate sets is however not particularly constructive. Let us therefore propose another family of candidate sets, whose construction follows a similar principle as the one for the graph representation model from the previous subsection.

The basic observation is that by convexity, for any $x \in\operatorname{int}\Lambda_+^0$, the vertical ray $y \mapsto x + ye_d$ intersects  $\partial \Lambda_+$ in exactly two points. The natural idea is therefore to build candidate sets from hypercuboids $\Sq_{d-1}(\gamma) \times [\underline{\zeta}{}_\gamma, \overline{\zeta}_\gamma]$ for $\underline{\zeta}_\gamma,\overline{\zeta}_\gamma$ living on the grid $\mathcal{G}$. Heuristically speaking, as the upper and lower intersection points of the vertical ray can be described by a concave and convex function, respectively, we aim to approximate those by piecewise constant functions on $\tilde{\operatorname{Sq}}_{d-1}(\gamma)$ in analogy to the graph representation of Section \ref{sec: graph}. In similarity to Section \ref{sec: graph}, let
\begin{align*} 
\zeta^\uparrow_\gamma &\coloneqq \delta \lceil\delta^{-1}\sup\{x_d: x \in \partial\Lambda_+^0\cap(\Sq_{d-1}(\gamma)^\circ\times[0,1])\}\rceil, \quad \gamma \in [n]^{d-1},\\ 
\zeta^\downarrow_\gamma &\coloneqq \delta \lfloor\delta^{-1}\inf\{x_d: x \in \partial\Lambda_+^0\cap(\Sq_{d-1}(\gamma)^\circ\times[0,1])\}\rfloor, \quad \gamma \in [n]^{d-1},
\end{align*}
be the grid projections of upper and lower limits of the intersection of the convex set $\Lambda^+_0$ with the strip $\Sq_{d-1}(\gamma)^\circ \times [0,1]$. Here, the supremum and the infimum of the empty is set to $0$. 
Consider candidate sets
\[\Lambda_+(\zeta) \coloneqq \bigcup_{\gamma \in [n]^{d-1}:\underline{\zeta}_\gamma<\overline{\zeta}_\gamma} \Sq_{d-1}(\gamma) \times [\underline{\zeta}_\gamma,\overline{\zeta}_\gamma], \quad \underline{\zeta}_\gamma, \overline{\zeta}_\gamma \in \mathcal{G},\]
and the minimal tiling
\[\Lambda^\updownarrow_+ \coloneqq \bigcup_{\gamma \in [n]^{d-1}:\zeta_\gamma^\downarrow<\zeta_\gamma^\uparrow} \Sq_{d-1}(\gamma) \times [\zeta^\downarrow_\gamma, \zeta^\uparrow_\gamma]\]
both belonging to 
\begin{equation}\label{eq: representation set}
\mathcal{A}_+ = \Big\{{\Lambda}_+(\zeta): \underline{\zeta}_\gamma, \overline{\zeta}_\gamma \in \mathcal{G}, \gamma \in [n]^{d-1} \Big\}.
\end{equation}
Then, it holds that $\lvert \mathcal{A}_+\rvert = n^{d-1}((n+n^2)/2+1) \asymp n^{d+1}$.
An exemplary illustration in $d=2$ is given in Figure \ref{fig:change_convex2d}.
\begin{figure}[t!]
\centering
\begin{subfigure}{.49\textwidth}
		\centering\includegraphics[width=1.\linewidth]{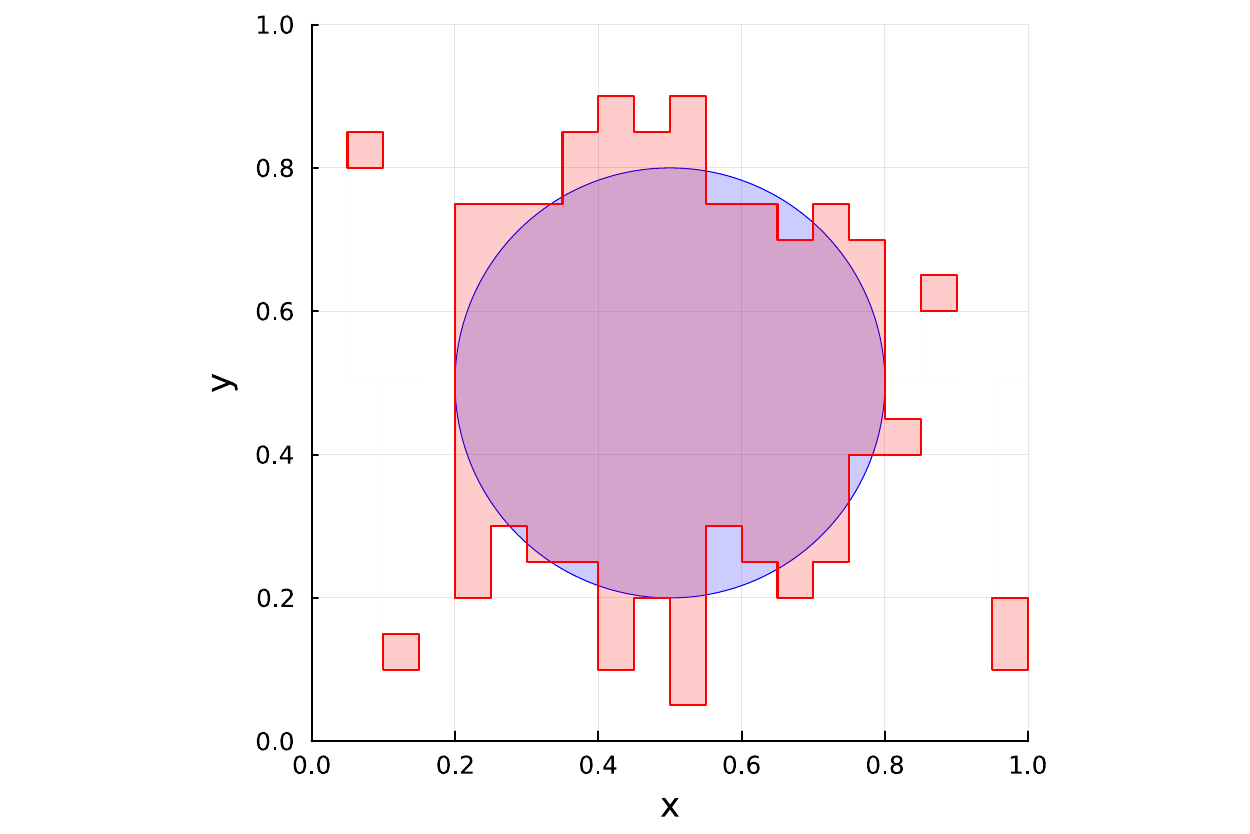}
	\end{subfigure}
\begin{subfigure}{.49\textwidth}
		\centering\includegraphics[width=1.\linewidth]{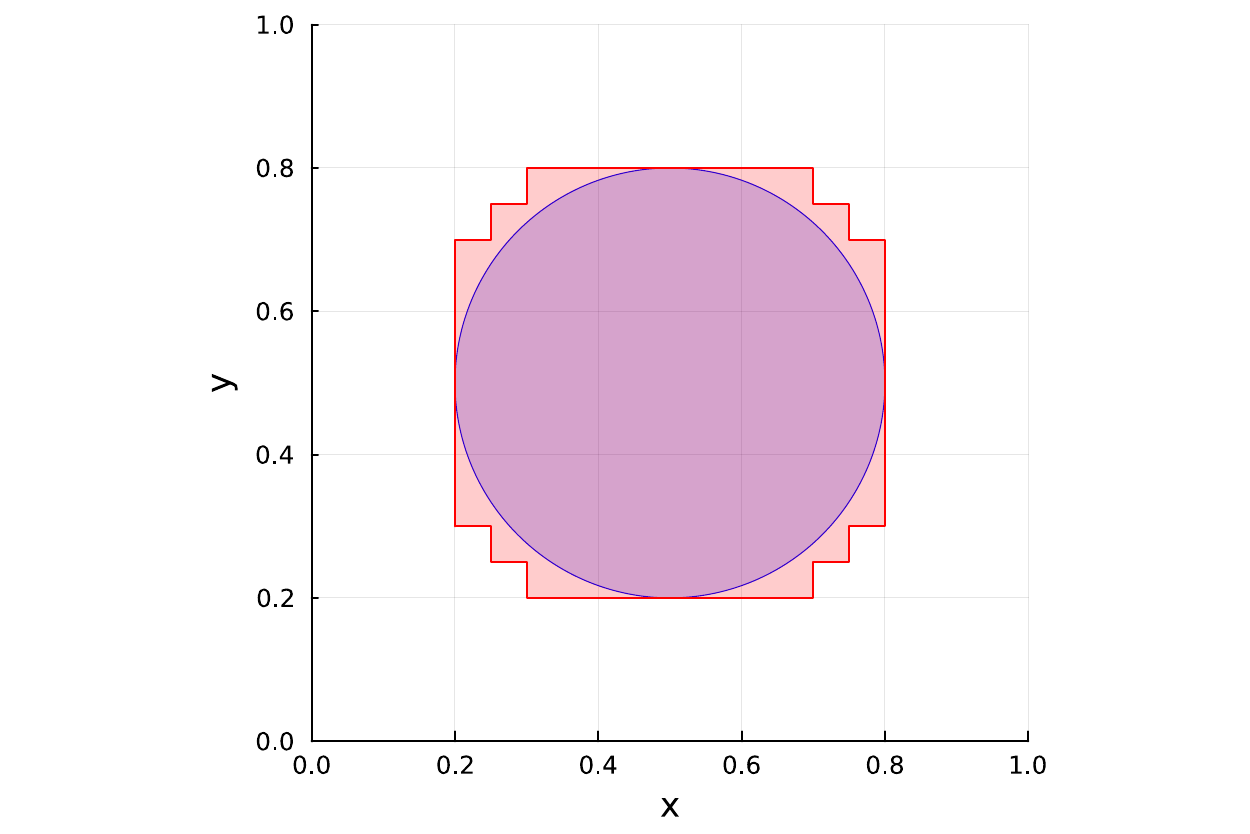}
	\end{subfigure}
\caption{Approximation (red) of $\Lambda_+^0$ (blue) for $\delta=0.05.$; left: by some ${\Lambda}_+(\zeta)$ from \eqref{eq: representation set}; right: by $\Lambda_+^\updownarrow$.}
\label{fig:change_convex2d}
\end{figure}

In order to apply Theorem \ref{theo:rate_cp}, it remains to control the size of the boundary tiling indices $\mathcal{B}$, which can be done with a classical result from convex geometry. By \cite[Corollary 2]{lassak_covering_1988}, the boundary of any convex set $C \subset [0,1]^d$ can be covered by at most 
\[n^d -(n-2)^d < 2dn^{d-1} = 2d \delta^{-d+1}\]
hypercubes from the tiling $\{\Sq(\alpha): \alpha \in [n]^d \}$ of $[0,1]^d$. This entails the bound 
\[\lvert \mathcal{B} \rvert =\lvert \{\alpha\in[n]^d:\operatorname{Sq}(\alpha)^\circ\cap \partial\Lambda_+^0\neq\varnothing\} \rvert < 2d\delta^{-d+1}.\]
Consequently, Theorem \ref{theo:rate_cp} and Corollary \ref{cor: rate} yield the following convergence result. 
\begin{proposition}\label{prop:conv}
Suppose that $\Lambda^0_+$ is convex and $\mathcal{A}_+$ be given by \eqref{eq: representation set}. Then, for some absolute constant $C$ depending only on $d,\underline{\theta},\overline{\theta},T$ and $\underline{\eta}$ it holds that
\[\E\big[\lebesgue(\hat{\Lambda}_+ \vartriangle \Lambda^0_+) \big] \leq {C}\delta.\] 
Moreover, if $\lebesgue(\Lambda^0_\pm) \neq 0$, it holds that $\lvert \hat{\theta}_\pm -\theta^0_\pm\rvert \in \mathcal{O}_{\PP}(\delta^{1/2})$. In particular, if $\lebesgue(\Lambda^0_+) \notin \{0,1\}$, we have $\lVert \hat{\theta} - \theta^0 \rVert_{L^1((0,1)^d}) \in \mathcal{O}_{\PP}(\delta^{1/2})$.
\end{proposition}
\begin{remark} 
\label{rmk: opt_convex}
For an indication of optimality of the domain convergence rate, let us again consider the regular design image reconstruction problem 
\[Y_\alpha = \theta(x_{\alpha}) + \varepsilon_{\alpha}, \quad \alpha \in [n]^d,\]
from Remark \ref{rem:opti}. Let $\Lambda_+^0 \subset (0,1)^d$ be an open hypercube with volume $v_0 \in (0,1)$ and edge length $l_0 = v_0^{1/d}$ such that the corners of the hypercube lie on $\{x_\alpha\}_{\alpha \in [n]^{d}}$ and let $\Lambda^1_+$ be the open hypercube containing $\Lambda^0_+$  such that $d_\infty(\Lambda_+^0,\Lambda_+^1) = \delta/2$, where $\delta$ is chosen small enough s.t.\ $\Lambda_+ \subset (0,1)^d$. Then $x_{\alpha} \in \Lambda_+^0$ if and only if $x_{\alpha} \in \Lambda_+^1$ and therefore $\PP_{\Lambda_+^0} = \PP_{\Lambda_+^1}$. Moreover, 
\[\lebesgue(\Lambda_+^0 \vartriangle \Lambda^1_+) = (l_0 + \delta/2)^d - l_0^d \geq \frac{d}{2}l_0^{d-1}\delta =  \delta \frac{d}{2}v_0^{\frac{d-1}{d}}, \]
yielding the minimax lower bound
\[ \inf_{\hat{\Lambda}_+}\sup_{\Lambda_+ \in \mathcal{C}(v_0)}  \E_{\Lambda_+}\big[\lebesgue(\hat{\Lambda}_+ \vartriangle \Lambda_+)\big] \gtrsim C(v_0)\delta,\]
 for the class $\mathcal{C}(v_0)$ of convex sets in $[0,1]^d$ with volume at least $v_0 \in (0,1)$.
\end{remark}
The constructed estimator $\hat{\Lambda}_+$ has the drawback of generally not being a convex or even connected set. However, we can easily transform our estimator $\hat{\Lambda}_+$ into a convex estimator $\hat{\Lambda}{}^{\mathrm{con}}_+$ that converges at the same rate. To this end we employ a minimum distance fit by choosing an estimator $\hat{\Lambda}{}^{\mathrm{con}}_+$ such that
\[\lebesgue\big(\hat{\Lambda}{}^{\mathrm{con}}_+ \vartriangle \hat{\Lambda}_+\big) \leq \inf_{C \in \mathcal{K}_d} \lebesgue\big(C \vartriangle \hat{\Lambda}_+ \big) + \delta,\]
where 
\[\mathcal{K}_d \coloneqq \big\{C \subset[0,1]^d : C \text{ is convex and closed}\big\}.\]
Up to a $\delta$-margin, $\hat{\Lambda}^{\mathrm{con}}_+$ is therefore  a convex set with maximum volume overlap with $\hat{\Lambda}_+$. 
\begin{corollary}\label{coro:conv}
The convex estimator $\hat{\Lambda}{}^{\mathrm{con}}_+$ satisfies  
\[\E\big[\lebesgue(\hat{\Lambda}{}^{\mathrm{con}}_+ \vartriangle \Lambda^0_+) \big] \leq (2C+1)\delta,\]
where $C$ is the constant from Proposition \ref{prop:conv}.
\end{corollary}
\begin{proof}
By triangle inequality for the symmetric difference pseudometric, it follows that 
\begin{align*}
\E\big[\lebesgue\big(\Lambda_+^0 \vartriangle \hat{\Lambda}{}^{\mathrm{con}}_+\big) \big] &\leq \E\big[\lebesgue\big(\hat{\Lambda}{}^{\mathrm{con}}_+ \vartriangle \hat{\Lambda}_+\big) \big] + \E\big[\lebesgue\big(\Lambda_+^0 \vartriangle \hat{\Lambda}_+\big) \big] \\ 
&\leq 2 \E\big[\lebesgue\big(\Lambda_+^0 \vartriangle \hat{\Lambda}_+\big) \big] +\delta,
\end{align*}
where we used that 
\[\lebesgue\big(\hat{\Lambda}{}^{\mathrm{con}}_+ \vartriangle \hat{\Lambda}_+\big) \leq \lebesgue\big(\Lambda_+^0 \vartriangle \hat{\Lambda}_+\big) + \delta\]
since $\Lambda_+^0$ is convex. The assertion therefore follows from Proposition \ref{prop:conv}.
\end{proof}

The convex estimator $\hat{\Lambda}^{\mathrm{con}}_+$ comes with the additional benefit of converging with respect to other distances than the  symmetric difference volume. Recall from before that
\[d_H(A,B) \coloneqq \max\big\{\sup_{x \in B} d(x,A), \sup_{x \in A} d(x,B) \big\}, \quad A,B \subset \R^d,\] 
denotes the Hausdorff distance on $\R^d$ and let  
\[\Delta_s(A, B) \coloneqq \mathcal{H}^{d-1}(\partial(A \cup B)) - \mathcal{H}^{d-1}(\partial(A \cap B)), \quad A,B \subset \R^d,\] 
be the symmetric area deviation, where $\mathcal{H}^{d-1}$ denotes the $(d-1)$-dimensional Hausdorff measure. In \cite{groemer00} it is shown that for any two compact convex sets $K,L \subset \R^d$ of diameter less than $1$ such that $(K \cap L)^\circ \neq 0$ it holds that 
\begin{equation}\label{eq:hausdorff_conv}
\Delta_s(K,L) \leq c(d) d_H(K,L) \leq C(d) \lebesgue(K \vartriangle L)^{1/d},
\end{equation}
for some constants $c(d),C(d)$ only depending on $d$. Combining \eqref{eq:hausdorff_conv} with Corollary \ref{coro:conv} then yields the following consistency result.

\begin{corollary}\label{coro:conv2}
It holds that $d_H(\hat{\Lambda}_+^{\mathrm{con}}, \Lambda^0_+) \overset{\PP}{\longrightarrow} 0$ and $\Delta_s(\hat{\Lambda}_+^{\mathrm{con}}, \Lambda^0_+) \overset{\PP}{\longrightarrow} 0$.
\end{corollary}

Given the estimator $\hat{\Lambda}_+$, numerical implementation of  $\hat{\Lambda}^{\mathrm{con}}_+$ can be conducted with methods for convexity constrained image segmentation based on the binary image input $\{(x_\alpha,\one_{\hat{\Lambda}_+}(x_\alpha)): \alpha \in [n]^d\}$, see, e.g., the recent implicit representation approach in \cite{schneider24}. 

Finally, let us consider the special case when  $\Lambda^0_+$ is a hypercuboid anchored on $\mathcal{P}$. 
Then, Theorem \ref{theo:tiling} yields the following corollary.
\begin{corollary}\label{coro:conv3}
Suppose that for any $\delta$, $\Lambda^0_+$ is a hypercuboid in $\mathcal{P}(\delta)$ such that Assumption \ref{ass:bounds_domain} is satisfied. Then, 
\[\lim_{\delta \to 0} \PP\Big(\hat{\Lambda}^\ast_+ = \overline{\Lambda}^0_+\Big) = 1.\]
If, moreover, Assumption \ref{ass:volume_conv} is satisfied, then 
\[\delta^{-(d/2+1)}\big(\hat{\theta}^\ast_\pm - \theta^0_\pm) \overset{d}{\longrightarrow} \mathcal{N}\left(0, \frac{2\theta^0_\pm}{T\lVert \nabla K\rVert^2_{L^2}v^0_\pm} \right).\]
\end{corollary}
\begin{remark} 
If $\Lambda^0_+$ is not a half-space in $[0,1]^d$, the estimators $\hat{\Lambda}{}^\ast_+$ and $\hat{\theta}^\ast_\pm$ may be replaced by $\tilde{\Lambda}_+$ and $\tilde{\theta}_\pm$ in the statement, since then $\Lambda^0_-$ cannot be a hypercuboid as well.
\end{remark}

\section{Conclusion and outlook}
\label{sec: summary}
Before discussing limitations and possible extensions of our work, let us briefly summarize our results. We have studied a change estimation problem for a stochastic heat equation \eqref{eq: intro SPDE} in $d\geq2.$ The underlying space is partitioned into $\Lambda_-\cup \Lambda_+$ by a separating hypersurface $\Gamma= \partial\Lambda_+$, where the piecewise constant diffusivity $\theta$ exhibits a jump. Following a modified likelihood approach, we have constructed an M-estimator $\hat{\Lambda}_+$ based on local measurements on a fixed uniform $\delta$-grid that exhibits certain analogies to regular design estimators in statistical image reconstruction. Our main results, Theorem \ref{theo:rate_cp} and Theorem \ref{theo:tiling}, show how the convergence properties of our estimator are determined by the number of tiles that are sliced by $\Gamma$. The estimation principle and rates are made concrete for two specific models that impose shape restrictions on the change domain $\Lambda_+$:  (A) a graph representation of $\Gamma$  with Hölder smoothness $\beta\in(0,1],$ and (B) a convex shape. We have  established the rates of convergences $\delta^\beta$ for model A and $\delta$ for model B with respect to the symmetric difference risk $\E[\lebesgue(\hat{\Lambda}_+\vartriangle\Lambda_+)]$, which are the optimal rates of convergence in the corresponding image reconstruction problems with regular design. Furthermore, the diffusivity parameters can be recovered with rate $\delta^{\beta/2}$ and $\delta^{1/2}$, respectively, which substantially improves to minimax optimal rates $\delta^{d/2+1}$ induced by the CLT in Theorem \ref{theo:clt} if the graph and the convex set are anchored on the $\delta$-grid. To conclude the paper, let us now give an outlook on potential future work that can build on our results.

Based on Theorem \ref{theo:rate_cp}, an extension of Proposition \ref{prop:interface} to a known number $m$ of change interfaces $\tau_1,\dots,\tau_m$ that yields a partition $\overline{\Lambda} =\bigcup_{0\leq i\leq m}\Lambda_i$ of layers $\Lambda_i$ with alternating diffusivities $\theta_\pm$, is straightforward. Assuming that each $\tau_m$ belongs to $\mathcal{H}(\beta,L)$, the same rate of convergence can be established. 
A more challenging model extension would allow a partition $\overline{\Lambda} = \bigcup_{0 \leq i \leq m} \Lambda_i$ with $\theta \equiv \theta_i$ on $\Lambda_i$ and $\theta_i \neq \theta_j$ for $j \neq i$, where the number $m$ of layers (for instance, sediment layers) is unknown. This model extension is unproblematic from an SPDE perspective, but poses additional statistical challenges.

In this paper, we have fixed the parameters $\theta_\pm$ and therefore also the absolute jump height $\eta=\lvert\theta_+-\theta_-\rvert$. 
The vanishing jump height regime $\eta \to 0$ as $\delta \to 0$ that has been considered for the one-dimensional change point problem in \cite{reiß_2023_change}, introduces significant technical challenges that require a sharper concentration analysis. Similarly to how the limit result from \cite{reiß_2023_change} in this regime draws analogies to classical change point limit theorems, in the multivariate case one would expect asymptotics that are comparable to \cite{mueller_indexed_96}. 

As mentioned in Remark \ref{rem:opti} and Remark \ref{rmk: opt_convex}, the convergence rates for $N^{-\beta/d}$ and $N^{-1/d}$, respectively, are optimal in the related image reconstruction problem when working with a regular design. However, as alluded to before, it is shown in \cite{tsy93} that the minimax optimal rate for irregular measurement designs that introduce a certain level of randomness is given by $N^{-\beta/(\beta+d-1)}$ for arbitrary $\beta>0$, which allows to exploit higher-order smoothness for change interface estimation. Appropriately introducing such randomness in the measurement locations $x_\alpha$ of the local observation scheme while preserving favorable probabilistic properties such as independence of the associated Brownian motions $B_{\delta,\alpha}$ is a conceptually interesting problem.

Finally, let us reiterate that we have mainly focused on optimal change domain estimation for regular local measurements and put less emphasis on rate-optimal estimators for the diffusivity parameters $\theta_\pm$. For the binary image reconstruction model with regression function  $\theta(x)=\theta_+\bm{1}_{\Lambda_+}(x)+\theta_-\bm{1}_{\Lambda_-}(x),$ \cite[Theorem 5.1.2]{tsy93} establishes the usual parametric  rate $N^{-1/2}=\delta^{d/2}$ for $\theta_\pm$. On the other hand, \cite{altmeyer_anisotrop2021} prove the minimax rate $\delta^{d/2+1}$ for a constant diffusivity, which we also obtain in the mathematically idealized setting of Theorem \ref{theo:clt} where the domain $\Lambda_+^0$ is anchored on the $\delta$-grid. This demonstrates a significant difference between heat diffusivity and image estimation and can also explain the stronger identification result of $\Lambda_+^0$, i.e. Theorem \ref{theo:tiling}, compared to ones in the literature, e.g. \cite{mueller_cube_94,mueller_indexed_96}. While in the one-dimensional change point problem recovery of $\theta_\pm$ at optimal rate $\delta^{d/2+1}$ is generally enabled by the introduction of an additional nuisance parameter $\theta_\circ$ that reduces the bias from a constant approximation $\theta \equiv \theta_\circ$ on a proposed spatial change interval, cf.\ \cite[Theorem 3.12]{reiß_2023_change}, it remains an open problem for future work to construct an optimal estimator for $\theta_\pm$ in the general multivariate framework of Section \ref{sec:boundary complexity}.

\section{Remaining proofs}\label{app:proofs}
\begin{proof} [Proof of Lemma \ref{lem: orderrest}]
By Proposition \ref{lem: orderFisher} we know that $\E[I_{\delta,\alpha}] \lesssim \delta^{-2}$ for any $\alpha \in [n]^d$. 
Since 
\[R_{\delta,\alpha} = \int_0^T X^\Delta_{\delta,\alpha}(t) \int_0^t \langle \Delta_{\theta^0} S_{\theta^0}(t-s) K_{\delta,\alpha}, \diff{W(s)} \rangle \diff{t} - \theta_+^0I_{\delta,\alpha},\]
it is enough to show 
\begin{equation}\label{eq:rem1}
\E\Big[\Big\lvert\int_0^T X^\Delta_{\delta,\alpha}(t) \int_0^t \langle \Delta_{\theta^0} S_{\theta^0}(t-s) K_{\delta,\alpha}, \diff{W(s)} \rangle \diff{t} \Big\rvert \Big] \lesssim \delta^{-2}.
\end{equation}
We have 
\begin{align*} 
&\E\Big[\Big\lvert\int_0^T X^\Delta_{\delta,\alpha}(t) \int_0^t \langle \Delta_{\theta^0} S_{\theta^0}(t-s) K_{\delta,\alpha}, \diff{W(s)} \rangle \diff{t} \Big\rvert \Big]\\
&\quad\leq \E\big[ I_{\delta,\alpha}\big]^{1/2} \E\Big[\int_0^T \Big(\int_0^t \langle \Delta_{\theta^0} S_{\theta^0}(t-s) K_{\delta,\alpha}, \diff{W(s)} \rangle\Big)^2 \diff{t}  \Big]^{1/2} \\
&\quad\lesssim \delta^{-1} \Big(\E\Big[\int_0^T \Big(\int_0^t \langle \Delta_{\theta^0} S_{\theta^0}(t-s) K_{\delta,\alpha}, \diff{W(s)}\rangle\Big)^2  \diff{t}  \Big]\Big)^{1/2}\\
&\quad= \delta^{-1} \Big(\int_0^T \int_0^t \big\lVert \Delta_{\theta^0}S_{\theta^0}(t-s) K_{\delta,\alpha} \big\rVert^2 \diff{s} \diff{t} \Big)^{1/2},
\end{align*}
where we used the Cauchy--Schwarz inequality for the first two inequalities and Fubini's theorem together with the Itô-isometry for the last line. Since 
\begin{align*} 
\int_0^T \int_0^t \big\lVert \Delta_{\theta^0}S_{\theta^0}(t-s) K_{\delta,\alpha} \big\rVert^2 \diff{s} \diff{t} &\leq T\int_0^T \big\lVert \Delta_{\theta^0}S_{\theta^0}(t) K_{\delta,\alpha} \big\rVert^2 \diff{t}\\
&= T\sum_{k \in \N} \int_0^T \lambda_k^2 \mathrm{e}^{-2\lambda_k t} \diff{t} \langle e_k,K_{\delta,\alpha} \rangle^2 \\
&\leq \frac{T}{2} \sum_{k \in \N} \lambda_k \langle e_k,K_{\delta,\alpha} \rangle^2\\ 
&= \frac{T}{2} \big\lVert (-\Delta_{\theta^0})^{1/2} K_{\delta,\alpha} \big\rVert^2\\ 
&= \frac{T}{2} \int_{\Lambda} \theta^0(x) \lvert \nabla K_{\delta,\alpha}(x) \rvert^2 \diff{x}\\ 
&\leq \frac{T\overline{\theta}}{2} \lVert \nabla K_{\delta,\alpha} \rVert^2_{L^2(\R^d)}\\
&=\delta^{-2}\frac{T \overline{\theta}\norm{\nabla K}^2_{L^2(\R^d)}}{2} ,
\end{align*}
the bound \eqref{eq:rem1} follows, proving the assertion.
\end{proof}

\begin{proof}[Proof of Corollary \ref{cor: rate}]
We only prove the assertion on $\hat{\theta}_+$ given $\liminf_{\delta \to 0} \lebesgue(\Lambda^0_+) > 0$; the case for $\hat{\theta}_-$ under the assumption $\liminf_{\delta \to 0} \lebesgue(\Lambda^0_-) > 0$ is analogous and the final statement on the convergence rate of $\hat{\theta}$ then follows from combining the first statement and Theorem \ref{theo:rate_cp} based on the inequality
\[\lvert \hat{\theta}(x) - \theta^0(x) \rvert \leq \lvert \hat{\theta}_- - \theta^0_- \rvert + \lvert \hat{\theta}_+ - \theta^0_+ \rvert +2 \overline{\theta} \lvert \one_{\hat{\Lambda}_+}(x) - \one_{\Lambda_+^0}(x) \rvert, \quad x \in (0,1)^{d}. \]
Let $\kappa = \lebesgue(\Lambda^0_+) > 0$. From \eqref{eq:tiling_err}  it follows that on the event $\{\lebesgue(\hat{\Lambda}_+ \vartriangle \Lambda^0_+) \leq \delta^{\beta/2}\}$ we have for $\delta$ small enough
\[\frac{\kappa}{2} \leq \lebesgue(\hat{\Lambda}_+ \cap \Lambda^\updownarrow_+) = \delta^d \big\lvert \big\{\alpha \in [n]^d: \Sq^\circ(\alpha) \subset (\hat{\Lambda}_+ \cap \Lambda^\updownarrow_+)\big\}\big\rvert,\]
where the equality follows from the fact that for $A,B \in \mathcal{P}$ we also have $A \cap B \in \mathcal{P}$.
Thus, for $\delta$ small enough, it follows from the second line of the calculation in \eqref{eq:low_exp} that on the event $\{\lebesgue(\hat{\Lambda}_+ \vartriangle \Lambda^0_+) \leq \delta^{\beta/2}\}$ we have 
\[\tilde{L}_\delta(\hat{\chi}) - \tilde{L}_\delta(\chi^0) \geq \frac{\kappa}{2}(\hat{\theta}_+ -\theta^0_+)^2.\]
Consequently, there exists $C^\prime > 0$ such that for any $\delta^{-1} \geq C^\prime$ and $z > 0$, 
\begin{align*} 
&\PP\big(\delta^{-\beta}(\hat{\theta}_+ -\theta^0_+)^2 \geq z\big)\\ 
&\,\leq \PP\Big(\delta^{-\beta}(\hat{\theta}_+ -\theta^0_+)^2  \geq z, \lebesgue(\hat{\Lambda}_+ \vartriangle \Lambda^0_+) \leq \delta^{\beta/2}\Big) + \PP\Big(\lebesgue(\hat{\Lambda}_+ \vartriangle \Lambda^0_+)  > \delta^{\beta/2}\Big)\\ 
&\,\leq \PP\big(\tilde{L}_\delta(\hat{\chi}) - \tilde{L}_\delta(\chi^0) \geq \tfrac{\kappa}{2}\delta^{\beta} z \big) + \PP\Big(\lebesgue(\hat{\Lambda}_+ \vartriangle \Lambda^0_+)  > \delta^{\beta/2}\Big)\\ 
&\,\leq 2\frac{\E[\tilde{L}_\delta(\hat{\chi}) - \tilde{L}_\delta(\chi^0)] }{\kappa\delta^\beta z} + \frac{\E\big[\lebesgue(\hat{\Lambda}_+ \vartriangle \Lambda^0_+) \big]}{\delta^{\beta/2}}\\ 
&\,\leq \tilde{C}\big(\tfrac{1}{z} + \delta^{\beta/2}\big),
\end{align*}
for some finite constant $\tilde{C}> 0$ independent of $\delta$ where the last line follows from \eqref{eq:cons_rate1} and Theorem \ref{theo:rate_cp}. Thus, if for given $\varepsilon > 0$ we choose $z = 2\tilde{C}/\varepsilon$, it follows that for any $\delta^{-1} \geq (2\tilde{C}/\varepsilon)^{-2/\beta} \vee C^\prime$ we have
\[\PP\big(\delta^{-\beta}(\hat{\theta}_+ -\theta^0_+)^2 \geq z\big) \leq \varepsilon,\]
which establishes $(\hat{\theta}_+ -\theta^0_+)^2 = \mathcal{O}_{\PP}(\delta^\beta)$ as $\delta^{-1} \to \infty$.
\end{proof}

\begin{proof}[Proof of \eqref{eq:Z_split}]
    Using Proposition \ref{prop: genweaksol}, the domain-specific estimators defined in \eqref{eq: theta_augm} admit the decomposition 
\begin{align}
    \label{eq: estdecomp1}
    (\theta_+^{\Lambda_+}-\theta_+^0)&=-\eta\frac{\sum_{\Sq(\alpha)^\circ \subset\Lambda_+\cap\Lambda_-^0}I_{\delta,\alpha}}{\sum_{\Sq(\alpha)^\circ \subset\Lambda_+}I_{\delta,\alpha}}+\frac{\sum_{\Sq(\alpha)^\circ \subset\Lambda_+}M_{\delta,\alpha}}{\sum_{\Sq(\alpha)^\circ \subset\Lambda_+}I_{\delta,\alpha}}\\
    \label{eq: estdecomp2}
    (\theta_+^{\Lambda_+}-\theta_-^0)&=\eta\frac{\sum_{\Sq(\alpha)^\circ \subset\Lambda_+\cap\Lambda_+^0}I_{\delta,\alpha}}{\sum_{\Sq(\alpha)^\circ \subset\Lambda_+}I_{\delta,\alpha}}+\frac{\sum_{\Sq(\alpha)^\circ \subset\Lambda_+}M_{\delta,\alpha}}{\sum_{\Sq(\alpha)^\circ \subset\Lambda_+}I_{\delta,\alpha}}\\
    \label{eq: estdecomp3}
    (\theta_-^{\Lambda_+}-\theta_-^0)&=\eta\frac{\sum_{\Sq(\alpha)^\circ \subset\Lambda_-\cap\Lambda_+^0}I_{\delta,\alpha}}{\sum_{\Sq(\alpha)^\circ \subset\Lambda_-}I_{\delta,\alpha}}+\frac{\sum_{\Sq(\alpha)^\circ \subset\Lambda_-}M_{\delta,\alpha}}{\sum_{\Sq(\alpha)^\circ \subset\Lambda_-}I_{\delta,\alpha}}\\
    \label{eq: estdecomp4}
    (\theta_-^{\Lambda_+}-\theta_+^0)&=-\eta\frac{\sum_{\Sq(\alpha)^\circ \subset\Lambda_-\cap\Lambda_-^0}I_{\delta,\alpha}}{\sum_{\Sq(\alpha)^\circ \subset\Lambda_-}I_{\delta,\alpha}}+\frac{\sum_{\Sq(\alpha)^\circ \subset\Lambda_-}M_{\delta,\alpha}}{\sum_{\Sq(\alpha)^\circ \subset\Lambda_-}I_{\delta,\alpha}}.
\end{align}
We can also write 
\begin{align}
    \label{eq: processdecomp1}
    Z_{\delta}(\theta_-^{\Lambda_+},\theta_+^{\Lambda_+},\Lambda_+)&=\frac{1}{2}\sum_{\Sq(\alpha)^\circ \subset\Lambda_+\cap\Lambda_+^0}(\theta_+^{\Lambda_+}-\theta_+^0)^2I_{\delta,\alpha}+\frac{1}{2}\sum_{\Sq(\alpha)^\circ \subset \Lambda_+\cap\Lambda_-^0}(\theta_+^{\Lambda_+}-\theta_-^0)^2I_{\delta,\alpha}\\
    \label{eq: processdecomp3}
    &\phantom{=}+\frac{1}{2}\sum_{\Sq(\alpha)^\circ \subset\Lambda_-\cap\Lambda_-^0}(\theta_-^{\Lambda_+}-\theta_-^0)^2I_{\delta,\alpha}+\frac{1}{2}\sum_{\Sq(\alpha)^\circ \subset\Lambda_-\cap\Lambda_+^0}(\theta_-^{\Lambda_+}-\theta_+^0)^2I_{\delta,\alpha}\\
    \label{eq: processdecomp5}
    &\phantom{=}-\sum_{\Sq(\alpha)^\circ \subset\Lambda_+\cap\Lambda_+^0}(\theta_+^{\Lambda_+}-\theta_+^0)M_{\delta,\alpha}-\sum_{\Sq(\alpha)^\circ \subset\Lambda_+\cap\Lambda_-^0}(\theta_+^{\Lambda_+}-\theta_-^0)M_{\delta,\alpha}\\
    \label{eq: processdecomp7}
    &\phantom{=}-\sum_{\Sq(\alpha)^\circ \subset\Lambda_-\cap\Lambda_-^0}(\theta_-^{\Lambda_+}-\theta_-^0)M_{\delta,\alpha}-\sum_{\Sq(\alpha)^\circ \subset\Lambda_-\cap\Lambda_+^0}(\theta_-^{\Lambda_+}-\theta_+^0)M_{\delta,\alpha}
\end{align}
Plugging \eqref{eq: estdecomp1} and \eqref{eq: estdecomp2} into \eqref{eq: processdecomp1} yields the simplification
\begin{equation}
\label{eq: firstsum}
\begin{split}
    &\frac{1}{2}\sum_{\Sq(\alpha)^\circ \subset\Lambda_+\cap\Lambda_+^0}(\theta_+^{\Lambda_+}-\theta_+^0)^2I_{\delta,\alpha}+\frac{1}{2}\sum_{\Sq(\alpha)^\circ \subset\Lambda_+\cap\Lambda_-^0}(\theta_+^{\Lambda_+}-\theta_-^0)^2I_{\delta,\alpha}\\
    &\,=\frac{1}{2}\frac{(\sum_{\Sq(\alpha)^\circ \subset\Lambda_+}M_{\delta,\alpha})^2}{\sum_{\Sq(\alpha)^\circ \subset\Lambda_+}I_{\delta,\alpha}}+\frac{\eta^2}{2}\frac{\sum_{\Sq(\alpha)^\circ \subset\Lambda_+\cap\Lambda_+^0}I_{\delta,\alpha}\sum_{\Sq(\alpha)^\circ \subset\Lambda_+\setminus\Lambda_+^0}I_{\delta,\alpha}}{\sum_{\Sq(\alpha)^\circ \subset\Lambda_+}I_{\delta,\alpha}}.
    \end{split}
\end{equation}
Similarly, using \eqref{eq: estdecomp3} and \eqref{eq: estdecomp4} in \eqref{eq: processdecomp3} gives 
\begin{equation}
\label{eq: secondsum}
\begin{split}
    &\frac{1}{2}\sum_{\Sq(\alpha)^\circ \subset\Lambda_-\cap\Lambda_-^0}(\theta_-^{\Lambda_+}-\theta_-^0)^2I_{\delta,\alpha}+\frac{1}{2}\sum_{\Sq(\alpha)^\circ \subset\Lambda_-\cap\Lambda_+^0}(\theta_-^{\Lambda_+}-\theta_+^0)^2I_{\delta,\alpha}\\
    &\,=\frac{1}{2}\frac{(\sum_{\Sq(\alpha)^\circ \subset\Lambda_-}M_{\delta,\alpha})^2}{\sum_{\Sq(\alpha)^\circ \subset\Lambda_-}I_{\delta,\alpha}}+\frac{\eta^2}{2}\frac{\sum_{\Sq(\alpha)^\circ \subset\Lambda_-\cap\Lambda_-^0}I_{\delta,\alpha}\sum_{\Sq(\alpha)^\circ \subset\Lambda_+^0\setminus\Lambda_+}I_{\delta,\alpha}}{\sum_{\Sq(\alpha)^\circ \subset\Lambda_-}I_{\delta,\alpha}}.
\end{split}
\end{equation}
As for the martingale expressions \eqref{eq: processdecomp5} we obtain
\begin{equation}
\label{eq: thirdsum}
\begin{split}
    &\sum_{\Sq(\alpha)^\circ \subset\Lambda_+\cap\Lambda_+^0}(\theta_+^{\Lambda_+}-\theta_+^0)M_{\delta,\alpha}+\sum_{\Sq(\alpha)^\circ \subset\Lambda_+\cap\Lambda_-^0}(\theta_+^{\Lambda_+}-\theta_-^0)M_{\delta,\alpha}\\
    &\quad=\frac{(\sum_{\Sq(\alpha)^\circ \subset\Lambda_+}M_{\delta,\alpha})^2}{\sum_{\Sq(\alpha)^\circ \subset\Lambda_+}I_{\delta,\alpha}}\\
    &\qquad+
    \frac{\eta}{\sum_{\Sq(\alpha)^\circ \subset\Lambda_+}I_{\delta,\alpha}}\left(\sum_{\Sq(\alpha)^\circ \subset\Lambda_+\cap\Lambda_-^0}M_{\delta,\alpha}\sum_{\Sq(\alpha)^\circ \subset\Lambda_+\cap \Lambda_+^0}I_{\delta,\alpha}-\sum_{\Sq(\alpha)^\circ \subset\Lambda_+\cap\Lambda_+^0}M_{\delta,\alpha}\sum_{\Sq(\alpha)^\circ \subset\Lambda_+\cap\Lambda_-^0}I_{\delta,\alpha}\right)
    \end{split}
\end{equation}
and likewise for \eqref{eq: processdecomp7}
\begin{equation}
\label{eq: forthsum}
\begin{split}
    &\sum_{\Sq(\alpha)^\circ \subset\Lambda_-\cap\Lambda_-^0}(\theta_-^{\Lambda_+}-\theta_-^0)M_{\delta,\alpha}+\sum_{\Sq(\alpha)^\circ \subset\Lambda_-\cap\Lambda_+^0}(\theta_-^{\Lambda_+}-\theta_+^0)M_{\delta,\alpha}\\
    &\quad=\frac{(\sum_{\Sq(\alpha)^\circ \subset\Lambda_-}M_{\delta,\alpha})^2}{\sum_{\Sq(\alpha)^\circ \subset\Lambda_-}I_{\delta,\alpha}}\\
    &\qquad+\frac{\eta}{\sum_{\Sq(\alpha)^\circ \subset\Lambda_-}I_{\delta,\alpha}}\left(\sum_{\Sq(\alpha)^\circ \subset\Lambda_-\cap\Lambda_+^0}I_{\delta,\alpha}\sum_{\Sq(\alpha)^\circ \subset\Lambda_-\cap\Lambda_-^0}M_{\delta,\alpha}-\sum_{\Sq(\alpha)^\circ \subset\Lambda_-\cap\Lambda_-^0}I_{\delta,\alpha}\sum_{\Sq(\alpha)^\circ \subset\Lambda_-\cap\Lambda_+^0}M_{\delta,\alpha}\right).
\end{split}
\end{equation}
Combining \eqref{eq: firstsum}, \eqref{eq: secondsum}, \eqref{eq: thirdsum} and \eqref{eq: forthsum} implies
\begin{equation*}
\begin{split}
    &Z_{\delta}(\theta_-^{\Lambda_+},\theta_+^{\Lambda_+},\Lambda_+)\\
    &\,=\frac{\eta^2}{2}\frac{\sum_{\Sq(\alpha)^\circ \subset\Lambda_+\cap\Lambda_+^0}I_{\delta,\alpha}\sum_{\Sq(\alpha)^\circ \subset\Lambda_+\setminus\Lambda_+^0}I_{\delta,\alpha}}{\sum_{\Sq(\alpha)^\circ \subset\Lambda_+}I_{\delta,\alpha}}+\frac{\eta^2}{2}\frac{\sum_{\Sq(\alpha)^\circ \subset\Lambda_-\cap\Lambda_-^0}I_{\delta,\alpha}\sum_{\Sq(\alpha)^\circ \subset\Lambda_+^0\setminus\Lambda_+}I_{\delta,\alpha}}{\sum_{\Sq(\alpha)^\circ \subset\Lambda_-}I_{\delta,\alpha}}\\
    &\,\phantom{=}-\frac{(\sum_{\Sq(\alpha)^\circ \subset\Lambda_+}M_{\delta,\alpha})^2}{2\sum_{\Sq(\alpha)^\circ \subset\Lambda_+}I_{\delta,\alpha}}-\frac{(\sum_{\Sq(\alpha)^\circ \subset\Lambda_-}M_{\delta,\alpha})^2}{2\sum_{\Sq(\alpha)^\circ \subset\Lambda_-}I_{\delta,\alpha}}\\
    &\,\phantom{=}-\frac{\eta}{\sum_{\Sq(\alpha)^\circ \subset\Lambda_+}I_{\delta,\alpha}}\left(\sum_{\Sq(\alpha)^\circ \subset\Lambda_+\cap\Lambda_-^0}M_{\delta,\alpha}\sum_{\Sq(\alpha)^\circ \subset\Lambda_+\cap \Lambda_+^0}I_{\delta,\alpha}-\sum_{\Sq(\alpha)^\circ \subset\Lambda_+\cap\Lambda_+^0}M_{\delta,\alpha}\sum_{\Sq(\alpha)^\circ \subset\Lambda_+\cap\Lambda_-^0}I_{\delta,\alpha}\right)\\
    &\,\phantom{=} -\frac{\eta}{\sum_{\Sq(\alpha)^\circ \subset\Lambda_-}I_{\delta,\alpha}}\left(\sum_{\Sq(\alpha)^\circ \subset\Lambda_-\cap\Lambda_+^0}I_{\delta,\alpha}\sum_{\Sq(\alpha)^\circ \subset\Lambda_-\cap\Lambda_-^0}M_{\delta,\alpha}-\sum_{\Sq(\alpha)^\circ \subset\Lambda_-\cap\Lambda_-^0}I_{\delta,\alpha}\sum_{\Sq(\alpha)^\circ \subset\Lambda_-\cap\Lambda_+^0}M_{\delta,\alpha}\right)
\end{split}
\end{equation*}
proving \eqref{eq:Z_split}.
\end{proof}

\begin{proof}[Proof of Lemma \ref{lem:order_frac}]
\begin{enumerate}
\item[(i)] We only prove the assertion for $\sup_{\Lambda_+ \in \mathcal{A}_+^\prime} \frac{\sum_{\Sq(\alpha) \subset \Lambda_+} M_{\delta,\alpha}}{\sum_{\Sq(\alpha) \subset \Lambda_+} I_{\delta,\alpha}}$, the other assertion follows similarly. Since $\E[I_{\delta,\alpha}] \asymp \delta^{-2}$, there is a  constant $C_1$ independent of $\delta$ such that 
\[\sum_{\Sq(\alpha) \subset \Lambda_+} \E[I_{\delta,\alpha}] \leq C_1 \delta^{-2} \lvert\{\alpha: \Sq(\alpha) \subset \Lambda_+\}\rvert,\] 
and thus
\begin{align*}
&\PP\bigg(\sup_{\Lambda_+ \in \mathcal{A}_+^\prime} \Big\lvert\frac{\sum_{\Sq(\alpha) \subset \Lambda_+} M_{\delta,\alpha}}{\sum_{\Sq(\alpha) \subset \Lambda_+} I_{\delta,\alpha}}\Big\rvert \geq z \bigg)\\
&\,\leq \PP\Big(\sup_{\Lambda_+ \in \mathcal{A}_+^\prime} \delta^{2}\lvert\{\alpha: \Sq(\alpha) \subset \Lambda_+\}\rvert^{-1} \big\lvert \sum_{\Sq(\alpha) \subset \Lambda_+} M_{\delta,\alpha} \big\rvert \geq C_1z/2\Big)\\
&\,\quad + \PP\Big(\sup_{\Lambda_+ \in \mathcal{A}_+^\prime} \delta^{2}\lvert\{\alpha: \Sq(\alpha) \subset \Lambda_+\}\rvert^{-1}\big\lvert \sum_{\Sq(\alpha) \subset \Lambda_+}  (I_{\delta,\alpha} - \E[I_{\delta,\alpha}])\big\rvert \geq C_1/2\Big).
\end{align*}
Let $\mathcal{A}_+^i \coloneqq \{\Lambda_+ \in \mathcal{A}_+^\prime: \delta^{-(i-1)} \leq \lvert\{\alpha: \Sq(\alpha) \subset \Lambda_+\}\rvert \leq \delta^{-i}\}$. For the first probability we get for some constants $C_i$, $i=2,3,4$, that are independent of $\delta$,
\begin{align*} 
&\PP\Big(\sup_{\Lambda_+ \in \mathcal{A}_+^\prime} \delta^{2}\lvert\{\alpha: \Sq(\alpha) \subset \Lambda_+\}\rvert \big\lvert^{-1} \sum_{\Sq(\alpha) \subset \Lambda_+} M_{\delta,\alpha} \big\rvert \geq C_1z/2\Big)\\
&\,\leq \sum_{i=1}^d \bigg\{\PP\Big(\sup_{\Lambda_+ \in \mathcal{A}_+^i} \delta^{i+1}\big\lvert \sum_{\Sq(\alpha) \subset \Lambda_+ }\overline{M}_{\delta,\alpha} \big\rvert \geq C_1z/4 \Big)  \\
&\,\quad\qquad + \PP\Big(\sup_{\Lambda_+ \in \mathcal{A}_+^i} \delta^{i+1} \sum_{\Sq(\alpha) \subset \Lambda_+} (M_{\delta,\alpha} - \overline{M}_{\delta,\alpha}) \big\rvert \geq C_1z/4,  \sup_{\Lambda_+ \in \mathcal{A}_+^i} \sum_{\Sq(\alpha) \subset \Lambda_+} \lvert I_{\delta,\alpha} - \E[I_{\delta,\alpha}] \rvert \leq L_i \Big)\\
&\,\quad\qquad + \PP\Big(\sup_{\Lambda_+ \in \mathcal{A}_+^i} \sum_{\Sq(\alpha) \subset \Lambda_+} \lvert I_{\delta,\alpha} - \E[I_{\delta,\alpha}] \rvert > L_i \Big) \bigg\}\\
&\,\leq \sum_{i=1}^d \bigg\{\PP\Big(\sup_{\Lambda_+ \in \mathcal{A}_+^i} \delta^{i+1} \big\lvert \sum_{\Sq(\alpha) \subset \Lambda_+ }\overline{M}_{\delta,\alpha} \big\rvert \geq C_1z/4 \Big)  \\
&\,\quad\qquad + \PP\Big(\sup_{\Lambda_+ \in \mathcal{A}_+^i} \delta^{i+1} \sum_{\Sq(\alpha) \subset \Lambda_+} (M_{\delta,\alpha} - \overline{M}_{\delta,\alpha}) \big\rvert \geq C_1z/4,  \sup_{\Lambda_+ \in \mathcal{A}_+^i} \sum_{\Sq(\alpha) \subset \Lambda_+} \lvert I_{\delta,\alpha} - \E[I_{\delta,\alpha}] \rvert \leq L_i \Big)\\
&\,\quad\qquad + \PP\Big(\delta^{-i}\sup_{\alpha \in [n]^d} \lvert I_{\delta,\alpha} - \E[I_{\delta,\alpha}] \rvert > L_i \Big) \bigg\}\\
&\,\lesssim  \sum_{i=1}^d \bigg\{\delta^{-d}\exp\Big(-C_3 \frac{L_i^2\delta^{2i}}{L_i\delta^{i} + C_4 \delta^{-2}}\Big) \\
&\qquad\quad+ \sum_{\Lambda_+ \in \mathcal{A}_+^i} \bigg( \exp\big(-C_2\delta^{-i}z^2 \big) + \exp\big(-\delta^{-2(1+i)}C_1^2z^2/(32L_i) \big) \bigg)  \bigg\}\\ 
&\,\lesssim \sum_{i=1}^d \bigg\{\delta^{-d}\exp\Big(-C_3 \frac{L_i^2\delta^{2i}}{L_i\delta^{i} + C_4 \delta^{-2}}\Big) + \delta^{-c}\bigg( \exp\big(-\delta^{-1}C_1^2z^2/32 \big) + \exp\Big(-\tfrac{\delta^{-2(1+i)}C_1^2z^2}{32L_i} \Big) \bigg) \bigg\},
\end{align*}
where we used Proposition \ref{lem: orderFisher} and Proposition \ref{prop:coupling} as well as a Gaussian tail inequality for the normally distributed random variables $\sum_{\Sq(\alpha) \subset \Lambda_+} \overline{M}_{\delta,\alpha}$ for the penultimate line.
Then, choosing $L_i = R_L\delta^{-(1+i)} \sqrt{\log (\delta^{-1})}$ and  $z = z_L = R_L\sqrt{\delta\log(\delta^{-1})}$  shows
\[\lim_{R_L \to \infty} \PP\Big(\sup_{\Lambda_+ \in \mathcal{A}_+^\prime} \delta^{2}\lvert\{\alpha: \Sq(\alpha) \subset \Lambda_+\}\rvert^{-1} \big\lvert \sum_{\Sq(\alpha) \subset \Lambda_+} M_{\delta,\alpha} \big\rvert \geq C_1R_L\sqrt{\delta\log(\delta^{-1})}/2\Big) = 0, \]
whence 
\[\sup_{\Lambda_+ \in \mathcal{A}_+^\prime} \delta^{2}\lvert\{\alpha: \Sq(\alpha) \subset \Lambda_+\}\rvert^{-1} \big\lvert \sum_{\Sq(\alpha) \subset \Lambda_+} M_{\delta,\alpha} \big\rvert = \mathcal{O}_{\PP}\Big(\sqrt{\delta\log(\delta^{-1})}\Big).\]
Another application of a union bound and Proposition \ref{lem: orderFisher} yields
\begin{align*} 
&\PP\Big(\sup_{\Lambda_+ \in \mathcal{A}_+^\prime} \delta^{2}\lvert\{\alpha: \Sq(\alpha) \subset \Lambda_+\}\rvert^{-1}\big\lvert \sum_{\Sq(\alpha) \subset \Lambda_+}  (I_{\delta,\alpha} - \E[I_{\delta,\alpha}])\big\rvert \geq C_1/2\Big) \\
&\,\lesssim \sum_{\Lambda_+ \in \mathcal{A}_+^\prime} \exp\Big(-C_3\delta^{-2}\lvert\{\alpha: \Sq(\alpha) \subset \Lambda_+\}\rvert\Big)\\
&\,\leq \delta^{-c}\exp(-C_3\delta^{-2}) \underset{\delta \to 0}{\longrightarrow} 0.
\end{align*}
The above estimates now yield the first assertion.  The statement with $\Lambda_-$ replacing $\Lambda_+$ is proved in the same way. 
\item[(ii)] Same as (i).
\item[(iii)] Similarly to the proof of (i) we find
\begin{align*} 
&\PP\Big(\sup_{\Lambda_+ \in \mathcal{A}_+^\prime} \delta^{2}\lvert\{\alpha: \Sq(\alpha) \subset \Lambda_+\}\rvert^{-1} \big\lvert \sum_{\Sq(\alpha) \subset \Lambda_+} M_{\delta,\alpha} \big\rvert^2 \geq C_1z/2\Big)\\
&\, \leq \sum_{i=1}^d \bigg\{\PP\Big(\sup_{\Lambda_+ \in \mathcal{A}_+^i} \delta^{(i+1)/2} \big\lvert \sum_{\Sq(\alpha) \subset \Lambda_+ }\overline{M}_{\delta,\alpha} \big\rvert \geq (C_1z/8)^{1/2} \Big)  \\
&\,\quad\qquad + \PP\Big(\sup_{\Lambda_+ \in \mathcal{A}_+^i} \delta^{(i+1)/2} \sum_{\Sq(\alpha) \subset \Lambda_+} (M_{\delta,\alpha} - \overline{M}_{\delta,\alpha}) \big\rvert \geq (C_1z/8)^{1/2},  \sup_{\Lambda_+ \in \mathcal{A}_+^i} \sum_{\Sq(\alpha) \subset \Lambda_+} \lvert I_{\delta,\alpha} - \E[I_{\delta,\alpha}] \rvert \leq L_i \Big)\\
&\,\quad\qquad + \PP\Big(\delta^{-i}\sup_{\alpha \in [n]^d} \lvert I_{\delta,\alpha} - \E[I_{\delta,\alpha}] \rvert > L_i \Big) \bigg\}\\
&\,\lesssim  \sum_{i=1}^d \bigg\{\delta^{-d}\exp\Big(-C_3 \frac{L_i^2\delta^{2i}}{L_i\delta^{i} + C_4 \delta^{-2}}\Big)\\
&\qquad\quad+ \sum_{\Lambda_+ \in \mathcal{A}_+^i} \bigg( \exp\big(-\delta C_1z/16 \big) + \exp\big(-\delta^{-(1+i)}C_1z/(16L_i) \big) \bigg)  \bigg\}\\ 
&\,\lesssim  \sum_{i=1}^d \bigg\{\delta^{-d}\exp\Big(-C_3 \frac{L_i^2\delta^{2i}}{L_i\delta^{i} + C_4 \delta^{-2}}\Big) + \delta^{-c} \bigg( \exp\big(-\delta C_1z/16 \big) + \exp\big(-\delta^{-(1+i)}C_1z/(16L_i) \big) \bigg)  \bigg\}.
\end{align*}
Thus, choosing $L_i = R_L \delta^{-(1+i)}\sqrt{\log \delta^{-1}}$ and $z_L = R_L \delta^{-1} \log(\delta^{-1})$ for $R_L \to \infty$ shows that 
\[\sup_{\Lambda_+ \in \mathcal{A}_+} \delta^{-2}\lvert\{\alpha: \Sq(\alpha) \subset \Lambda_+\}\rvert \big\lvert \sum_{\Sq(\alpha) \subset \Lambda_+} M_{\delta,\alpha} \big\rvert^2 = \mathcal{O}_{\PP}(\delta^{-1}\log(\delta^{-1})).\]
The remainder of the proof is now carried out analogously to part (i).
\item[(iv)] For any $z > 0$ we have
\begin{align*} 
&\PP\Big(\big\lvert \sup_{ \Lambda_+ \in \mathcal{A}_+^\prime} \frac{\sum_{\Sq(\alpha) \subset \Lambda_+}\E[I_{\delta,\alpha}]}{\sum_{\Sq(\alpha) \subset \Lambda_+} I_{\delta,\alpha}} - 1\big\rvert \geq z\Big)\\
&\,= \PP\Big(\big\lvert \sup_{\Lambda_+ \in \mathcal{A}_+^\prime} \frac{\sum_{\Sq(\alpha) \subset \Lambda_+}(\E[I_{\delta,\alpha}] - I_{\delta,\alpha})}{I_{\delta,\alpha}}\big\rvert \geq z\Big)\\ 
&\,\leq \sum_{\Lambda_+ \in \mathcal{A}_+^\prime} \bigg\{\PP\Big(\frac{\lvert \sum_{\Sq(\alpha) \subset \Lambda_+}(I_{\delta,\alpha}- \E[I_{\delta,\alpha}])\rvert}{\sum_{\Sq(\alpha) \subset \Lambda_+} \E[I_{\delta,\alpha}]} \geq z/2\Big)\\
&\qquad\qquad+ \PP\Big(\big\lvert \sum_{\Sq(\alpha) \subset \Lambda_+} (I_{\delta,\alpha} -\E[I_{\delta,\alpha}]) \big\rvert \geq \sum_{\Sq(\alpha) \subset \Lambda_+} \E[I_{\delta,\alpha}]/2\Big)\bigg\}\\ 
 &\leq 2\sum_{\Lambda_+ \in \mathcal{A}_+^\prime} \bigg\{\exp\Big(-\frac{C_1\lvert \{\alpha: \Sq(\alpha) \subset \Lambda_+ \}\rvert \delta^{-2}z^2}{z + C_2} \Big)  + \exp\big(-C_3\lvert\{\alpha: \Sq(\alpha) \subset \Lambda_+\}\rvert \delta^{-2}\big)\bigg\}.
\end{align*}
Thus, using that $\lvert \mathcal{A}_+^\prime \rvert \leq \delta^{-c}$, it follows that 
\[\lim_{M \to \infty} \PP\Big(\big\lvert \sup_{\Lambda_+ \in \mathcal{A}_+^\prime} \frac{\sum_{\Sq(\alpha) \subset \Lambda_+}\E[I_{\delta,\alpha}]}{\sum_{\Sq(\alpha) \subset \Lambda_+} I_{\delta,\alpha}} - 1\big\rvert \geq M\delta\sqrt{\log(\delta^{-1}}\Big) = 0, \]
and therefore, 
\[\sup_{\Lambda_+ \in \mathcal{A}_+^\prime} \big\lvert \frac{\sum_{\Sq(\alpha) \subset \Lambda_+} \E[I_{\delta,\alpha}]}{\sum_{\Sq(\alpha) \subset \Lambda_+} I_{\delta,\alpha}} - 1\big\rvert = \mathcal{O}_{\PP}\Big(\delta \sqrt{\log(\delta^{-1}}) \Big).\]
The remaining assertion follows in the same way.
\item[(v)] Follows from similar (but slightly simpler) arguments to the proof of (iv).
\end{enumerate}
\end{proof}

\begin{proof}[Proof of Corollary \ref{coro:order_frac}]
 We can write 
\begin{align*} 
&\frac{\delta^{d+2}}{\sum_{\Sq(\alpha)^\circ \subset \tilde{\Lambda}_+} I_{\delta,\alpha}}\sum_{\Sq(\alpha)^\circ \subset \tilde{\Lambda}_+ \cap \Lambda^0_-} I_{\delta,\alpha}\sum_{\Sq(\alpha)^\circ \subset \tilde{\Lambda}_+ \cap \Lambda^0_+} M_{\delta,\alpha}\\
&\,= \delta^{d+2}\frac{\sum_{\Sq(\alpha)^\circ \subset \tilde{\Lambda}_+ \setminus \Lambda^0_+} \E[I_{\delta,\alpha}]\sum_{\Sq(\alpha)^\circ  \in \tilde{\Lambda}_+ \cap \Lambda^0_+} \E[I_{\delta,\alpha}]}{\sum_{\Sq(\alpha)^\circ \subset \tilde{\Lambda}_+} \E[I_{\delta,\alpha}]} \\
&\qquad \times \frac{\sum_{\Sq(\alpha)^\circ \subset \tilde{\Lambda}_+ \cap \Lambda^0_+} M_{\delta,\alpha}}{\sum_{\Sq(\alpha)^\circ \subset \tilde{\Lambda}_+ \cap \Lambda^0_+} I_{\delta,\alpha}} \frac{\frac{\sum_{\Sq(\alpha)^\circ \subset \tilde{\Lambda}_+ \setminus \Lambda^0_+} I_{\delta}(\alpha)}{\sum_{\Sq(\alpha)^\circ \subset \tilde{\Lambda}_+ \setminus \Lambda^0_+} \E[I_{\delta}(\alpha)]}\frac{\sum_{\Sq(\alpha)^\circ \subset \tilde{\Lambda}_+ \cap \Lambda^0_+} I_{\delta}(\alpha)}{\sum_{\Sq(\alpha)^\circ \subset \tilde{\Lambda}_+ \cap \Lambda^0_+} \E[I_{\delta}(\alpha)]} }{\frac{\sum_{\Sq(\alpha)^\circ \subset \tilde{\Lambda}_+} I_{\delta}(\alpha)}{\sum_{\Sq(\alpha)^\circ \subset \tilde{\Lambda}_+} \E[I_{\delta}(\alpha)]} } \one_{\tilde{\Lambda}_+ \cap \Lambda^0_+ \neq \varnothing}. 
\end{align*}
The first statement now follows from the fact that by Proposition \ref{lem: orderFisher} we have 
\[\delta^{d+2}\frac{\sum_{\Sq(\alpha)^\circ \subset \tilde{\Lambda}_+ \setminus \Lambda^0_+} \E[I_{\delta,\alpha}]\sum_{\Sq(\alpha)^\circ  \in \tilde{\Lambda}_+ \cap \Lambda^0_+} \E[I_{\delta,\alpha}]}{\sum_{\Sq(\alpha)^\circ \subset \tilde{\Lambda}_+} \E[I_{\delta,\alpha}]} \asymp \frac{\lebesgue(\tilde{\Lambda}_+ \setminus \Lambda^0_+)\lebesgue(\tilde{\Lambda}_+ \cap \Lambda^0_+)}{\lebesgue(\tilde{\Lambda}_+)}\]
and that by Lemma \ref{lem:order_frac} 
\begin{align*} 
&\bigg\lvert \frac{\sum_{\Sq(\alpha)^\circ \subset \tilde{\Lambda}_+\cap \Lambda^0_+} M_{\delta,\alpha}}{\sum_{\Sq(\alpha)^\circ \subset \tilde{\Lambda}_+} I_{\delta,\alpha}} \frac{\frac{\sum_{\Sq(\alpha)^\circ \subset \tilde{\Lambda}_+ \setminus \Lambda^0_+} I_{\delta}(\alpha)}{\sum_{\Sq(\alpha)^\circ \subset \tilde{\Lambda}_+ \setminus \Lambda^0_+} \E[I_{\delta}(\alpha)]}\frac{\sum_{\Sq(\alpha)^\circ \subset \tilde{\Lambda}_+ \cap \Lambda^0_+} I_{\delta}(\alpha)}{\sum_{\Sq(\alpha)^\circ \subset \tilde{\Lambda}_+ \cap \Lambda^0_+} \E[I_{\delta}(\alpha)]} }{\frac{\sum_{\Sq(\alpha)^\circ \subset \tilde{\Lambda}_+} I_{\delta}(\alpha)}{\sum_{\Sq(\alpha)^\circ \subset \tilde{\Lambda}_+} \E[I_{\delta}(\alpha)]} } \one_{\tilde{\Lambda}_+ \cap \Lambda^0_+ \neq \varnothing} \bigg\rvert\\ 
&\,\leq \sup_{\Lambda_+\in \mathcal{A}_+: \Lambda_+ \cap \Lambda^0_+ \neq \varnothing} \frac{\sum_{\Sq(\alpha)^\circ \subset \Lambda_+ \cap \Lambda^0_+} M_{\delta,\alpha}}{\sum_{\Sq(\alpha)^\circ \subset \Lambda_+ \cap \Lambda^0_+} I_{\delta,\alpha}} \sup_{\Lambda_+ \in \mathcal{A}_+: \Lambda_+ \cap \Lambda^0_{-} \neq \varnothing} \frac{\sum_{\Sq(\alpha)^\circ \subset \Lambda_+ \cap \Lambda^0_-}I_{\delta,\alpha}}{\sum_{\Sq(\alpha)^\circ \subset \Lambda_+\cap \Lambda^0_-} \E[I_{\delta,\alpha}]}\\ 
&\qquad\times \sup_{\Lambda_+ \in \mathcal{A}_+: \Lambda_+ \cap \Lambda^0_{+} \neq \varnothing} \frac{\sum_{\Sq(\alpha)^\circ \subset \Lambda_+ \cap \Lambda^0_+}I_{\delta,\alpha}}{\sum_{\Sq(\alpha)^\circ \subset \Lambda_+\cap \Lambda^0_+} \E[I_{\delta,\alpha}]} \sup_{\Lambda_+ \in \mathcal{A}_+^\prime} \frac{\sum_{\Sq(\alpha) \subset \Lambda_+}\E[I_{\delta,\alpha}]}{\sum_{\Sq(\alpha) \subset \Lambda_+} I_{\delta,\alpha}}\\ 
&\,= \mathcal{O}_{\PP}\Big(\sqrt{\delta\log(\delta^{-1})}\Big)
\end{align*}
The remaining statements are proved analogously.
\end{proof}

\begin{proof}[Proof of equation \eqref{eq:lower}]
    First, note that since in the context of \eqref{eq:lower}  the sets $\overline{\Lambda^0_+}$ and $\tilde{\Lambda}_+$ are tiling sets, the assumption $\tilde{\Lambda}_+ \notin \{\overline{\Lambda^0_+}, \Lambda^0_-\}$  is equivalent to 
$\lebesgue(\tilde{\Lambda}_+ \vartriangle \Lambda^0_+) \geq \delta^d$ and $\lebesgue(\tilde{\Lambda}_+ \cup \Lambda^0_+) \leq 1- \delta^d$. Furthermore, we observe that 
\[\frac{\lebesgue(\tilde{\Lambda}_+\setminus\Lambda_+^0)\lebesgue(\tilde{\Lambda}_+\cap\Lambda_+^0)}{\lebesgue(\tilde{\Lambda}_+)}+\frac{\lebesgue(\Lambda_+^0\setminus\tilde{\Lambda}_+)\lebesgue(\tilde{\Lambda}_-\cap \Lambda_-^0)}{\lebesgue(\tilde{\Lambda}_-)}\geq \lebesgue(\tilde{\Lambda}_+ \vartriangle \Lambda^0_+)\Big\{\lebesgue(\tilde{\Lambda}_+ \cap \Lambda^0_+) \wedge \underbrace{(1-\lebesgue(\tilde{\Lambda}_+ \cup \Lambda^0_+)}_{=\lebesgue(\tilde{\Lambda}_- \cap \Lambda^0_-)}\Big\}.\]
We  now distinguish three cases:

\noindent\underline{Case 1: $\lebesgue(\tilde{\Lambda}_+ \cap \Lambda^0_+) = 0$}

Then, $\lebesgue(\tilde{\Lambda}_- \cap \Lambda^0_-) \geq \delta^d$, for otherwise $\lebesgue(\tilde{\Lambda}_+ \cup \Lambda^0_+) = 1$. This implies 
\begin{align*}
\frac{\lebesgue(\tilde{\Lambda}_+\setminus\Lambda_+^0)\lebesgue(\tilde{\Lambda}_+\cap\Lambda_+^0)}{\lebesgue(\tilde{\Lambda}_+)}+\frac{\lebesgue(\Lambda_+^0\setminus\tilde{\Lambda}_+)\lebesgue(\tilde{\Lambda}_-\cap \Lambda_-^0)}{\lebesgue(\tilde{\Lambda}_-)} &\geq \frac{\lebesgue(\Lambda_+^0\setminus\tilde{\Lambda}_+)\lebesgue(\tilde{\Lambda}_-\cap \Lambda_-^0)}{\lebesgue(\tilde{\Lambda}_-)}\\
&\geq \lebesgue(\Lambda_+^0\setminus\tilde{\Lambda}_+)\lebesgue(\tilde{\Lambda}_-\cap \Lambda_-^0)\\
&= \lebesgue(\Lambda^0_+)\lebesgue(\tilde{\Lambda}_-\cap \Lambda_-^0) \geq \delta^d \kappa,
\end{align*}
for $\delta$ small enough.

\noindent\underline{Case 2: $\delta^d \leq \lebesgue(\tilde{\Lambda}_+ \cap \Lambda^0_+) \leq \lebesgue(\Lambda^0_+)/2$}

Then, for $\delta$ small enough
\[\lebesgue(\tilde{\Lambda}_+ \vartriangle \Lambda^0_+) \geq \lebesgue(\Lambda^0_+) - \lebesgue(\tilde{\Lambda}_+ \cap \Lambda^0_+) \geq \lebesgue(\Lambda^0_+)/2 \geq \kappa/2,\] 
and hence
\[\lebesgue(\tilde{\Lambda}_+ \vartriangle \Lambda^0_+)\Big\{\lebesgue(\tilde{\Lambda}_+ \cap \Lambda^0_+)\wedge (1-\lebesgue(\tilde{\Lambda}_+ \cup \Lambda^0_+))\Big\} \geq \delta^d \kappa/2.\]

\noindent\underline{Case 3: $\lebesgue(\Lambda^0_+)/2 < \lebesgue(\tilde{\Lambda}_+ \cap \Lambda^0_+)$}

Suppose first that $\lebesgue(\Lambda^0_+ \vartriangle \tilde{\Lambda}_+) \leq (1-\lebesgue(\Lambda^0_+))/2$. Then, for $\delta$ small enough
\[\lebesgue(\Lambda^0_+ \cup \tilde{\Lambda}_+) \leq (1-\lebesgue(\Lambda^0_+))/2+\lebesgue(\Lambda^0_+ \cap \tilde{\Lambda}_+) \leq (1+\lebesgue(\Lambda^0_+))/2 \leq (1+\kappa)/2,\]
and therefore 
\[1- \lebesgue(\Lambda^0_+ \cup \tilde{\Lambda}_+) \geq  (1 - \kappa)/2,\]
which, because $\kappa \leq 1- \kappa$, yields
\[\lebesgue(\tilde{\Lambda}_+ \vartriangle \Lambda^0_+)\Big\{\lebesgue(\tilde{\Lambda}_+ \cap \Lambda^0_+)\wedge (1-\lebesgue(\tilde{\Lambda}_+ \cup \Lambda^0_+))\Big\} \geq \delta^d \kappa/2.\]
If one the other hand $\lebesgue(\Lambda^0_+ \vartriangle \tilde{\Lambda}_+) > (1-\lebesgue(\Lambda^0_+))/2 \geq \kappa/2$, then we obtain directly
\begin{align*}
\lebesgue(\tilde{\Lambda}_+ \vartriangle \Lambda^0_+)\Big\{\lebesgue(\tilde{\Lambda}_+ \cap \Lambda^0_+)\wedge (1-\lebesgue(\tilde{\Lambda}_+ \cup \Lambda^0_+))\Big\} \geq \delta^d \kappa/2
\end{align*}
Combining all cases, the assertion follows.
\end{proof}

\section*{Acknowledgements}
The authors gratefully acknowledge financial support of Carlsberg Foundation Young Researcher Fellowship grant CF20-0640 ``Exploring the potential of nonparametric modelling of complex systems via SPDEs''.

\printbibliography

\end{document}